\documentclass[sn-mathphys,Numbered]{sn-jnl}

\usepackage{graphicx}%
\usepackage{multirow}%
\usepackage{amsmath,amssymb,amsfonts}%
\usepackage{amsthm}%
\usepackage{mathrsfs}%
\usepackage[title]{appendix}%
\usepackage{xcolor}%
\usepackage{textcomp}%
\usepackage{manyfoot}%
\usepackage{booktabs}%
\usepackage{algorithm}%
\usepackage{algorithmicx}%
\usepackage{algpseudocode}%
\usepackage{listings}%

\usepackage{cleveref}
\usepackage{booktabs}
\usepackage[framemethod=TikZ]{mdframed}
\usepackage{xpatch}

\usepackage{tikz}
\usepackage{pgfplots}
\usetikzlibrary{positioning}
\usetikzlibrary{plotmarks}
\newlength\figureheight
\newlength\figurewidth

\theoremstyle{thmstyleone}%
\newtheorem{thm}{Theorem}
\newtheorem{lemma}{Lemma}

\theoremstyle{thmstyletwo}%

\theoremstyle{thmstylethree}%

\newcommand{\matrx}[1]{\mathbf{#1}}
\renewcommand{\vec}[1]{\mathbf{#1}}
\newcommand{\vecT}[1]{\mathbf{#1}^*}
\newcommand{\R}{\mathbb{R}}
\newcommand{\K}{\mathcal{K}}
\newcommand{\N}{\mathcal{N}}
\newcommand{\T}{\mathcal{T}}
\newcommand{\E}{\mathcal{E}}

\newcommand{\sign}{\text{sign}}

\newcommand{\mM}{\matrx{M}}

\newcommand{\mMs}{\matrx{M}^\mathrm{S}}

\newcommand{\Anorm}[1]{\| #1 \|^2_{\matrx{A}_0}}

\newcommand{\mMKsj}[1]{\matrx{M}_{#1,K}^\mathrm{S}}
\newcommand{\mMsKj}[1]{\matrx{M}_{#1,K}^\mathrm{S}}

\newcommand{\inter}{\mathrm{int}}
\newcommand{\RHS}{\mathrm{RHS}}
\newcommand{\JUMP}{\mathrm{JUMP}}
\newcommand{\cls}{\mathrm{cls}}
\newcommand{\HS}{\mathrm{HS}}
\newcommand{\TI}{\mathrm{TI}}
\newcommand{\BH}{\mathrm{BH}}
\newcommand{\TR}{\mathrm{TR}}
\newcommand{\INV}{\mathrm{INV}}

\newcommand{\interior}[1]{%
  {\kern0pt#1}^{\mathrm{o}}%
}
\newcommand{\meas}[1]{|#1|}

\newcommand{\Covrlp}{C_\mathrm{ovrlp}}
\newcommand{\osc}{\mathrm{osc}}

\usepackage[normalem]{ulem} 

\newcommand{\notred}[1]{}

\newcommand{\notblue}[1]{}
\definecolor{gray}{rgb}{0.5, 0.5, 0.5}

\makeatletter
\xpatchcmd{\endmdframed}
  {\aftergroup\endmdf@trivlist\color@endgroup}
  {\endmdf@trivlist\color@endgroup\@doendpe}
  {}{}
\makeatother

\mdfdefinestyle{total}{linecolor=black,linewidth=0pt,skipabove=0.5cm, skipbelow=0.5cm}

\newtheorem{estim}{Estimate on total error}
 \newenvironment{est}
   {\begin{mdframed}[style=total]\begin{estim}}
   {\end{estim}\end{mdframed}}

\mdfdefinestyle{alg}{linecolor=gray,linewidth=0pt,skipabove=0.3cm,skipbelow=0.25cm}

\newtheorem{estimalg}{Estimate on algebraic error}
 \newenvironment{estalg}
   {\begin{mdframed}[style=alg]\begin{estimalg}}
   {\end{estimalg}\end{mdframed}}

\begin{document}

\title[A posteriori error estimates based on multilevel decompositions with large problems on the coarsest level]{A posteriori error estimates based on multilevel decompositions with large problems on the coarsest level}


\author*[1]{\fnm{Petr} \sur{Vacek}}\email{vacek@karlin.mff.cuni.cz}

\author[2]{\fnm{Jan} \sur{Papež}}\email{papez@math.cas.cz}

\author[1]{\fnm{Zdeněk} \sur{Strakoš}}\email{strakos@karlin.mff.cuni.cz}

\affil*[1]{\orgdiv{Department of Numerical Mathematics}, \orgname{Faculty of Mathematics and Physics, Charles University}, \orgaddress{\street{Sokolovská 83}, \city{Prague}, \postcode{186 75}, \country{Czech Republic}}}

\affil[2]{\orgdiv{Institute of Mathematics}, \orgname{Czech Academy of Sciences}, \orgaddress{\street{Žitná 25}, \city{Prague}, \postcode{115 67}, \country{Czech Republic}}}


\abstract{
Multilevel methods represent a powerful approach in numerical solution of partial differential equations. The multilevel structure can also be used to construct estimates for total and algebraic errors of computed approximations. This paper deals with residual-based error estimates that are based on properties of quasi-interpolation operators, stable-splittings, or frames. We focus on the settings where the system matrix on the coarsest level is still large and the associated terms in the estimates can only be approximated. 
We show that the way in which the error term associated with the coarsest level is approximated is substantial. It can significantly affect both the efficiency (accuracy) of the overall error estimates and their robustness with respect to the size of the coarsest problem.
The newly proposed approximation of the coarsest-level term is based on using the conjugate gradient method with an appropriate stopping criterion. We prove that the resulting estimates are efficient and robust with respect to the size of the coarsest-level problem. Numerical experiments illustrate the theoretical findings.
}



\maketitle

\section{Introduction}\label{sec:intro}

Multilevel methods \cite{Brandt2011,Hackbusch2016,Briggs2000,Trottenberg2001} are frequently used for solving systems of linear equations obtained from the discretization of partial differential equations (PDEs). They are applied either as standalone iterative solvers or as preconditioners.
In \emph{geometric} multigrid methods the hierarchy of systems is obtained by discretizations of an infinite dimensional problem on a sequence of nested meshes. In \emph{algebraic} multigrid methods the coarse systems are constructed using algebraic properties of the matrix.
Each multigrid cycle contains smoothing on fine levels, prolongation, and solving a system of linear equations on the coarsest level. 
Smoothing is typically done by a few iterations of a stationary iterative method. 
If the size permits, it is typical to solve the coarsest-level problem using a direct method based on LU or Cholesky decomposition. Although this does not provide a computed result with a zero error, many theoretical results on multigrid methods are proved under the assumption that the coarsest-level problem is solved exactly; see, e.g., \cite{Xu1992,Yserentant1993}.

Multilevel methods can in practice also use hierarchies where the problem on the coarsest level is large and can only be solved approximately to a properly chosen accuracy, e.g., by Krylov subspace methods, or direct methods can be used with low-rank matrix approximations. This arises for problems on complicated domains or for large-scale problems solved on modern parallel computers; see, e.g., \cite{Buttari2022}. Effects of approximate coarsest-level solves on convergence of multigrid method were analysed, e.g., in \cite{Notay2007,Xu2022, VacekCarsonSoodhalter23}.

The multilevel structure can also be used to construct estimates of total and algebraic errors; see, e.g., \cite{Becker1995,Ruede1993,Harbrecht2016,Huber2019,Papez2017,MirPapVoh21}. 
The estimates of \cite{Becker1995,Ruede1993,Harbrecht2016,Huber2019,Papez2017,MirPapVoh21} are, however, not suited for multilevel hierarchies with large coarsest-level problems, which are being used for complicated domains and/or in parallel implementations.
They either assume that the coarsest-level problem is solved exactly \cite{Becker1995,Papez2017,MirPapVoh21}, or they require computation of the term $\vecT{r}_0\matrx{A}^{-1}_0\vec{r}_0$ associated with the coarsest level, where~$\matrx{A}_0$ is the coarsest-level system matrix and $\vec{r}_0$ a projection of a finest-level residual to the coarsest level, \cite{Ruede1993,Harbrecht2016,Huber2019}.
The term $\vecT{r}_0\matrx{A}^{-1}_0\vec{r}_0$ can be approximated, e.g., using the conjugate gradient method (CG) as in \cite{Huber2019}, or by replacing the system matrix with a diagonal matrix as in \cite{Harbrecht2016}.
Then proving efficiency and robustness of estimates becomes an important challenge.

In this text, we discuss properties of the error estimates in multilevel settings where the system matrix on the coarsest level is large and the associated terms are only approximated.
We consider several a~posteriori estimates on total and algebraic errors based on decomposing the error into a sequence of finite element subspaces and using either approximation properties of quasi-interpolation operators \cite{Becker1995}, stable splittings \cite{Ruede1993,Huber2019}, or so-called frames \cite{Harbrecht2016}.
The main contribution of this paper is a new procedure for approximating the term associated with the coarsest level that is based on using the conjugate gradient method with an appropriate stopping criterion. We prove that the resulting estimates are \emph{efficient} \emph{and robust} with respect to the size of the coarsest-level problem.

The text is organized as follows. First, we present a model problem, its discretization, and the notation used in the text. Derivations of error estimates for total and algebraic errors are presented in \Cref{sec:estimates}. In \Cref{sec:effectivity}, we comment on the efficiency of the bounds. 
Main results are presented in \Cref{sec:evaluation} where we describe how to replace the (uncomputable) terms in the estimates by a computable approximation and present an adaptive procedure for approximating the coarsest-level term $\vecT{r}_0\matrx{A}^{-1}_0\vec{r}_0$.
Numerical illustrations are given in \Cref{sec:numexp} and conclusions in \Cref{sec:conclusions}.
Not to interrupt the presentation, we present detailed theoretical results, which are used in the derivation of the estimates, in Appendices.
Appendix~\ref{sec:PDElemmas} recalls some standard results from PDE and finite element method (FEM) analyses. Appendix~\ref{sec:quasi-interpolation-operator} presents properties of the quasi-interpolation operator, and Appendix~\ref{sec:splitting} recalls results on stable-splittings and frames. 
This enables an easy comparison of different results that are presented separately in literature.


\section{Model problem, setting, and notation}
\label{sec:notation}

The estimates will be studied for a standard model problem, a prototype for elliptic equations, the Poisson's problem with homogeneous Dirichlet boundary conditions.
Let $\Omega\subset \R^d$, $d = 2,3$, be an open bounded polytope with a Lipschitz-continuous boundary. Given $f \in L^2(\Omega)$, the weak form reads: find \mbox{$u\in H^1_0(\Omega)$} such that 
\begin{equation} \label{eq:model_problem_weak}
    \int_\Omega \nabla u \cdot \nabla v = \int_\Omega fv \qquad \forall v \in  H^1_0(\Omega).
\end{equation}
In this section, we introduce notation for meshes and finite element spaces, and the multilevel framework. Further, we present the Galerkin finite element discretization of the model problem on a particular level, define its approximate solution, the error, and (scaled) residuals associated with individual levels of the multilevel hierarchy.

Similarly to a standard literature, we introduce some simplifying assumptions, e.g., on the model problem or mesh hierarchies. This is done in order to reduce the complexity of proofs (that are already quite technical) and to allow us to refer to particular results in the literature. We use a standard notation for Lebesgue and Sobolev (Hilbert) spaces, norms, and seminorms; see, e.g., \cite{BookBrezis2011}.


\subsection{Notation for a single level}
\label{sec:triangulation}

Throughout the paper, we consider simplicial meshes of $\Omega$, matching in the sense that for two distinct elements of a mesh~$\T$ (triangles in 2D, or tetrahedra in 3D), their intersection is either an empty set or a common node (vertex), edge, or face.
By $\E_\T$ and $\N_\T$ we denote the set of $(d-1)$-dimensional faces and set of nodes in the mesh~$\T$, respectively.
By $\E_{\T,\inter} $ we denote the set of all faces that are not on the boundary $\partial\Omega$. 
By $\K_\T \subset \N_\T$ we denote the set of all nodes in the mesh $\T$, which are not on the boundary, i.e., \emph{free nodes}.
For any element (simplex) $K \in \T$, $\E_K \subset \E_\T$ denotes the set of faces of the element $K$, $\N_K \subset \N_\T$ denotes the set of nodes of the element $K$,  $\E_{K,\inter} = \E_K \cap \E_{\T,\inter}$, and $\K_K = \N_K \cap \K_\T $.
We use hash to denote the cardinality of a set, for example $\#\K_\T$ denotes the number of free nodes in the mesh $\T$. For the ease of presentation, we will assume that the nodes in $\N_\T$ are ordered such that nodes $1, \ldots,  \#\K_\T$ belong to $\K_\T$, i.e., we first have the free nodes and then the nodes on the boundary.

By $h_K$ we denote the diameter of $K \in \T$ 
and define a mesh-size $h_{\T}\in L^{\infty}(\Omega)$ as 
\begin{align*}
h_{\T}(x) &= h_K, \quad x \in K, \quad \forall K\in \T.
\end{align*}
Similarly $h_\omega$ denotes the diameter of a domain $\omega$. We in particular use $h_\Omega$, the diameter of the domain $\Omega$. By $|\omega|$ we denote the Lebesgue measure of a domain $\omega$.

For any element $K \in \T$, $\omega_K$ denotes the patch of elements that share at least one common vertex with $K$, i.e.,
\begin{equation*}
\omega_K = \bigcup_{K' \in \T;  K'\cap K \neq \emptyset} K'.
\end{equation*}
By $\rho_K$ we denote the diameter of the largest ball inscribed in the element $K$.

For every node $z \in \N_\T$, let $\phi_z$ be the continuous piecewise linear function (\emph{hat function}) that has a value one at node $z$ and vanishes at all the other nodes in $\N_\T$.
Let $S_{\T}$ denote the space of continuous, piecewise linear functions, 
\[
    S_{\T} = \{ v \in H^1(\Omega), v|_K \in \mathbb{P}^1(K), \ \forall K\in \T\} = \mbox{span}\{ \phi_z,\ z \in \N_\T \}
\]
and $V_{\T} \subset S_{\T}$ the subspace of functions vanishing on the boundary~$\partial\Omega$, 
\[
    V_{\T} = \{ v \in H^1_0(\Omega), v|_K \in \mathbb{P}^1(K), \ \forall K\in \T\} = \mbox{span}\{ \phi_z,\ z \in \K_\T \}.
\] 
We write the basis of $V_{\T}$ as $\Phi_{\T}=(\phi_1,\ldots, \phi_{\# \K_\T})$.

One of the key properties of a mesh that affects the size of the constants in the estimates derived below in this text is the so-called \emph{shape regularity} of the mesh. This can be quantified by the shape-regularity constant, i.e., the smallest $\gamma_{\T}>0$ satisfying
\begin{equation}\label{eq:shaperegularity}
\frac{h_K}{\rho_K}\leq \gamma_{\T}, \quad \forall K \in \T;
\end{equation}
see, e.g., \cite[p.~484]{Scott1990}.


\subsection{Multilevel framework}
\label{sec:multilevel}

As the title of the paper suggests, we will work with a sequence of levels $j = 0, 1, \ldots, J$. For some parts of the theory, we will consider also infinite sequences of levels $j = 0,1, \ldots, J, \ldots$. To simplify the previously introduced notation, we will replace in the subscripts $\T_j$ by $j$ to denote objects associated with the mesh $\T_j$ on the $j$th level.  

Let $\T_0$ be an initial mesh of $\Omega$. We consider a sequence of meshes $\T_1,\T_2,\ldots$ obtained by successive uniform dyadic refinements of $\T_0$, i.e., each element is refined into $2^d$ elements (congruent triangles in 2D, for a proper nondegenerating 3D mesh refinement; see, e.g.,~\cite{Zha95}). We recall that {$S_j$ and $V_j$, $j=0,1,\ldots$}, are the finite element spaces of continuous piecewise linear functions on $\T_j$, respectively spaces of continuous piecewise linear functions on $\T_j$ that vanish on the boundary $\partial \Omega$. These spaces are nested, i.e.,
\begin{align*}
& S_0 \subset S_1 \subset \cdots \subset H^1(\Omega),  
& V_0 \subset V_1 \subset \cdots \subset H^1_0(\Omega).    
\end{align*}
On each level~$j$, we consider a quasi-interpolation operator
\[
    I_{V_j} : L^1(\Omega) \to V_j
\]
with the definition and properties described in detail in Appendix~\ref{sec:quasi-interpolation-operator}.

Due to the uniform refinement, the mesh sizes $h_j$ of $\T_j$, $j\geq 0$, satisfy $h_j = 2^{-j}h_0$.
Moreover, the uniform refinement assures that the shape-regularity constants $\gamma_j$ of the meshes are the same on all levels in 2D, i.e., $\gamma_0 = \gamma_j$, $j \in \mathbb{N}$, and that in 3D there exists a constant $C_{\mathrm{3D}}>0$ such that $\gamma_j \leq C_{\mathrm{3D}} \gamma_0$,  $j \in \mathbb{N}$; see \cite{Zha95}.


\subsection{Discretization, approximate solution, and residuals}
\label{sec:residuals}

Discretizing the model problem \eqref{eq:model_problem_weak} on the subspace $V_J$,  for some $J\geq0$, using the Galerkin method reads as: find $u_J \in  V_J$ such that
\begin{equation}
\label{eq:dicretized_problem}
\int_\Omega \nabla u_J \cdot \nabla w_J = \int_\Omega fw_J, \quad \forall w_J \in V_J.
\end{equation}

Let $v_J \in V_J$ be a (computed) approximation of the discrete solution $u_J$. Our goal is to bound the energy norm of the total error $e= u-v_J$ using computable quantities involving $v_J$ and $f$. 
The squared energy norm of the error $\| \nabla e \|^2$ can be expressed as
\begin{equation*}
    \| \nabla e \|^2 = \| \nabla (u-v_J) \|^2 = \int_{\Omega}\nabla (u-v_J) \cdot \nabla (u-v_J)  = \int_{\Omega} f (u-v_J) - \nabla v_J \cdot \nabla (u-v_J).
\end{equation*}
Denote by $\left( H^1_0(\Omega)\right)^{\#}$ the dual space to $H^1_0(\Omega)$ and define the residual $r\in \left( H^1_0(\Omega)\right)^{\#}$ as 
\begin{equation}
\label{eq:funcresidual}
    \langle r,w \rangle = \int_{\Omega} f w  - \nabla v_J \cdot \nabla w, \quad \forall w \in H^1_0(\Omega). 
\end{equation}
Then \eqref{eq:funcresidual}
yields the so-called residual equation
\begin{equation}\label{eq:residual_equations}
    \| \nabla e \|^2 = \langle r,e \rangle,
\end{equation}
which is the key formula for the development of error bounds presented below.
Moreover, it can be shown (see, e.g.,  \cite[Section~1.4.1]{Verfurth2013}) that
\begin{equation*}
\| \nabla e \| = \| r\|_{\left(  H^1_0(\Omega )\right )^{\#}}.
\end{equation*}

In order to derive computable estimates we consider Riesz representations of the infinite-dimensional residual $r$ in the finite-dimensional spaces $V_j$, $j= 0,1,\ldots$.
In particular, let $r_j \in V_j$, $j=1,\ldots$, be the Riesz representation of $r$ in the space $V_j$ with the scaled $L^2$-inner product, i.e.,
\begin{equation}
\label{eq:r_j}
\langle r,w_j \rangle = \int_{\Omega} h^{-2}_{j} r_j w_j, \quad \forall w_j \in V_j, 
\end{equation}
and let $r_0 \in V_0$ be the Riesz representation of the residual $r$ in the space $V_0$ with the $H^1_0$-inner product, i.e.,
\begin{equation}
\label{eq:r_0}
\langle r,w_0 \rangle = \int_{\Omega} \nabla r_0 \cdot \nabla w_0,\quad  \forall  w_0 \in V_0.
\end{equation}
These definitions are used in  \cite[Section~2.6]{Ruede1993} where $r_j$ are called \emph{scaled residuals}. 
In \cite[Section~5]{Becker1995} the authors use Riesz representations of $r$ in the spaces~$V_j$, $j=1,\ldots,J$, with the classical $L^2$-inner products and call them discrete residuals. The different definition we use results in a slightly different form of the estimates below in comparison to \cite[Section~5]{Becker1995}.

\section{Residual-based error estimates}
\label{sec:estimates}

In this section we recall several published error estimates with their derivation.
We first recall the standard residual-based error estimator for the discretization error in a single-level setting assuming exact algebraic computations or to steer an adaptive mesh refinement.

Consider the model problem~\eqref{eq:model_problem_weak} discretized on a level $J\geq0$ of a multilevel hierarchy as in \Cref{sec:multilevel}. 
The classical residual-based estimator (see, e.g., \cite[Section~3]{Ainsworth1997}, \cite[Section~1.4]{Verfurth2013}) is for a (computed) approximation $v_J \in V_J$ defined as
\begin{align*}
\eta^2_{J} &= \left( \eta^{\RHS}_{J} \right)^2 + \left( \eta^{\JUMP}_{J} \right)^2 + \left( \osc_{J} \right)^2, \\
\left ( \eta^{\RHS}_{J} \right)^2 &=  \sum_{K\in \T_J} h^2_K \| f_{K} \|^2_K, 
\\
\left (\eta^{\JUMP}_{J}\right)^2  &= \frac{1}{2}\sum_{K\in \T_J} h_K \!\!\! \sum_{E \in \E_{K,\inter}} \!\!\! \|  \left[ \nabla v_J \right] \|^2_{E}, \\
\left( \osc_{J} \right)^2 & = \sum_{K\in \T_J} h^2_K \| f - f_{K} \|^2_K,
\end{align*}
where $\left[ \cdot \right]$ denotes the jump of a piecewise constant function over the $(d-1)$-dimensional faces (faces in 3D and edges in 2D) and $f_K$ is the mean value of~$f$ on~$K$. Other choices of~$f_K$ are also possible; see, e.g., \cite{Harbrecht2016}.

The following result (see, e.g., \cite[Lemma~3]{Becker1995}, \cite[Section~4]{Stevenson2007}, or \cite[Section~1.4]{Verfurth2013}) will be useful below. There exists a constant $C_{\cls}>0$ depending only on the dimension $d$ and the shape-regularity parameter $\gamma_0$ such that 
\begin{equation}\label{eq:residual_finest_level_quasi_interpol}
\langle r, w -I_{V_J} w \rangle \leq C_{\cls}  \eta_{J} \|\nabla w\|, \quad \forall w \in H^1_0(\Omega).
\end{equation}
Note that if $v_J$ is equal to the Galerkin solution~$u_J$, the associated residual \mbox{$r=r(u_J)$} satisfies the Galerkin orthogonality on the finest level, i.e.,
\begin{equation}\label{eq:galerkin_orthogonality_finest}
\langle r, w_J \rangle = 0, \quad \forall w_J \in V_J.
\end{equation}
Then
\begin{equation*}
\|\nabla ( u - u_J) \|^2    = \langle r, ( u - u_J) -I_{V_J} ( u - u_J) \rangle,
\end{equation*}
and using \eqref{eq:residual_finest_level_quasi_interpol} for $w = u-u_J$ yields the standard bound on the discretization error
\begin{equation*}
\|\nabla ( u - u_J) \|  \leq  C_{\cls}  \eta_{J}(u_J).
\end{equation*}


\subsection{Estimates of Becker, Johnson \& Rannacher}
\label{sec:estimBJR}

The following derivation is motivated by \cite{Becker1995} and uses decomposition of the error via quasi-interpolation operators.  
Considering the residual equation \eqref{eq:residual_equations} and writing the error $e = u - v_J$ as
\begin{equation}\label{eq:BJR_decomposition}
e = e - I_{V_J} e + \sum^{J}_{j=1} \left(  I_{V_j} e - I_{V_{j-1}} e  \right) + I_{V_0} e,
\end{equation}
yields
\begin{equation}\label{eq:BJR_decomposition_norms}
\| \nabla e \|^2  = \langle r,e \rangle
= \langle r, e - I_{V_J} e \rangle + \sum^{J}_{j=1}  \langle r, I_{V_j} e  -I_{V_j-1} e \rangle + \langle r, I_{V_0} e \rangle.
\end{equation}
The first term on the right-hand side of \eqref{eq:BJR_decomposition_norms} can be bounded using \eqref{eq:residual_finest_level_quasi_interpol} as
\begin{equation}\label{eq:finest_level_residual_bound}
\langle r, e - I_{V_J} e \rangle \leq C_{\cls} \eta_{J} \| \nabla e \|.
\end{equation}

\noindent
The second and the third term on the right-hand side of \eqref{eq:BJR_decomposition_norms} can be rewritten using the scaled residuals~\eqref{eq:r_j}, \eqref{eq:r_0} and subsequently bounded as
\begin{align}
\begin{split}\label{eq:BJR_decomposition_scaled_res}
&\sum^J_{j=1}\langle r, I_{V_j} e  - I_{V_{j-1}} e \rangle + \langle r, I_{V_0} e \rangle =  \sum^J_{j=1} \int_{\Omega}h^{-2}_{j} r_{j} (I_{V_{j}} e  - I_{V_{j-1}} e) + \int_{\Omega} \nabla r_0\cdot \nabla I_{V_{0}} e \\
& \qquad \qquad\qquad \qquad\qquad\leq   \sum^J_{j=1} \| h^{-1}_{j} r_j\| \cdot \| h^{-1}_{j} ( I_{V_{j}} e -I_{V_{j-1}} e ) \| + \| \nabla r_0\| \cdot \| \nabla I_{V_{0}} e \|.
\end{split}
\end{align}
Further, using the bound on the difference of the quasi-interpolants on two consecutive levels (Appendix~\ref{sec:quasi-interpolation-operator}, \Cref{lemma:difference-quasi-interpolation}) and the stability of the quasi-interpolation operator on the coarsest level in the $H^1_0(\Omega)$-norm (Appendix~\ref{sec:quasi-interpolation-operator}, \Cref{thm:quasi-interpolation-to_V-global-estimates}, inequality \eqref{eq:quasi-interpolation-to_V-global_grad}), we get
\begin{align}
\begin{split}\label{eq:BJR_approach}
& \sum^J_{j=1} \| h^{-1}_{j} r_j\| \cdot \| h^{-1}_{j} ( I_{V_{j}} e -I_{V_{j-1}} e ) \| + \| \nabla r_0\| \cdot \| \nabla I_{V_{0}} e \|\\
& \qquad \leq
 C_{I,\mathrm{2lvl}} \left( \sum^{J}_{j=1} \| h^{-1}_{j} r_j\| \right) \| \nabla e\| + \| \nabla r_0 \| \cdot C_{I_{V_0},4} \cdot \| \nabla e\|.
\end{split}
\end{align}
Combining \eqref{eq:BJR_decomposition_norms}--\eqref{eq:BJR_approach}  yields
\begin{est} 
\begin{equation}
 \label{est:BJR_inspired_classic}
    \| \nabla e \|  \leq C_{\cls} \eta_{J}  + C_{I,\mathrm{2lvl}}  \sum^J_{j=1} \| h^{-1}_{j} r_{j} \| +  C_{I_{V_0},4}  \| \nabla r_0 \|.
\end{equation}
\end{est}
In \cite{Becker1995} the authors assume that the approximation $v_J$ is computed by a multigrid scheme without post-smoothing and with the exact solution of the problem on the coarsest level. This yields the Galerkin orthogonality on the coarsest level, i.e.,
\begin{equation}
 \left\langle r,w_0 \right\rangle = 0,\quad \forall w_0 \in V_0.
\end{equation}
As a consequence, their estimate on the energy norm of the error (see \cite[Theorem~1]{Becker1995}) does not contain the term corresponding to the coarsest level. Another difference between \eqref{est:BJR_inspired_classic} and the estimate in \cite[Theorem~1]{Becker1995} is due to the difference in the definitions of the scaled/discrete residuals described in \Cref{sec:residuals}.

\medskip
Instead of using the bound on the difference of the quasi-interpolants on two consecutive levels (Appendix~\ref{sec:quasi-interpolation-operator}, \Cref{lemma:difference-quasi-interpolation}), and the stability of the quasi-interpolation operator on the coarsest level (Appendix~\ref{sec:quasi-interpolation-operator}, \Cref{thm:quasi-interpolation-to_V-global-estimates}, inequality \eqref{eq:quasi-interpolation-to_V-global_grad}), we can use the stability of the decomposition of the space $H^1_0(\Omega)$ via the quasi-interpolation operators~$I_{V_j}$ (Appendix~\ref{sec:quasi-interpolation-operator}, \Cref{thm:decomposition-using-quasi-interpol-to-V}). In particular,
\begin{multline*}
\sum^{J}_{j=1} \| h^{-1}_{j} r_j\| \cdot \| h^{-1}_{j} ( I_{V_{j}} e - I_{V_{j-1}} e ) \| + \| \nabla r_0\| \cdot \| \nabla I_{V_{0}} e \|\\ 
\leq \left(  \sum^{J}_{j=1} \| h^{-1}_{j} r_j\|^2 + \| \nabla r_0\|^2 \right)^{\frac{1}{2}} \left(  \sum^{J}_{j=1} \| h^{-1}_{j} ( I_{V_{j}} e - I_{V_{j-1}} e ) \|^2  + \| \nabla I_{V_{0}} e \|^2  \right)^{\frac{1}{2}} \\
\leq \left(  \sum^{J}_{j=1} \| h^{-1}_{j} r_j\|^2 + \| \nabla r_0\|^2 \right)^{\frac{1}{2}} {C_{S,I_V}}^{\frac{1}{2}} \| \nabla e \|. 
\end{multline*}
Combining this inequality with \eqref{eq:BJR_decomposition_norms}--\eqref{eq:BJR_decomposition_scaled_res} and using $\sqrt{a} + \sqrt{b} \leq \sqrt{2}\sqrt{a+b}$ leads~to

\begin{est}
\begin{equation}
\label{est:BJR_improved}
\| \nabla e \| \leq  \sqrt{2}\left(  C^2_{\cls}  \eta^2_{J} +      C_{S,I_V} \left (\sum^J_{j=1} \| h^{-1}_{j} r_{j} \|^2 +   \| \nabla r_0\|^2 \right)\right)^{\frac{1}{2}}.
\end{equation}
\end{est}
As we will see in Section \ref{sec:effectivity}, this estimate is efficient with an efficiency constant independent of the number of levels in the hierarchy, i.e., independent of~$J$.

Observing that 
\begin{align*}
\| \nabla (u_J-v_J) \|^2 &= \int_{\Omega} f (u_J-v_J) - \int_{\Omega} \nabla v_J \cdot \nabla (u_J-v_J) = \langle r,  u_J-v_J \rangle\\
&= \sum^{J}_{j=1} \langle r, I_{V_{j}} (u_J-v_J) - I_{V_{j-1}} (u_J-v_J)  \rangle + \langle r, I_{V_{0}} (u_J-v_J) \rangle,
\end{align*}
analogous steps can be applied to show that the following ``algebraic parts''  of the presented estimates provide upper bounds on the algebraic error,
\begin{estalg}
\begin{equation}\label{eq:BJR_alg}
\| \nabla (u_J-v_J)\| \leq C_{I,\mathrm{2lvl}}  \sum^J_{j=1} \| h^{-1}_{j} r_{j} \| +  C_{I_0,3}  \|\nabla r_0 \|,
\end{equation}
\end{estalg}
\begin{estalg}
\begin{equation}\label{eq:BJR_improved_alg}
\| \nabla (u_J-v_J)\| \leq    {C_{S,I_V}}^{\frac{1}{2}} \left(   \sum^J_{j=1} \| h^{-1}_{j} r_{j} \|^2 +  \| \nabla r_0\|^2\right)^{\frac{1}{2}}.
\end{equation}
\end{estalg}


\subsection{Estimates of R\"{u}de \& Huber}
\label{sec:estimRH}

The following derivation is motivated by \cite[Section~2.6]{Ruede1993} and \cite[Sections~4.1--4.3]{Huber2019}. 
Considering the residual equation \eqref{eq:residual_equations} and 
decomposing the error using the quasi-interpolation operator on the finest level $I_{V_{J}}$  yields
\begin{equation}\label{eq:RH_decomposition}
\| \nabla e \|^2 = \langle r ,e- I_{V_{J}} e \rangle + \langle r, I_{V_{J}} e\rangle.
\end{equation}
The first term can be bounded as in \eqref{eq:finest_level_residual_bound}. Rewriting the second term using the exact solution of the discrete problem $u_J$  gives
\begin{align*}
\langle r, I_{V_{J}} e \rangle &= \int_{\Omega} \nabla (u-v_J) \nabla I_{V_{J}} e \\
&= \int_{\Omega} \nabla (u-u_J) \nabla I_{V_{J}} e + \int_{\Omega} \nabla (u_J-v_J) \nabla I_{V_{J}} e.
\end{align*}
The Galerkin orthogonality on the finest level yields that $\int_{\Omega} \nabla (u-u_J) \nabla I_{V_{J}} e$ vanishes and thus
\begin{equation}\label{eq:RH_res_bound_CS}
\langle r, I_{V_{J}} e \rangle =\int_{\Omega} \nabla (u_J-v_J) \nabla I_{V_{J}} e \leq \|\nabla (u_J-v_J) \| \, \|\nabla I_{V_{J}} e \|.
\end{equation} 
After bounding the term $\| \nabla I_{V_{J}} e \|$  using the stability property of the quasi-interpolation operator (Appendix~\ref{sec:quasi-interpolation-operator}, \Cref{thm:quasi-interpolation-to_V-global-estimates}, inequality \eqref{eq:quasi-interpolation-to_V-global_grad}) as
\begin{equation}\label{eq:RH_quasi_stab}
\| \nabla I_{V_{J}} e \| \leq C_{I_{V_{J}},4} \| \nabla e \|,
\end{equation} 
it remains to bound the energy norm of the algebraic error $\|\nabla (u_J-v_J) \|$. This can be done using stable splitting of piecewise linear function space, see 
Appendix~\ref{sec:splitting}, \Cref{thm:stable_splitting_of_V_J} or \cite[Theorem~2.6.2]{Ruede1993}.
Consider an arbitrary decomposition of the algebraic error $u_J-v_J$ into the subspaces  $V_j$, i.e.,
\begin{equation}\label{eq:rude_huber_alg_decomposition}
 u_J-v_J = \sum^{J}_{j=0} e_j, \quad e_j\in V_j, \quad j=0,1,\ldots,J. 
\end{equation}
Then
\begin{align*}
\|\nabla(u_J-v_J) \|^2 &= \langle r,u_J-v_J \rangle  = \sum^{J}_{j=0} \langle r ,e_j \rangle \\
&\leq  \|\nabla  r_0 \| \cdot  \|\nabla e_0\| + \sum^{J}_{j=1} \| h^{-1}_{j} r_j \| \cdot  \|h^{-1}_{j} e_j\|  \\
& \leq \left( \| \nabla r_0 \|^2 +  \sum^{J}_{j=1} \| h^{-1}_{j} r_j \|^2    \right)^{\frac{1}{2}}
\cdot
\left(\|\nabla e_0\|^2 + \sum^{J}_{j=1}   \| h^{-1}_{j} e_j\|^2 
 \right)^{\frac{1}{2}}.
\end{align*}
Taking the infimum over all possible decompositions  \eqref{eq:rude_huber_alg_decomposition} and using Appendix \ref{sec:splitting}, \Cref{thm:stable_splitting_of_V_J} yields
\begin{estalg}
\begin{equation}
\label{estalg:Ruede93}
\|\nabla (u_J-v_J) \| \leq C_{S}^{\frac{1}{2}} \left(  \sum^{J}_{j=1} \| h^{-1}_{j} r_j \|^2  + \|\nabla r_0 \|^2 \right)^{\frac{1}{2}}.
\end{equation}
\end{estalg}

Combining \eqref{eq:RH_decomposition}--\eqref{eq:RH_quasi_stab}, the estimate \eqref{estalg:Ruede93} on the algebraic error, and using the inequality $\sqrt{a} + \sqrt{b} \leq \sqrt{2}\sqrt{a+b}$, we have 
\begin{est}
\begin{equation}
\label{est:RudeHuber}
\|\nabla  e\| \leq \sqrt{2} \left( C^2_{\cls} \eta^2_{J}  + C^2_{I_{V_J},4} C_{S}  \left( \sum^{J}_{j=0}  \| h^{-1}_j r_j \|^2 + \| \nabla r_0 \|^2 \right)\right)^{\frac{1}{2}}.
\end{equation}
\end{est}

\subsection{Estimates of Harbrecht \& Schneider}
\label{sec:estimHS}

In this section we present a derivation motivated by \cite{Harbrecht2016},
which is based on the fact that the basis functions provide a frame in $\left( H^1_0(\Omega) \right)^{\#}$; see Appendix~\ref{sec:splitting}, \Cref{thm:frame_infinite}. Recall that $\left(H^1_0(\Omega) \right)^{\#}$ is the dual space to $H^1_0(\Omega)$. Using the upper bound for the residual yields
\begin{equation}\label{eq:HS_frame}
\| \nabla e \| = \| r\|_{\left( H^1_0(\Omega) \right)^{\#}} \leq   C_{S}^{\frac{1}{2}} \overline{C}_{B}^{\frac{1}{2}}\left(
\| \nabla  r_0 \|^2  +  \sum^{+ \infty}_{j=1} \sum^{\#\K_j}_{i=1}  \frac{\langle r,\phi^{(j)}_i  \rangle^2}{\| \nabla \phi^{(j)}_i\|^2} \right)^{\frac{1}{2}}.
\end{equation}
Following the derivation in \cite[Proof of Theorem~5.1]{Harbrecht2016}, it can be shown that the sum of the terms corresponding to levels $j>J$, i.e.,
\[ \sum^{+ \infty}_{j=J+1}  \sum^{\#\K_j}_{i=1}  \frac{\langle r,\phi^{(j)}_i  \rangle^2}{\| \nabla \phi^{(j)}_i\|^2},
\]
can be bounded by the classic residual based estimator on the $J$th level up to a constant $C_{\HS}>0$ depending only on $d$ and $\gamma_0$, i.e.,
\begin{equation}\label{eq:HS_sum_bound}
\sum^{+ \infty}_{j=J+1} \sum^{\#\K_j}_{i=1}  \frac{\langle r,\phi^{(j)}_i  \rangle^2}{\| \nabla \phi^{(j)}_i\|^2} \leq C_{\HS} \eta^2_J.
\end{equation}

\noindent
Combining  \eqref{eq:HS_frame} and \eqref{eq:HS_sum_bound} yields
\begin{est}
\begin{equation}
\label{est:HS}
\|\nabla  e\| \leq  C_{S}^{\frac{1}{2}} \overline{C}_{B}^{\frac{1}{2}} \left( C_{\HS} \eta^2_{J}  +   \sum^{J}_{j=1} \sum^{\#\K_j}_{i=1}  \frac{\langle r,\phi^{(j)}_i  \rangle^2}{\| \nabla \phi^{(j)}_i\|^2} + \| \nabla r_0 \|^2 \right)^{\frac{1}{2}}.
\end{equation}
\end{est}

Considering the residual~$r$ as a functional on $V_J $, which is possible since $\left( H^1_0(\Omega) \right)^{\#}\subset  V^{\#}_J$, one can show that
\[
\| \nabla (u_J - v_J) \| = \| r \|_{V^{\#}_J}.
\]
From Appendix~\ref{sec:splitting}, \Cref{thm:frames_finite}, it yields that a part of the total error estimator~\eqref{est:HS} is an upper bound on the algebraic error,
\begin{estalg}[]
\begin{equation}
\label{eq:HS_alg}
\| \nabla (u_J - v_J) \|  \leq   C_{S}^{\frac{1}{2}} \overline{C}_{B}^{\frac{1}{2}}  \left( \sum^{J}_{j=1} \sum^{\#\K_j}_{i=1}  \frac{\langle r,\phi^{(j)}_i  \rangle^2}{\| \nabla \phi^{(j)}_i\|^2} + \| \nabla r_0 \|^2 \right)^{\frac{1}{2}}.
\end{equation}
\end{estalg}

\subsection{New estimate derived using stable splitting}
\label{sec:estimss}
The approach from~\cite{Harbrecht2016} can also be modified in the following way. Consider the residual equation \eqref{eq:residual_equations} and an arbitrary decomposition of the error $e=\sum^{+\infty}_{j=0}e_j$, $e_j\in V_j$, $j\in \mathbb{N}_0$. Using the definition of scaled residuals  \eqref{eq:r_j} and \eqref{eq:r_0}, and the Cauchy--Schwarz inequality, we have
\begin{equation}\label{eq:estimss-decompose}
\|\nabla e \|^2 = \langle r,e \rangle  = \sum^{+\infty}_{j=0} \langle r ,{e}_j \rangle 
\leq  \|\nabla  r_0 \| \cdot  \|\nabla {e}_0\| + \sum^{J}_{j=1} \| h^{-1}_{j} r_j \| \cdot  \|h^{-1}_{j} {e}_j\| + 
\sum^{+\infty}_{j=J+1}  \langle r ,{e}_j \rangle.
\end{equation}

\noindent
Consider first the terms $\langle r ,{e}_j \rangle$, for $j>J$.
Using the definition of $r$, Green's theorem on elements $K\in \T_J$, and the definition of jump leads to
\begin{align*}
\langle r,  e_j  \rangle &= \int_{\Omega} f  e_j  - \int_{\Omega} \nabla v_J \cdot \nabla e_j  \\
&= \sum_{K \in \T_J} \left (\int_K  (f + \Delta v_J)  e_j  -  \frac{1}{2} \sum_{E \in \E_{K,\inter}} \int_E \left[   \nabla v_J \right]   e_j \right) .
\end{align*}
Since $v_J$ is a piecewise affine function, the term $\Delta v_J$ vanishes.
Adding and subtracting~$f_J$ yields
\begin{equation*}
\langle r,  e_j  \rangle =  \sum_{K \in \T_J} \left (\int_K  f_J   e_j + \int_K  (f - f_J)  e_j  -  \frac{1}{2} \sum_{E \in \E_{K,\inter}} \int_E \left[   \nabla v_J \right]   e_j \right) .
\end{equation*}
Inserting $h_K h_K^{-1}$ and  $h^{1/2}_K h_K^{-1/2}$, respectively, and using the Cauchy--Schwarz inequality for integrals and subsequently for sums leads to
\begin{align*}
\langle r,  e_j \rangle  &\leq  \left( \sum_{K_j \in \T_j} h^2_{K_j} \|  f_J \|^2_{K_{j}} + \sum_{K_j \in \T_j} h^2_{K_j} \|  f - f_J \|^2_{K_{j}}+  \frac{1}{2}\sum_{K_j \in \T_j} h_{K_j}  \sum_{E \in \E_{K_{j},\inter}}  \|  \left[   \nabla v_J \right] \|^2_E  \right)^{\frac{1}{2}}  
\\
& \qquad \cdot \left( 2 \sum_{K_j \in \T_j} h^{-2}_{K_j} \|  {e}_j \|^2_{K_{j}}  + \sum_{K_{j} \in \T_j}    h^{-1}_{K_j}  \| e_j   \|^2_{\partial K_j}  \right) ^{\frac{1}{2}}.
\end{align*}
Since $ h_j = 2^{J-j} h_J $, we have
\begin{align*}
\sum_{K_j \in \T_j} h^2_{K_j} \| f_J \|^2_{K_{j}}  &= 2^{2(J-j)} \sum_{K_J \in \T_J} h^2_{K_J} \| f_J \|^2_{K_{J}} = 2^{2(J-j)} (\eta^{\RHS}_{J})^2,\\
\sum_{K_j \in \T_j} h^2_{K_j} \| f-f_J \|^2_{K_{j}}  &= 2^{2(J-j)} \sum_{K_J \in \T_J} h^2_{K_J} \| f-f_J \|^2_{K_{J}} = 2^{2(J-j)} (\osc_{J})^2.
\end{align*}
Using that $\nabla v_J$ is constant on elements in $\T_J$ gives
\begin{equation*}
\sum_{K_j \in \T_j}  \sum_{E \in \E_{K_{j},\inter}} \frac{h_{K_j}}{2} \|  \left[   \nabla v \right] \|^2_E = 2^{J-j}\sum_{K_J \in \T_J} \frac{h_{K_J}}{2} \sum_{E \in \E_{K_{J},\inter}}  \|  \left[   \nabla v_J \right] \|^2_E = 2^{J-j}(\eta^{\JUMP}_J)^2.
\end{equation*}
The sum $\sum_{K_{j} \in \T_j}    h^{-1}_{K_j}  \| e_j   \|^2_{\partial K_j}$ can be bounded using Appendix \ref{sec:PDElemmas}, \Cref{lemma:trace+inverse_ineq} as 
\begin{equation*}
\sum_{K_{j} \in \T_j}    h^{-1}_{K_j}  \| e_j   \|^2_{\partial K_j}   \leq C_{\TI}\sum_{K_j \in \T_j} h^{-2}_{K_j} \|  {e}_j \|^2_{K_{j}}.
\end{equation*}
Thus we get
\begin{equation}\label{eq:new_res_action_bound}
    \langle r,  e_j \rangle \leq  \underbrace{\left( 2 + C_{\TI}\right)^{\frac{1}{2}} \left( 2^{2(J-j)} \left( \left ( \eta^{\RHS}_J \right)^2 + (\osc_J)^2\right) 
+  2^{J-j} \left ( \eta^{\JUMP}_J \right)^2 \right)^{\frac{1}{2}}}_{\theta_j} \| h^{-1}_j e_j\|.
\end{equation}

Combining \eqref{eq:estimss-decompose} and \eqref{eq:new_res_action_bound} yields
\begin{align*}
\|\nabla e \|^2 
&\leq  \|\nabla  r_0 \| \cdot  \|\nabla e_0\| + \sum^{J}_{j=1} \| h^{-1}_{j} r_j \| \cdot  \|h^{-1}_{j} e_j\| + 
\sum^{+\infty}_{j=J+1} \theta_{j}  \| h^{-1}_{j} e_j\| \\
& \leq \left( \| \nabla r_0 \|^2 +  \sum^{J}_{j=1} \| h^{-1}_{j} r_j \|^2  + 
\sum^{+\infty}_{j=J+1} \theta^2_{j}\right)^{\frac{1}{2}}  
\\
& \qquad \qquad
\cdot
\left(\|\nabla e_0\|^2 + \sum^{J}_{j=1}   \| h^{-1}_{j} e_j\|^2 + 
 \sum^{+\infty}_{j=J+1}   \|h^{-1}_{j} e_j\|^2 \right)^{\frac{1}{2}}.
\end{align*}
Since the decomposition of $e=\sum^{+\infty}_{j=0}e_j$, $e_j\in V_j$, $j\in \mathbb{N}_0$ is arbitrary, the stability of the splitting (Appendix~\ref{sec:splitting}, \Cref{thm:stable_splitting_to_V}) gives
\begin{equation*}
\|\nabla e \| \leq \left( \| \nabla r_0 \|^2 +  \sum^{J}_{j=1} \| h^{-1}_{j} r_j \|^2  + 
\sum^{+\infty}_{j=J+1} \theta^2_{j} \right)^{\frac{1}{2}}.
\end{equation*}
Using 
\begin{equation*}
\sum^{+\infty}_{j=J+1} 2^{2(J-j)} = \frac{1}{3}\quad \text{and} \quad \sum^{+\infty}_{j=J+1} 2^{J-j} = 1,
\end{equation*}
the infinite sum can be bounded as 
\begin{equation*}
\sum^{+\infty}_{j=J+1} \theta^2_{j} \leq C_{\theta}\eta^2_J,
\end{equation*}
where  $C_{\theta}>0$ is a constant depending only on the dimension $d$ and the shape-regularity parameter $\gamma_0$. 
This results in the following estimate.
\begin{est}
\begin{equation}
\label{est:new_stabsplit}
\|\nabla e \| \leq C_{S}^{\frac{1}{2}} \left( C_{\theta} \eta^2_J   +  \sum^{J}_{j=1} \| h^{-1}_{j} r_j \|^2  + \|\nabla r_0 \|^2 \right)^{\frac{1}{2}}.
\end{equation}
\end{est}
The fact that 
$
    C_{S}^{\frac{1}{2}} \left(  \sum^{J}_{j=1} \| h^{-1}_{j} r_j \|^2  + \|\nabla r_0 \|^2 \right)^{\frac{1}{2}}
$
provides an upper bound on the algebraic error has already been shown in \Cref{sec:estimRH}; see \eqref{estalg:Ruede93}.


\section{Efficiency of the estimates}
\label{sec:effectivity}

Efficiency of the estimates is described by the constant~$C_{\mathrm{eff}}$, such that
\[
    \mathrm{estimate} \leq C_{\mathrm{eff}} \cdot \| \mathrm{error} \| .
\]
Here we in particular focus on whether~$C_{\mathrm{eff}}$ depends on the number of~levels~$J$, quasi-uniformity of the coarsest mesh, and/or on the ratio $h_{\Omega} / \min_{K \in \T_0} h_K$, which is related to the size of the coarsest-level problem.

\subsection{Efficiency of the estimates on the algebraic error}
\label{sec:effectivity_alg}

We will first discuss the estimates in the form
\begin{equation}
\label{eq:form_ver2}
    \| \nabla (u_J-v_J) \| \leq C  \left(   \sum_{j=1}^J \|h^{-1}_{j} r_j \|^2 + \| \nabla r_0 \|^2 \right)^{1/2},
\end{equation}
where either $C = {C_{S,I_V}}^{\frac{1}{2}}$, or $C = {C_{S}}^{\frac{1}{2}}$ is a constant depending only on the dimension $d$ and the shape-regularity parameter~$\gamma_0$; see \eqref{eq:BJR_improved_alg} and \eqref{estalg:Ruede93}.
Using the definition of scaled residuals \eqref{eq:r_j}--\eqref{eq:r_0}, the Cauchy--Schwarz inequality and the lower bound from Appendix~\ref{sec:splitting}, \Cref{thm:stable_splitting_of_V_J} we have (see also the proof of Theorem~2.6.2 in \cite{Ruede1993})
\begin{multline*}
\sum^J_{j=1} \|h^{-1}_{j} r_j \|^2 + \|\nabla r_0 \|^2   = \sum^J_{j=0} \langle r,r_j \rangle =  \\
= \int_{\Omega} \nabla (u_J-v_J) \cdot \nabla \left( \sum^J_{j=0} r_j   \right)
 \leq \| \nabla (u_J-v_J) \| \cdot \left \lVert \nabla \left( \sum^J_{j=0} r_j   \right) \right \rVert \leq \\
 \leq \| \nabla (u_J-v_J) \| \cdot {{c_{S}}^{-\frac{1}{2}}} \left( \sum^J_{j=1} \|h^{-2}_{j} r_j \|^2 +   \|  \nabla r_0\|^2 \right)^{1/2}.
\end{multline*}
Consequently, 
\begin{align}
\label{eq:form_ver2_eff}
\left( \sum^J_{j=1} \|h^{-1}_{j} r_j \|^2 + \|\nabla r_0 \|^2   \right)^{1/2}
& \leq {{c_{S}}^{-\frac{1}{2}}} \| \nabla (u_J-v_J) \|,
\end{align}
i.e., the efficiency constant depends only on the dimension $d$ and the shape-regularity parameter $\gamma_0$.

The efficiency of the estimate~\eqref{eq:HS_alg}
\begin{equation*}
    \| \nabla (u_J-v_J) \| \leq C_{S}^{\frac{1}{2}} \overline{C}_{B}^{\frac{1}{2}}  \left(  \sum^J_{j=1} \sum^{\#\K_j}_{i=1} \frac{\langle r,\phi^{(j)}_i  \rangle^2}{\| \nabla \phi^{(j)}_i\|^2}  + \| \nabla r_0 \|^2 \right)^{1/2},
\end{equation*} 
can be shown by using $\|\nabla (u_J-v_J)  \| = \| r \|_{\left( V_J \right)^{\#}}$ and the lower bound from Appendix~\ref{sec:splitting}, \Cref{thm:frames_finite}, giving 
\begin{equation*}
    \left(  \sum^J_{j=1} \sum^{\#\K_j}_{i=1} \frac{\langle r,\phi^{(j)}_i  \rangle^2}{\| \nabla \phi^{(j)}_i\|^2}  + \| \nabla r_0 \|^2 \right)^{1/2} \leq c_{S}^{-\frac{1}{2}} \overline{c}_{B}^{-\frac{1}{2}} \| \nabla (u_J-v_J) \|.
\end{equation*} 
The efficiency constant depends only on the dimension~$d$ and the shape-regularity parameter $\gamma_0$.

Finally, for the estimate \eqref{eq:BJR_alg}
\[
    \| \nabla (u_J-v_J)\| \leq C_{I,\mathrm{2lvl}}  \sum^J_{j=1} \| h^{-1}_{j} r_{j} \| +  C_{I_0,3}  \|\nabla r_0 \|,
\]
the equivalence of the Euclidean and $\ell^1$-norm,
\[
    \sum^J_{j=1} \| h^{-1}_{j} r_{j} \| +  \|\nabla r_0 \| \leq \sqrt{J} \left(   \sum^J_{j=1} \| h^{-1}_{j} r_{j} \|^2 +   \|\nabla r_0 \|^2 \right)^{1/2},
\]
and \eqref{eq:form_ver2_eff} yield
\[
    C_{I,\mathrm{2lvl}}  \sum^J_{j=1} \| h^{-1}_{j} r_{j} \| +  C_{I_0,3}  \|\nabla r_0 \| \leq \sqrt{J} \max \{C_{I,\mathrm{2lvl}}, C_{I_0,3}\} \, c^{-\frac12}_S\| \nabla (u_J-v_J)\|.
\]
This shows the efficiency of~\eqref{eq:BJR_alg} with $C_{\mathrm{eff}}=\sqrt{J}\widetilde{C}_{\mathrm{eff}}$, where $\widetilde{C}_{\mathrm{eff}}$ depends only on~$d$ and~$\gamma_0$. 
This result does not necessarily imply a dependence of~\eqref{eq:BJR_alg} on the number of levels~$J$. However, we present below in \Cref{sec:numexp_J} a numerical experiment indicating such behaviour.

\subsection{Efficiency of estimates on total error}
\label{sec:effectivity_total}

The efficiency of the total error estimates follows from the standard result on the efficiency of the classical (one-level) residual-based error estimator. There exists a positive constant~$\overline{C}_{\mathrm{eff}}$ depending on the shape regularity of~$\T_J$ such that
\begin{equation}
\label{eq:eta_effectivity}
\left( \left( \eta^{\RHS}_J\right)^2 + \left( \eta^{\JUMP}_J \right)^2 \right)^{\frac{1}{2}} \leq \overline{C}_{\mathrm{eff}} \left( \| \nabla e \| + \osc_J\right);
\end{equation}
see, e.g.,~\cite[Section~1.4]{Verfurth2013}.
Since $\| \nabla(u_J-v_J)\| \leq \| \nabla e\|$,
we can use the efficiency of the algebraic error estimates together with \eqref{eq:eta_effectivity} to show the efficiency of the estimates on the total error (up to the oscillation term). The resulting efficiency constants depend on the same quantities as the efficiency constants for the algebraic error estimates.

For example, for the estimate \eqref{est:RudeHuber} associated with the algebraic error estimate \eqref{estalg:Ruede93},
\[
    \sqrt{2} \left( C^2_{\cls} \eta^2_{J}  + C^2_{I_{V_J},4} C_{S}  \left( \sum^{J}_{j=0}  \| h^{-1}_j r_j \|^2 + \| \nabla r_0 \|^2 \right)\right)^{\frac{1}{2}}
    \leq
    C \left( \| \nabla e \| + \osc_J\right),
\]
with $C^2 = 2 (C^2_{\cls} (\overline{C}^2_{\mathrm{eff}} + 1) + C^2_{I_{V_J},4} C_S c^{-1}_S)$.


\section{Computability of the error estimates}
\label{sec:evaluation}

In this section we address several ways in which the scaled residual norms from the estimates presented in \Cref{sec:estimates} can be evaluated or bounded. When the scaled residual norms are replaced by their bounds, proving the efficiency of the estimates from \Cref{sec:estimates} becomes a nontrivial task.

We first state an algebraic formulation of the problem~\eqref{eq:dicretized_problem}. Then we present the algebraic representation of the scaled residual norms and some of their bounds from the literature. As the main contribution of this paper, we present in \Cref{sec:strategy} a new approach for approximating the scaled residual norm on the coarsest level using adaptive number of conjugate gradient iterations. This yields total and algebraic error estimates which are provably efficient and robust with respect to the size of the coarsest-level problem.

\subsection{Algebraic formulation of the problem, residual vectors}

Given a basis $\Phi_J$ of $V_J$, the problem \eqref{eq:dicretized_problem} can be algebraically formulated as finding the vector of coefficients $\vec{u}_{{J}}\in \R^{\#\K_J}$ of the function $u_J$ in the basis $\Phi_J$ such that
\begin{equation}
\label{eq:model_problem_algebra}
    \matrx{A}_{J}\vec{u}_{{J}} = \vec{f}_J,
\end{equation}
where $\matrx{A}_{J}$ is the \emph{stiffness matrix} on the finest level~$J$, 
\[
    [\matrx{A}_{J}]_{mn} = \int_\Omega \nabla \phi^{(J)}_n \cdot \nabla \phi^{(J)}_m,
\]
and 
\[
    [\vec{f}_J]_{m} = \int_\Omega f \phi^{(J)}_m, \qquad m,n = 1, \ldots, \#\K_J.
\]
Recall that $\#\K_J$ is the cardinality of the basis $\Phi_J$. We use the standard assumption that the right-hand side vector $\vec{f}_J$ can be computed exactly using a numerical quadrature. If $\vec{f}_J$ is only known approximately, an additional term must be added to the error bounds presented above; see e.g., the discussion in \cite[Section~6]{Stevenson2007}. 
 
Let $v_J$ be an approximation of the solution $u_J$ of \eqref{eq:dicretized_problem} and $\vec{v}_J$ be the vector of coefficients of $v_J$ in the basis $\Phi_J$. Let $r$ be the residual~\eqref{eq:funcresidual} associated with~$v_J$.
Consider the residual vectors $\vec{r}_j  \in \mathbb{R}^{\# \K_j} $, $j=0,\ldots,J$, 
\begin{equation}\label{eq:def_R_j}
\left[ \vec{r}_{j}\right]_m = \langle r , \phi^{(j)}_m \rangle 
,\quad m = 1, \ldots, {\#\K_j}.
\end{equation}
The vector $\vec{r}_J$ corresponding to the finest level can be computed as 
\begin{equation}
\vec{r}_{{J}} = \vec{f}_{J} - \matrx{A}_{J}\vec{v}_{{J}}.
\end{equation}
The residual vectors corresponding to coarser levels can be computed from $\vec{r}_J$ by restriction. For the prolongation matrices $\matrx{P}^{J}_{j} \in \mathbb{R}^{\#\K_{J} \times \#\K_j}$ associated with the (nested) finite element spaces $V_j$, $V_J$ and the bases~$\Phi_j$, $\Phi_J$,
\begin{equation}
    \vec{r}_{j} = ( \matrx{P}^{J}_{j} )^{\top} \vec{r}_{J};
\end{equation}
see, e.g., \cite[Section~3.2]{Ruede1993}, or \cite[Section~2.4]{Vacek2020} for a more detailed explanation.

\subsection{The terms associated with fine levels}
\label{sec:eval_finelevels}

In this section we present an algebraic form of the term~$\| h^{-1}_j r_j \|^2$, $j=1,\ldots,J$, and several ways of bounding it by computable quantities presented in literature.

Let $\vec{c}_{j}$ be the vector of coefficients of $r_j$ in the basis $\Phi_j$. The definitions \eqref{eq:def_R_j} of $\vec{r}_j$ and \eqref{eq:r_j} of $r_j$ give
\begin{equation}\label{eq:evaluate_coarsest_residual_relation}
\left[ \vec{r}_{j} \right]_m = \langle r , \phi^{(j)}_m \rangle = \int_{\Omega} h^{-2}_j r_j \phi^{(j)}_m = \sum_n \int_{\Omega} h^{-2}_j  \left[ \vec{c}_{j} \right]_n \phi^{(j)}_n\phi^{(j)}_m, \qquad \forall m = 1, \ldots, \#\K_j.
\end{equation}
Let $\mMs_j$ be a \emph{scaled mass matrix} defined as 
\begin{equation*}
    \left[ \mMs_j \right]_{m,n} = \int_{\Omega} h^{-2}_j \phi^{(j)}_n \phi^{(j)}_m, \qquad \forall m,n = 1, \ldots, \#\K_j.
\end{equation*}
The equation \eqref{eq:evaluate_coarsest_residual_relation} can then be expressed as $\vec{r}_{j} = \mMs_j \vec{c}_{j}$ and therefore
\begin{equation}\label{eq:evaluate_term_fine_levels_expression}
\| h^{-1}_j r_j \|^2 = \int_{\Omega} h_j^{-2}  \Phi_j \vec{c}_{j} \cdot  \Phi_j \vec{c}_{j} = \vecT{c}_{j} \mMs_j \vec{c}_{j} = \vecT{r}_j ( \mMs_j )^{-1} \vec{r}_{j}.
\end{equation}
The evaluation of the term \eqref{eq:evaluate_term_fine_levels_expression} thus involves the solution of a system with a possibly large 
matrix~$\mMs_j$. Instead of computing this quantity, one can seek a computable upper bound.

Let $\matrx{D}_j$ be a diagonal matrix $\left [ \matrx{D}_j \right ]_{m,m} = \int_{\Omega}\nabla \phi^{(j)}_m\cdot \nabla \phi^{(j)}_m$, $m=1,\ldots,\#\K_j$.
The stability of basis functions (Appendix \ref{sec:splitting}, \Cref{lemma:stability_of_basis_functions}) and \eqref{eq:stable_basis_matrix_equivalence_inverse} give
\begin{align}
 c_{B} \vecT{r}_{j} \matrx{D}_j^{-1} \vec{r}_{j} \leq 
 \| h^{-1}_j r_j\|^2 &= \vecT{r}_j ( \mMs_j )^{-1} \vec{r}_{j} \leq  C_{B} \vecT{r}_{j} \matrx{D}_j^{-1} \vec{r}_{j}.
 \label{eq:bound_smm_diagstiff}
\end{align}
The upper bound in \eqref{eq:bound_smm_diagstiff} is used in \cite{Ruede1993,Huber2019} to bound the algebraic error as
\begin{equation}
\label{estalg:Ruede93_rjdiag}
\|\nabla (u_J-v_J) \| \leq C_{S}^{\frac{1}{2}} \Big( C_{B} \sum^{J}_{j=1} \vecT{r}_{j} \matrx{D}_j^{-1} \vec{r}_{j}  + \|\nabla r_0 \|^2 \Big)^{\frac{1}{2}} \leq C_{S}^{\frac{1}{2}} \overline{C}_{B}^{\frac{1}{2}} \Big( \sum^{J}_{j=1} \vecT{r}_{j} \matrx{D}_j^{-1} \vec{r}_{j}  + \|\nabla r_0 \|^2 \Big)^{\frac{1}{2}},
\end{equation}
where $\overline{C}_{B} = \max\{1, C_{B}\}$. For $\overline{c}_{B} = \min\{1, c_{B}\}$, using the lower bound in \eqref{eq:bound_smm_diagstiff} and \eqref{eq:form_ver2_eff}
\begin{equation}
\label{estalg:Ruede93_rjdiag_eff}
\Big( \sum^{J}_{j=1} \vecT{r}_{j} \matrx{D}_j^{-1} \vec{r}_{j}  + \|\nabla r_0 \|^2 \Big)^{\frac{1}{2}} \leq
\Big( c^{-1}_B\sum^{J}_{j=1} \| h^{-1}_j r_j\|^2  + \|\nabla r_0 \|^2 \Big)^{\frac{1}{2}} \leq
\overline{c}_{B}^{-\frac{1}{2}} c_{S}^{-\frac{1}{2}} \|\nabla (u_J-v_J) \|,
\end{equation}
which proves the efficiency of the bound~\eqref{estalg:Ruede93_rjdiag}. Recall that $c_{B}$, $C_{B}$, $c_{S}$, and $C_{S}$ only depend on~$d$ and~$\gamma_0$.

Noting that
\begin{equation}
\label{eq:HS_algterm}
    \vecT{r}_{j} \matrx{D}_j^{-1} \vec{r}_{j}=\sum^{\#\K_j}_{i=1}  \frac{\langle r,\phi^{(j)}_i  \rangle^2}{\| \nabla \phi^{(j)}_i\|^2} ,
\end{equation}
we see that the algebraic error bounds~\eqref{eq:HS_alg} and \eqref{estalg:Ruede93_rjdiag} are identical. 

The term \eqref{eq:evaluate_term_fine_levels_expression} can also be bounded using other techniques, e.g., using the so-called mass lumping (suggested in \cite[Section~4]{Becker1995}) or the multigrid smoothing routines (see the discussion in \cite[Section~4.5.2]{Huber2019}). By using these techniques, however, we introduce another unknown constant into the overall estimate and possibly weaken its efficiency.

In order to get a fully computable bound on \eqref{eq:evaluate_term_fine_levels_expression} (i.e.,~a bound without any unknown constant) and to avoid solving an algebraic problem with a large matrix, we can proceed similarly to \cite{Papez2018}. Define $\bar{r}_j\in L^2(\Omega)$ to be (discontinuous) piecewise affine functions on $\T_j$ such that for all $K \in \T_j$,
\begin{equation}
\int_K h^{-2}_j \bar{r}_j \phi^{(j)}_m = \left[ \vec{r}_{{j}} \right]_m \cdot \frac{1}
{\# \left\lbrace \bar{K} \in \T_j; m \mbox{ is vertex of }\bar{K} \right\rbrace}
=: \left[ \vec{r}_{{j},K} \right]_m \quad \forall \phi^{(j)}_m.
\end{equation}
This ensures that
\begin{equation}
\int_{\Omega} h^{-2}_j \bar{r}_j \phi^{(j)}_m =  \left[ \vec{r}_{{j}} \right]_m = \langle r, \phi^{(j)}_m \rangle.
\end{equation}
Since $\bar{r}_j$ is piecewise affine on elements, the norms $\| h^{-1}_j\bar{r}_j \|^2_K$ can be computed using the solutions of systems with local scaled mass matrices, i.e.,  $\| h^{-1}_j\bar{r}_j \|^2_K = \vecT{r}_{{j},K} ( \mMKsj{j} )^{-1} \vec{r}_{{j},K}$, where
\begin{equation*}
\left[ \matrx{M}^{S}_{j,K} \right]_{m,n} = \int_{K} h^{-2}_j  \phi^{(j)}_n\phi^{(j)}_m, \qquad \forall m,n \in \N_K.
\end{equation*}
For the whole term $\| h^{-1}_j {r}_j \|$ we have
\begin{equation}
\| h^{-1}_j {r}_j \|^2 \leq \| h^{-1}_j\bar{r}_j \|^2 = \sum_{K\in \T_j} \vecT{r}_{{j},K} ( \mMKsj{j} )^{-1} \vec{r}_{{j},K};
\end{equation}
cf.~\cite[Eq.~(5.9)]{Papez2018}.

\subsection{The term associated with the coarsest level}
\label{sec:evaluation_coarse}

In this section we present the algebraic form of the term $\| \nabla {r}_0\|$ and several ways of bounding it adapted from literature.

Let $\vec{c}_0$ be the vector of coefficients of $r_0$ in the basis $\Phi_0$.
Analogously to \eqref{eq:evaluate_term_fine_levels_expression}, using the definitions \eqref{eq:def_R_j} of $\vec{r}_{0}$ and \eqref{eq:r_0} of $r_0$, we have
\begin{equation*}
\left[ \vec{r}_{0} \right]_m = \langle r , \phi^{(0)}_m \rangle = \int_{\Omega} \nabla r_0 \cdot \nabla \phi^{(0)}_m = \sum_n \int_{\Omega}   \left[ \vec{c}_{0} \right]_n \nabla  \phi^{(0)}_n \cdot \nabla \phi^{(0)}_m, \qquad \forall m = 1, \ldots, \#\K_0.
\end{equation*}
Let $\matrx{A}_{0}$  be the stiffness matrix associated with the coarsest level  
\begin{equation*}
[\matrx{A}_{0}]_{mn} = \int_\Omega \nabla \phi^{(0)}_n \cdot \nabla \phi^{(0)}_m, \qquad m,n = 1, \ldots, \#\K_0.    
\end{equation*}
The vector of coefficients  $\vec{c}_{0}$ then satisfies $\matrx{A}_{0} \vec{c}_{0} = \vec{r}_{0}$.
This leads to 
\begin{equation}
\label{eq:eval_coarsest_resnorm}
\| \nabla {r}_0\|^2 = \vecT{c}_0 \matrx{A}_{0}\vec{c}_0 = \vecT{r}_0 \matrx{A}^{-1}_{0} \vec{r}_{0}. 
\end{equation}
The evaluation of the term $\| \nabla {r}_0\|^2$ thus requires solution of the system with the stiffness matrix associated with the coarsest level.
For problems where the stiffness matrix is large, this can be too costly and in some settings even unfeasible. 

An approximate solution $\widetilde{\vec{c}}_0$ of $\matrx{A}_0 \vec{c}_0 = \vec{r}_0$ computed by the (preconditioned) conjugate gradient method with a fixed number of iterations was used in \cite[Section~4.5.2]{Huber2019}.
The resulting term $\widetilde{\vec{c}}_0^*\vec{r}_0$ might not be, however, an upper bound on $\| \nabla {r}_0\|^2$. Therefore, the resulting value may not led to an upper bound on the algebraic nor the total error.

The term \eqref{eq:eval_coarsest_resnorm} can also be bounded using a quantity involving only the inverse of a diagonal matrix.
Friedrich's inequality (Appendix \ref{sec:PDElemmas}, \Cref{lemma:Friedrichs}) imply that
\begin{equation}\label{eq:fridrichs_coarsest}
\| w_0 \|^2 \leq C_F^{2} h_{\Omega}^{2} \| \nabla w_0 \|^2, \quad \forall w_0 \in V_0.
\end{equation}
Let $\matrx{M}_0$ be the mass matrix associated with the coarsest level, i.e., 
$[\matrx{M}_{0}]_{mn} = \int_\Omega  \phi^{(0)}_n  \phi^{(0)}_m$, $m,n = 1, \ldots, \#\K_0.$
The inequality \eqref{eq:fridrichs_coarsest} can be equivalently expressed algebraically as
\[
\vecT{w} \mM_{0} \vec{w} \leq C_F^{2} h_{\Omega}^{2}  \vecT{w} \matrx{A}_{0} \vec{w}, \quad \forall \vec{w}\in \R^{\# \K_0}.
\]
Since $\matrx{A}_0$ and $\matrx{M}_0$ are symmetric positive definite matrices we have
\begin{equation*}
\vecT{w} \matrx{A}^{-1}_{0} \vec{w} 
\leq
 C_F^{2} h_{\Omega}^{2}  \vecT{w} \mM^{-1}_{0} \vec{w}, \quad \forall \vec{w} \in \mathbb{R}^{\#\K_0}.
\end{equation*}
This bound can be a possibly large overestimation; see the discussion in \cite[Sects.~3.1 and~5.2]{Papez2018}.
Define the diagonal matrix $\matrx{D}_0$ as $\left [ \matrx{D}_0 \right ]_{m,m} = \int_{\Omega}\nabla \phi^{(0)}_m\cdot \nabla \phi^{(0)}_m$, $m=1,\ldots,\#\K_0$.
The term on the right-hand side can be further simplified using 
the spectral equivalence of the mass matrix $\matrx{M}_0$ with $\matrx{D}_0$; see inequality \eqref{eq:massspeceq} in Appendix \ref{sec:splitting}. 
Altogether we have
\begin{equation}
\label{eq:res0upbound_massspeceq}
\| \nabla {r}_0 \|^2 = \vecT{r}_{0} \matrx{A}^{{-1}}_{0} \vec{r}_{0} \leq C_F^{2} h_{\Omega}^{2} \vecT{r}_{0} \mM^{{-1}}_{0} \vec{r}_{0}
\leq C_{M}C_F^{2} \frac{h_{\Omega}^{2}}{\min_{K \in \T_0} h^{2}_K} \vecT{r}_{0} \matrx{D}^{-1}_0 \vec{r}_{0}.
\end{equation}
As for the efficiency, this allows to prove, using the Inverse inequality (Appendix \ref{sec:PDElemmas}, \Cref{lemma:local_inverse_ineq}) and \eqref{eq:massspeceq},
\[
    \vecT{r}_{0}\matrx{D}^{-1}_0 \vec{r}_{0} \leq \frac{C^2_{\INV}}{c_M}   \frac{\max_{K \in \T_0}h^{2}_K}{\min_{K \in \T_0}h^{2}_K}  \| \nabla {r}_0 \|^2,
\]
which indicates that bound \eqref{eq:res0upbound_massspeceq} may not be robust with respect to ${h_{\Omega}^{2}}/{\min_{K \in \T_0} h^{2}_K}$. Numerical experiments in \Cref{sec:numexp} illustrate this deficiency.

\subsection{Adaptive approximation of the coarsest-level term}
\label{sec:strategy}

In order to overcome the deficiencies described above, we now present a new approach for approximating the term \eqref{eq:eval_coarsest_resnorm}. 
It consists of applying the preconditioned conjugate gradient method (PCG) to \mbox{$\matrx{A}_{0}\vec{c}_{0} = \vec{r}_{0}$} and using lower and upper bounds on the error in PCG. 
A number of PCG iterations is determined adaptively in order to ensure the efficiency of the resulting bounds on total and algebraic errors.

Let $\vec{c}^{(i)}_{0}$ be the approximation of~$\vec{c}_{0}=\matrx{A}^{{-1}}_{0} \vec{r}_{0}$ computed at the $i$-th iteration of PCG a with zero initial guess. Let $\| \cdot\|_{\matrx{A}_0}$ be the norm generated by the matrix $\matrx{A}_0$, i.e.,  $\Anorm{\vec{v}} = \vecT{v} \matrx{A}_{0} \vec{v}$, for all $\vec{v} \in \R^{\# \K_0}$.
The term \eqref{eq:eval_coarsest_resnorm} can be expressed using the  following decomposition
\begin{equation}
\label{eq:CG_decomposition}
    \vecT{c}_{0} \matrx{A}_{0} \vec{c}_{0}  =  \underbrace{\sum^{i-1}_{m=0} \Anorm{\vec{c}^{(m+1)}_{0} - \vec{c}^{(m)}_{0}}}_{\normalsize 
    {=: \mu^2_i}} + \Anorm{\vec{c}_{0} - \vec{c}^{(i)}_{0}},
\end{equation}
which is a consequence of the \emph{local orthogonality} in PCG. This formula was already shown for CG in the seminal paper \cite[Theorem~6:1, Eq.~(6:2)]{Hestenes1952}. The terms \mbox{$\Anorm{\vec{c}^{(m)}_{0} - \vec{c}^{(m+1)}_{0}}$} can be computed at a minimal cost from the scalars available during the computations.

It is crucial to note that the local orthogonality is in CG computations preserved proportionally to the machine precision. Therefore, \eqref{eq:CG_decomposition} is valid, up to a negligible error, also in finite-precision computations; see the derivation and proofs in \cite{Strakos2002} (resp.~in \cite{Strakos2005} for the preconditioned variant).

Let $\zeta_i^2$ be an upper bound on the squared $\matrx{A}_0$-norm of the error in the PCG computation, i.e., on $\Anorm{\vec{c}_{0} - \vec{c}^{(i)}_{0}}$; see, e.g., \cite{Golub1994}, \cite{Meurant2023} and the references therein\footnote{Strictly speaking, numerical stability of the upper bounds to the A-norm of the error in CG computations has not been rigorously proved. Well-justified heuristics supported by numerical experiments however suggest their validity also in finite-precision computations; see \cite{Golub1994}, \cite{Meurant2023}.}. 
The derivation of such a bound is based on the interpretation of CG as a procedure for computing the Gauss quadrature approximation to a Riemann--Stieltjes integral and typically requires a lower bound on the smallest eigenvalue of~$\matrx{A}_0$. 
A simple lower bound can be derived using \cite[Theorem~3]{Fri72}, as
\[
    \lambda_{\min}(\matrx{A}_0) \geq C_F^{-2} h^{-2}_\Omega \min_{K \in \T_0} \lambda_{\min} (\matrx{M}_{0,K}),
\]
where $C_F$ is the constant from Friedrich's inequality (Appendix \ref{sec:PDElemmas}, \Cref{lemma:Friedrichs}) and $\matrx{M}_{0,K}$ is the local mass matrix corresponding to $K \in \T_0$. If an upper bound~$\zeta_i$ is not available, the $\matrx{A}_0$-norm of the error $\Anorm{\vec{c}_{0} - \vec{c}^{(i)}_{0}}$ can be bounded using the ideas presented in \Cref{sec:evaluation_coarse}; see also \cite[Section~3.2]{Papez2018}.

The approach then consists of running PCG for the coarsest problem until
\begin{equation}
\label{eq:strategy_stop}    
    \zeta_i^2 \leq \theta \left(\sum^J_{j=1} \vecT{r}_{j} \matrx{D}_j^{-1} \vec{r}_{j} + \mu^2_i \right),
\end{equation}
where $\theta > 0$ is a chosen parameter. Then we consider the bound 
\begin{equation}
\label{eq:procedure_bound}
     \vecT{r}_{0} \matrx{A}^{{-1}}_{0} \vec{r}_{0} \leq \mu^2_i + \zeta_i^2,
\end{equation}
which can be combined, e.g., with \eqref{estalg:Ruede93_rjdiag} to get an upper bound on the algebraic error
\begin{equation}
\label{eq:new_alg}
\| \nabla (u_J - v_J) \|  \leq   C_{S}^{\frac{1}{2}} \overline{C}_{B}^{\frac{1}{2}}  \left( \sum^{J}_{j=1} \vecT{r}_{j} \matrx{D}_j^{-1} \vec{r}_{j} + \mu^2_i + \zeta_i^2  \right)^{\frac{1}{2}}.
\end{equation}

The criterion~\eqref{eq:strategy_stop} guarantees that
\begin{equation}
\label{eq:strategy_bound_foreff} 
   \vecT{r}_{0} \matrx{A}^{{-1}}_{0} \vec{r}_{0} \leq \mu^2_i + \zeta_i^2 \leq \theta \sum^J_{j=1} \vecT{r}_{j} \matrx{D}_j^{-1} \vec{r}_{j} + (1 + \theta) \mu^2_i,
\end{equation}
which allows us to prove the efficiency of~\eqref{eq:new_alg}.
Indeed, using \eqref{eq:strategy_bound_foreff}, $\mu^2_i \leq \| \nabla {r}_0 \|^2$ (see~\eqref{eq:CG_decomposition}), and \eqref{estalg:Ruede93_rjdiag_eff}
\begin{align*}
    \left( \sum_{j=1}^J \vecT{r}_{j} \matrx{D}_j^{-1} \vec{r}_{j} + \mu^2_i + \zeta_i^2 \right)^\frac12 
    & \leq (1+\theta)^{\frac12} \left( \sum_{j=1}^J \vecT{r}_{j} \matrx{D}_j^{-1} \vec{r}_{j} + \| \nabla {r}_0 \|^2 \right)^\frac12
    \\ & \leq (1+\theta)^{\frac12} \, c_S^{-\frac12} c_B^{-\frac12} \|\nabla (u_J-v_J) \| \,.
\end{align*}

The proposed strategy follows the ideas of \cite[Section~3.2]{Papez2018}. 
In principle, the possible overestimation in $ \Anorm{\vec{c}_{0} - \vec{c}^{(i)}_{0}} \leq \zeta_i^2$ is controlled by~\eqref{eq:strategy_stop} and it is compensated for within the procedure by performing extra iterations. This allows us to prove the efficiency even if the estimate~$\zeta_i$ is not very tight. However, in such case the number of extra iterations might be quite large; see \cite[Section 7.1]{Papez2018}.

Finally, we note that $\vecT{r}_{j} \matrx{D}_j^{-1} \vec{r}_{j}$ in \eqref{eq:strategy_stop} can be replaced by any (efficient) bound on $\| h_j^{-1}r_j \|^2$. Then the algebraic error bound \eqref{eq:new_alg} should be changed accordingly, replacing $\vecT{r}_{j} \matrx{D}_j^{-1} \vec{r}_{j}$ and $\overline{C}_{B}$.

\section{Numerical experiments}
\label{sec:numexp}

The experiments focus on the efficiency of the error estimates on the algebraic error. In particular, we consider the estimate
\begin{equation}
    \label{eq:numexp_prototype}
    C\Big( \sum^{J}_{j=1} \vecT{r}_{j} \matrx{D}_j^{-1} \vec{r}_{j}  + \vecT{r}_0 \matrx{A}^{-1}_{0} \vec{r}_{0} \Big)^{\frac{1}{2}}
\end{equation}
and variants where $\vecT{r}_0 \matrx{A}^{-1}_{0} \vec{r}_{0} = \|\nabla r_0 \|^2$ is replaced by computable approximations. This prototype covers most of the algebraic error estimates from \Cref{sec:estimates}, where the scaled residual norms $\| h_j^{-1} r_j \|^2$ on the fine levels are efficiently approximated by $\vecT{r}_{j} \matrx{D}_j^{-1} \vec{r}_{j}$ using \eqref{eq:bound_smm_diagstiff}.
As shown in the previous sections, approximating the coarsest-level term $\|\nabla r_0 \|^2$ while preserving the efficiency is more subtle.

For the experiments, we consider a 3D Poisson problem on a unit cube, $\Omega = (0,1)^3$, with the exact solution
\begin{equation*}
    u(x,y,z) = x(x-1)y(y-1)z(z-1)e^{-100((x-\frac{1}{2})^2+(y-\frac{1}{2})^2+(z-\frac{1}{2})^2)}.
\end{equation*}
The problem is discretized by the standard Galerkin finite element method with piecewise affine polynomials on a sequence of six uniformly refined meshes with the same shape regularity \eqref{eq:shaperegularity}. The associated matrices are generated in the FE software FEniCS \cite{Alnaes2015,Logg2012}, and the computations are done in MATLAB 2023a.
The codes for the experiments are available from \url{https://github.com/vacek-petr/inMLEstimate}.

Given the mesh~$\T_J$ (the finest mesh varies in the experiments), the associated Galerkin solution~$u_J$ of~\eqref{eq:dicretized_problem} is for the purpose of the evaluation of the efficiency of the estimates considered (with a negligible inaccuracy) as a result of using the MATLAB backslash, or, for very large problems, using the multigrid V-cycle with an excessive number (30) of V-cycle repetitions. 
The approximation~$v_J$ to~$u_J$ is given by a multigrid solver starting with a zero approximation and repeating V-cycles until the relative energy norm of the (algebraic) error $u_J - v_J$ drops below $10^{-11}$. Each multigrid V-cycle uses 3 pre and 3 post Gauss--Seidel smoothing iterations. The problem on the coarsest level is solved using CG where the stopping criterion is based on the relative residual with the tolerance~$10^{-1}$.
In order to monitor the efficiency for varying algebraic error, we will also plot below intermediate results after completing each multigrid V-cycle.

We observed very similar results also for a set of two-dimensional problems and a 3D problem with a more complicated geometry. These experiments can be found in the repository \url{https://github.com/vacek-petr/inMLEstimate} where also the data and codes are available.

\subsection{Robustness with respect to the number of levels}
\label{sec:numexp_J}

The first experiment studies the efficiency of the estimates while varying the number of levels $J=1,2,\ldots,5$ in the hierarchy. We fix the size of the problem on the coarsest-level and, consequently, the size of the finest problem grows; see \Cref{tab:matrix_sizes_levels}.

\begin{table}[!htb]
\centering
\begin{tabular}{|r|r|}
\hline
\multicolumn{1}{|l|}{coarsest-level DoFs} & \multicolumn{1}{l|}{finest-level DoFs} \\ \hline
125     & 1 331     \\ \hline
125     & 12 167    \\ \hline
125     & 103 823   \\ \hline
125     & 857 375   \\ \hline
125     & 6 967 871 \\ \hline
\end{tabular}
\caption{Size of the problems for the experiment in \Cref{sec:numexp_J}.}
\label{tab:matrix_sizes_levels}
\end{table}

For the prototype estimate \eqref{eq:numexp_prototype}, the efficiency index
\begin{align}
\label{eq:numexp2_effb}
    I_1 &= \frac{{C}_{\mathrm{numexp}} \left(   \sum^J_{j=1} \vecT{r}_{j} \matrx{D}_j^{-1} \vec{r}_{j}  + \vecT{r}_0 \matrx{A}^{-1}_{0} \vec{r}_{0}\right)^{\frac{1}{2}}} {\|\nabla (u_J-v_J) \|},
\end{align}
is evaluated for every $J$, $v_J$, and also for intermediate results after each V-cycle.
The factor $C_{\mathrm{numexp}}$ accounts for $C_{S}^{\frac{1}{2}} \overline{C}_{B}^{\frac{1}{2}}$; see~\eqref{estalg:Ruede93_rjdiag}. For the purpose of the experiment, it is chosen as the minimal value such that the efficiency indices $I_1$ are for all $J$ and in all V-cycle repetitions above or equal to one; \mbox{$C_{\mathrm{numexp}}=1.28$}.
In order to examine the difference, we also evaluate the index
\begin{align}
\label{eq:numexp2_effa}
    I_2 &= \frac{{C}_{\mathrm{numexp}} \left( \sum^J_{j=1} \left( \vecT{r}_{j} \matrx{D}_j^{-1} \vec{r}_{j} \right )^{\frac{1}{2}} + \left (  \vecT{r}_0 \matrx{A}^{-1}_{0} \vec{r}_{0} \right )^{\frac{1}{2}} \right)}{\|\nabla (u_J-v_J) \|}, 
\end{align}
which corresponds to the algebraic error bound~\eqref{eq:BJR_alg}.

The index $I_1$ \eqref{eq:numexp2_effb} corresponds to the estimate~\eqref{estalg:Ruede93_rjdiag} that is proved to be robust with respect to the number of levels~$J$ and consequently also to the size of the finest problem; see \Cref{sec:effectivity_alg} or the original papers \cite{Ruede1993,Harbrecht2016}.
This is what the experiment confirms; see \Cref{fig:numexp2}.
Contrary to that, $I_2$ \eqref{eq:numexp2_effa} 
deteriorates with increasing~$J$. This is with alignment with the discussion at the end of of~\Cref{sec:effectivity_alg}, where we proved the efficiency of the estimate with a factor depending on $\sqrt{J}$.

\begin{figure}
    \setlength\figureheight{6cm}
    \setlength\figurewidth{0.99\textwidth}
    \centering
    %
\definecolor{mycolor1}{RGB}{122, 209, 81}
\definecolor{mycolor2}{RGB}{40, 174, 128}

\begin{tikzpicture}
\begin{axis}[%
width=0.5\figurewidth,
height=\figureheight,
at={(0\figurewidth,0\figureheight)},
xmin=2,
xmax=6,
xtick = {2,3,4,5,6},
xticklabels = {2,3,4,5,6},
xminorticks=true,
xlabel={$\#$ levels},
ylabel = {efficiency indices},
ymin=1,
ymax=3,
axis background/.style={fill=white}
]
\addplot [color=mycolor1, line width=1.5pt, only marks, mark=x, mark options={solid, mycolor1}, forget plot]
  table[row sep=crcr]{%
2	1.060127549	\\
2	1.046178948	\\
2	1.045205057	\\
2	1.044897316	\\
2	1.044515003	\\
2	1.043765215	\\
2	1.042580805	\\
2	1.040983605	\\
2	1.039000125	\\
2	1.036697857	\\
3	1.156230808	\\
3	1.099428302	\\
3	1.058979752	\\
3	1.036765711	\\
3	1.027605621	\\
3	1.024179306	\\
3	1.02265895	\\
3	1.021540175	\\
3	1.020272348	\\
3	1.018670462	\\
3	1.01684869	\\
4	1.287174421	\\
4	1.226975726	\\
4	1.160042401	\\
4	1.10295407	\\
4	1.06265759	\\
4	1.037638001	\\
4	1.022887507	\\
4	1.013993609	\\
4	1.008171786	\\
4	1.003841566	\\
4	1	\\
5	1.322366159	\\
5	1.284365205	\\
5	1.236445975	\\
5	1.183277173	\\
5	1.131905893	\\
5	1.088837323	\\
5	1.056981399	\\
5	1.035376706	\\
5	1.021185624	\\
5	1.011594927	\\
5	1.004607167	\\
6	1.328871615	\\
6	1.299700904	\\
6	1.261496885	\\
6	1.216379537	\\
6	1.168602663	\\
6	1.123639896	\\
6	1.08594398	\\
6	1.057304073	\\
6	1.036959357	\\
6	1.022864346	\\
6	1.012731637	\\
};
\label{plot:rude}
\addplot [color=mycolor2, line width=1.5pt, only marks, mark=o, mark options={solid, mycolor2}, forget plot]
  table[row sep=crcr]{%
2	1.45624915	\\
2	1.415346562	\\
2	1.411369807	\\
2	1.4117484	\\
2	1.412620702	\\
2	1.41327342	\\
2	1.413518754	\\
2	1.413333081	\\
2	1.412717758	\\
2	1.411706356	\\
3	1.935027791	\\
3	1.825946083	\\
3	1.728309644	\\
3	1.662610586	\\
3	1.631036376	\\
3	1.618720922	\\
3	1.614628165	\\
3	1.613708611	\\
3	1.61393788	\\
3	1.614515739	\\
3	1.61536466	\\
4	2.388513338	\\
4	2.276313016	\\
4	2.147571845	\\
4	2.028437127	\\
4	1.932069054	\\
4	1.861313493	\\
4	1.81343213	\\
4	1.783362555	\\
4	1.765679145	\\
4	1.755825868	\\
4	1.750310486	\\
5	2.574526269	\\
5	2.518420368	\\
5	2.437020665	\\
5	2.336665096	\\
5	2.230099708	\\
5	2.130733887	\\
5	2.047380064	\\
5	1.982892355	\\
5	1.935902809	\\
5	1.903146574	\\
5	1.881078256	\\
6	2.65002998	\\
6	2.619189568	\\
6	2.566346797	\\
6	2.491733096	\\
6	2.401571586	\\
6	2.306048991	\\
6	2.215392922	\\
6	2.136744666	\\
6	2.073228832	\\
6	2.024714425	\\
6	1.988940745	\\
};
\label{plot:BJR}
\end{axis}

\end{tikzpicture}%
    \caption{
    Efficiency indices $I_1$ (\ref{plot:rude}) and $I_2$ (\ref{plot:BJR}), \eqref{eq:numexp2_effb} and \eqref{eq:numexp2_effa}, for varying number of levels~$J$. We plot the efficiency for approximations~$v_J$ and for the associated intermediate results after each V-cycle; each corresponds to a single mark.}
    \label{fig:numexp2}
\end{figure}
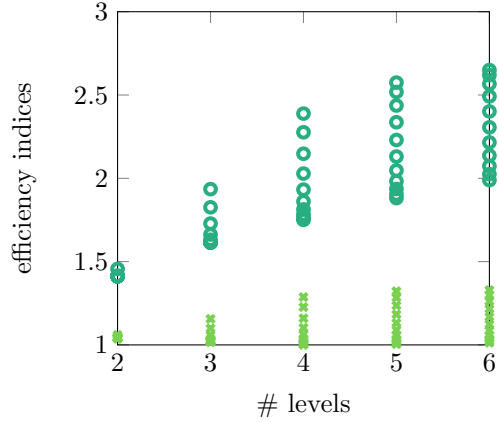

\subsection{Robustness with respect to the size of the coarsest-level problem}
\label{sec:numexp_coarse_size}
The second experiment describes the effect of the size of the coarsest-level problem on the efficiency of the estimates. 
We fix the number of levels to two ($J=1$) and vary the coarse and fine level problems; see \Cref{tab:matrix_sizes_coarse_size}.
For the approximation $v_1$ and intermediate results computed after each V-cycle, we plot the efficiency index
\begin{equation}
\label{eq:numexp1_eff}
      I_3 = \frac{C_{\mathrm{numexp}}  \Big( \vecT{r}_{1} \matrx{D}_1^{-1} \vec{r}_{1}  + \eta \Big)^{\frac{1}{2}}}{\|\nabla (u_1-v_1) \|},
\end{equation}
where $\eta$ denotes the following approximations to $\vecT{r}_0 \matrx{A}^{-1}_{0} \vec{r}_{0} = \|\nabla r_0 \|^2$:
\begin{itemize}
    \item[(i)] $ \eta = \vecT{r}_0 \overline{\vec{c}}_0$, where $\overline{\vec{c}}_0$ is computed using a direct solver for $\matrx{A}_0 \vec{c}_0 = \vec{r}_0$;
    \item[(ii)] $ \eta = \vecT{r}_0 \tilde{\vec{c}}_0$, where $\tilde{\vec{c}}_0$ is computed by 4~iterations of CG on $\matrx{A}_0 \vec{c}_0 = \vec{r}_0$ with a zero initial approximation;
    \item[(iii)] $ \eta = \dfrac{h_{\Omega}^{2}}{\min_{K \in \T_0} h^{2}_K} \vecT{r}_{0} \matrx{D}_0^{-1} \vec{r}_{0}$;
    \item[(iv)] $ \eta = \mu_i^2 + \zeta_i^2$; see \eqref{eq:procedure_bound} and the adaptive approach from \Cref{sec:strategy} using PCG. Here $\zeta_i^2$ is the upper bound on the $\matrx{A}_0$-norm in PCG from \cite[second inequality in (3.5) with updating formula for a coefficient (3.3)]{Meurant2023}. For evaluating $\zeta_i^2$, an estimate of the smallest eigenvalue of $\matrx{A}_0$ is computed by the MATLAB \texttt{eigs} function for the first four problems and extrapolated for the largest problem. In \eqref{eq:strategy_stop} we set $\theta = 0.1$ .
\end{itemize}

\begin{table}[!htb]
\centering
\begin{tabular}{|r|r|r|}
\hline
\multicolumn{1}{|l|}{coarsest-level DoFs} & \multicolumn{1}{l|}{finest-level DoFs} & \multicolumn{1}{l|}{${h_{\Omega}^{2}}/{\min_{K \in \T_0} h^{2}_K}$} \\ \hline
125     & 1 331     & 36    \\ \hline
1 331   & 12 167    & 144   \\ \hline
12 167  & 103 823   & 576   \\ \hline
103 823 & 857 375   & 2 304 \\ \hline
857 375 & 6 967 871 & 9 216 \\ \hline
\end{tabular}
\caption{Size of the problems for the experiment in \Cref{sec:numexp_coarse_size}. The table also gives the squared ratios of the diameter of the computational domain and the coarsest-level meshsize.}
\label{tab:matrix_sizes_coarse_size}
\end{table}

\begin{figure}
    \setlength\figureheight{12cm}
    \setlength\figurewidth{0.99\textwidth}
    \centering
    %
\definecolor{mycolor1}{RGB}{122, 209, 81}%
\definecolor{mycolor2}{RGB}{34, 168, 132}%
\definecolor{mycolor3}{RGB}{42, 120, 142}%
\definecolor{mycolor4}{RGB}{65, 68, 135}%

\begin{tikzpicture}
\begin{axis}[%
width=0.47\figurewidth,
height=0.45\figureheight,
at={(0\figurewidth,0.5\figureheight)},
xmin=1e2,
xmax=1e6,
xtick = {1e2,1e3,1e4,1e5,1e6},
xticklabels = {$10^2$,$10^3$,$10^4$,$10^5$,$10^6$},
xminorticks=true,
xmode=log,
xlabel={$\#$ coarsest-level DoFs},
ylabel = {efficiency index},
ymin=1,
ymax=1.35,
ytick = {1, 1.05, 1.15, 1.25, 1.35},
yticklabels = {1, 1.05, 1.15, 1.25, 1.35},
title = (i) direct solver,
axis background/.style={fill=white}
]
\addplot [color=mycolor1, line width=1.5pt, only marks, mark=x, mark options={solid, mycolor1}, forget plot]
  table[row sep=crcr]{%
125	1.060127549	\\
125	1.046178948	\\
125	1.045205057	\\
125	1.044897316	\\
125	1.044515003	\\
125	1.043765215	\\
125	1.042580805	\\
125	1.040983605	\\
125	1.039000125	\\
125	1.036697857	\\
1331	1.241958141	\\
1331	1.199012731	\\
1331	1.192282256	\\
1331	1.222500744	\\
1331	1.212439568	\\
1331	1.21999045	\\
1331	1.224665118	\\
1331	1.243752779	\\
1331	1.227780914	\\
12167	1.270424473	\\
12167	1.251099661	\\
12167	1.277174017	\\
12167	1.255803172	\\
12167	1.273413859	\\
12167	1.24343544	\\
12167	1.264416157	\\
12167	1.245783336	\\
12167	1.268351459	\\
103823	1.280934038	\\
103823	1.272934696	\\
103823	1.280369841	\\
103823	1.273461628	\\
103823	1.279940637	\\
103823	1.27422859	\\
103823	1.276800894	\\
103823	1.2625785	\\
103823	1.277571254	\\
857375	1.282374813	\\
857375	1.278947532	\\
857375	1.279553315	\\
857375	1.278681946	\\
857375	1.281295424	\\
857375	1.270848126	\\
857375	1.278286488	\\
857375	1.281089581	\\
857375	1.279854347	\\
};
\end{axis}

\begin{axis}[%
width=0.47\figurewidth,
height=0.45\figureheight,
at={(0\figurewidth,0\figureheight)},
xmin=1e2,
xmax=1e6,
xtick = {1e2,1e3,1e4,1e5,1e6},
xticklabels = {$10^2$,$10^3$,$10^4$,$10^5$,$10^6$},
xminorticks=true,
xmode=log,
xlabel={$\#$ coarsest-level DoFs},
ylabel = {efficiency index},
ymin=0.99,
ymax=30,
ytick = {1,10,20,30},
yticklabels = {1,10,20,30},
title = (iii) diagonal approximation,
axis background/.style={fill=white}
]
\addplot [color=mycolor4, line width=1.5pt, only marks, mark=x, mark options={solid, mycolor4}, forget plot]
  table[row sep=crcr]{%
125	3.643481741	\\
125	3.548795	\\
125	3.562505522	\\
125	3.567635116	\\
125	3.572399911	\\
125	3.579406648	\\
125	3.587820462	\\
125	3.596871748	\\
125	3.606222172	\\
125	3.615598463	\\
1331	5.475365475	\\
1331	7.187289423	\\
1331	7.519651387	\\
1331	6.663909092	\\
1331	7.237255081	\\
1331	6.974027357	\\
1331	6.752067312	\\
1331	5.691218262	\\
1331	6.548453535	\\
12167	7.594331288	\\
12167	11.5201775	\\
12167	5.104844821	\\
12167	10.29721426	\\
12167	6.291837076	\\
12167	11.23313155	\\
12167	8.039990797	\\
12167	10.40853205	\\
12167	6.964525459	\\
103823	7.990045524	\\
103823	13.7250479	\\
103823	7.586770692	\\
103823	13.03413794	\\
103823	7.20636941	\\
103823	11.44778497	\\
103823	9.734693637	\\
103823	17.84313534	\\
103823	8.969744381	\\
857375	9.487585085	\\
857375	17.51743141	\\
857375	15.14293264	\\
857375	16.83626299	\\
857375	10.75778222	\\
857375	28.49184729	\\
857375	17.50778078	\\
857375	11.13336615	\\
857375	13.9629496	\\
};
\end{axis}

\begin{axis}[%
width=0.47\figurewidth,
height=0.45\figureheight,
at={(0.5\figurewidth,0.5\figureheight)},
xmin=1e2,
xmax=1e6,
xtick = {1e2,1e3,1e4,1e5,1e6},
xticklabels = {$10^2$,$10^3$,$10^4$,$10^5$,$10^6$},
xminorticks=true,
xmode=log,
xlabel={$\#$ coarsest-level DoFs},
ylabel = {efficiency index},
ymin=0.3,
ymax=1.35,
ytick = {0.3, 0.6, 1, 1.35},
yticklabels = {0.3, 0.6,  1, 1.35},
title = (ii) 4 iterations of CG,
axis background/.style={fill=white}
]
\addplot [color=mycolor3, line width=1.5pt, only marks, mark=x, mark options={solid, mycolor3}, forget plot]
  table[row sep=crcr]{%
125	1.052953762	\\
125	1.041386975	\\
125	1.040707081	\\
125	1.040573096	\\
125	1.040445363	\\
125	1.039987783	\\
125	1.039049527	\\
125	1.037627744	\\
125	1.035763917	\\
125	1.033535311	\\
1331	1.090537644	\\
1331	1.122562996	\\
1331	1.161502496	\\
1331	1.121295003	\\
1331	1.163577878	\\
1331	1.16461258	\\
1331	1.151944576	\\
1331	1.124954089	\\
1331	1.167159049	\\
12167	0.894151311	\\
12167	1.177869434	\\
12167	0.747279677	\\
12167	1.092170551	\\
12167	0.859227929	\\
12167	1.099648561	\\
12167	0.852922481	\\
12167	1.028070438	\\
12167	0.811734957	\\
103823	0.664277809	\\
103823	0.944872136	\\
103823	0.604271656	\\
103823	0.950998885	\\
103823	0.51446871	\\
103823	0.721969875	\\
103823	0.617958347	\\
103823	0.994887473	\\
103823	0.617760879	\\
857375	0.419540647	\\
857375	0.633974861	\\
857375	0.50074618	\\
857375	0.551037089	\\
857375	0.367759372	\\
857375	0.886537728	\\
857375	0.539138841	\\
857375	0.37414455	\\
857375	0.456147149	\\
};
\end{axis}
\begin{axis}[%
width=0.47\figurewidth,
height=0.45\figureheight,
at={(0.5\figurewidth,0\figureheight)},
xmin=1e2,
xmax=1e6,
xtick = {1e2,1e3,1e4,1e5,1e6},
xticklabels = {$10^2$,$10^3$,$10^4$,$10^5$,$10^6$},
xminorticks=true,
xmode=log,
xlabel={$\#$ coarsest-level DoFs},
ylabel = {efficiency index},
ymin=1,
ymax=1.35,
ytick = {1, 1.05, 1.15, 1.25, 1.35},
yticklabels = {1, 1.05, 1.15, 1.25, 1.35},
title = (iv) adaptive CG approximation,
axis background/.style={fill=white}
]
\addplot [color=mycolor2, line width=1.5pt, only marks, mark=x, mark options={solid, mycolor2}, forget plot]
  table[row sep=crcr]{%
125	1.063035423	\\
125	1.060909231	\\
125	1.05875508	\\
125	1.057666199	\\
125	1.056599701	\\
125	1.055414964	\\
125	1.054045523	\\
125	1.052474878	\\
125	1.050731002	\\
125	1.048909537	\\
1331	1.24428591	\\
1331	1.207893296	\\
1331	1.20356704	\\
1331	1.241824051	\\
1331	1.221111451	\\
1331	1.236737634	\\
1331	1.234318921	\\
1331	1.255910934	\\
1331	1.239630094	\\
12167	1.27548755	\\
12167	1.281744451	\\
12167	1.279763172	\\
12167	1.286104778	\\
12167	1.279603729	\\
12167	1.271161411	\\
12167	1.271988005	\\
12167	1.284311264	\\
12167	1.276795799	\\
103823	1.288865338	\\
103823	1.316841398	\\
103823	1.287037246	\\
103823	1.312980409	\\
103823	1.286445416	\\
103823	1.320304378	\\
103823	1.282888109	\\
103823	1.311022803	\\
103823	1.28727102	\\
857375	1.287089274	\\
857375	1.328076966	\\
857375	1.284624962	\\
857375	1.326773106	\\
857375	1.286397041	\\
857375	1.319489497	\\
857375	1.311006872	\\
857375	1.302649872	\\
857375	1.314740602	\\
};
\end{axis}
\end{tikzpicture}%
    \caption{
    Efficiency indices $I_3$ \eqref{eq:numexp1_eff} for the experiment in \Cref{sec:numexp_coarse_size}.
    The estimates differ in the way of approximating the coarsest-level term $\|\nabla r_0 \|^2 = \vecT{r}_0\matrx{A}^{-1}_0\vec{r}_0$. This term is: computed by a direct solver for the coarsest problem (i), approximated using four~iterations of the CG solver (ii), approximated by replacing the stiffness matrix by its scaled diagonal approximation (iii), determined using the adaptive CG approximation (iv).}
    \label{fig:numexp1}
\end{figure}
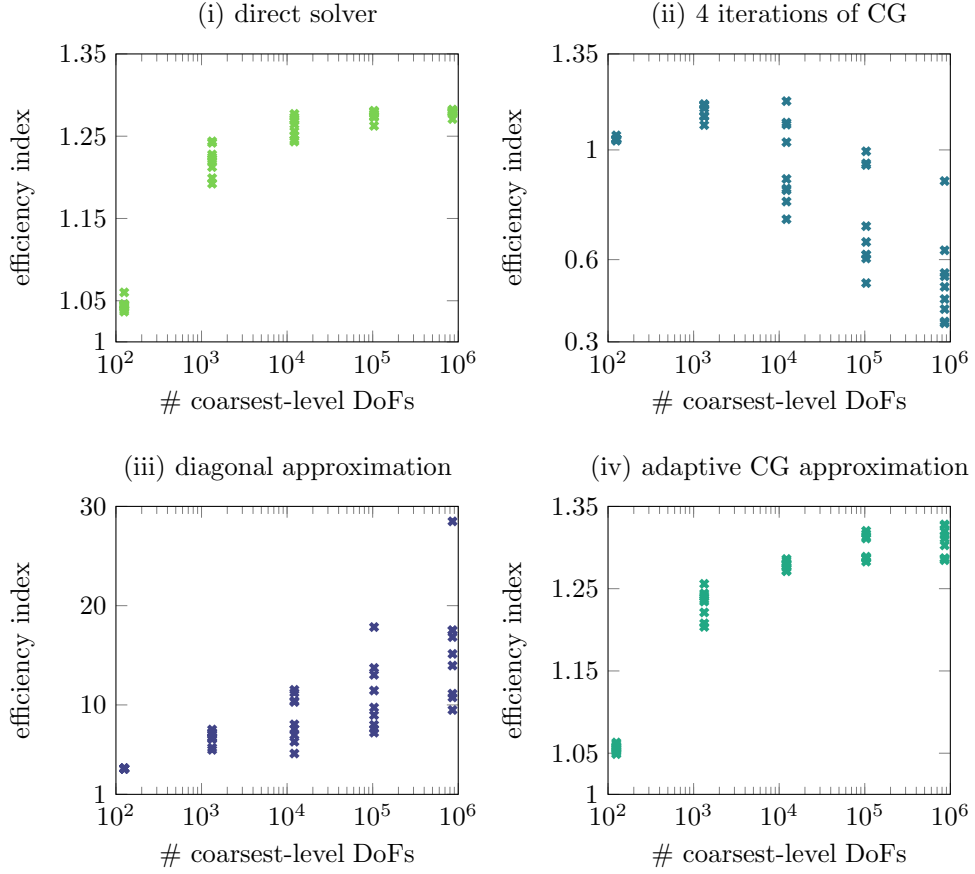

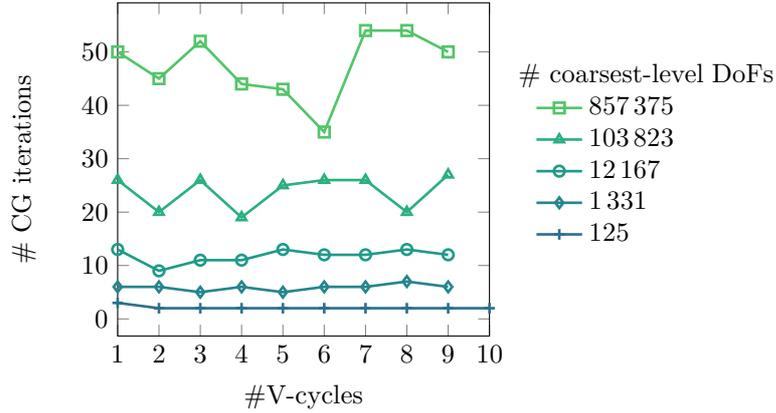
\begin{figure}
    \setlength\figureheight{6cm}
    \setlength\figurewidth{0.99\textwidth}
    \centering
\definecolor{mycolor1}{RGB}{82, 197, 105}
\definecolor{mycolor2}{RGB}{42, 176, 127}
\definecolor{mycolor3}{RGB}{30, 155, 138}
\definecolor{mycolor4}{RGB}{37, 133, 142}
\definecolor{mycolor5}{RGB}{45, 112, 142}

\begin{tikzpicture}
\begin{axis}[%
width=0.5\figurewidth,
height=\figureheight,
at={(0\figurewidth,0\figureheight)},
legend style={at={(1.1,0.75)},anchor=north west, draw = none},
legend cell align={left},
xtick = {1,2,3,4,5,6,7,8,9,10},
xticklabels = {1,2,3,4,5,6,7,8,9,10},
ytick = {0,10,20,30,40,50},
yticklabels = {0,10,20,30,40,50},
xmin=1,
xmax=10,
xlabel={\#V-cycles},
ylabel={\# CG iterations},
yminorticks=true,
axis background/.style={fill=white}
]
\addplot [color=mycolor1, line width=1.0pt, mark=square]
  table[row sep=crcr]{%
1	50	\\
2	45	\\
3	52	\\
4	44	\\
5	43	\\
6	35	\\
7	54	\\
8	54	\\
9	50	\\
};
\addlegendentry{857\,375}
\label{plot:new_iter_size_857375}
\addplot [color=mycolor2, line width=1.0pt, mark=triangle]
  table[row sep=crcr]{%
1	26	\\
2	20	\\
3	26	\\
4	19	\\
5	25	\\
6	26	\\
7	26	\\
8	20	\\
9	27	\\
};
\addlegendentry{103\,823}
\label{plot:new_iter_size_103823}
\addplot [color=mycolor3, line width=1.0pt, mark=o]
  table[row sep=crcr]{%
1	13	\\
2	9	\\
3	11	\\
4	11	\\
5	13	\\
6	12	\\
7	12	\\
8	13	\\
9	12	\\
};
\addlegendentry{12\,167}
\label{plot:new_iter_size_12167};
\addplot [color=mycolor4, line width=1.0pt, mark=diamond]
  table[row sep=crcr]{%
1	6	\\
2	6	\\
3	5	\\
4	6	\\
5	5	\\
6	6	\\
7	6	\\
8	7	\\
9	6	\\
};
\addlegendentry{1\,331}
\label{plot:new_iter_size_1331}

\addplot [color=mycolor5, line width=1.0pt, mark=+]
  table[row sep=crcr]{%
1	3	\\
2	2	\\
3	2	\\
4	2	\\
5	2	\\
6	2	\\
7	2	\\
8	2	\\
9	2	\\
10	2	\\
};
\label{plot:new_iter_size_125}
\addlegendentry{125}
\end{axis}
\node at (6.95,3.45) {\# coarsest-level DoFs};

\end{tikzpicture}
    \caption{
    Number of CG iterations determined by the adaptive approach described in \Cref{sec:strategy}, which is used to estimate the residual norm~$\|\nabla r_0 \|$ associated with the coarsest level. The horizontal axis indicates the number of V-cycles used in computing the approximation~$v_J$.
    }
    \label{fig:numexp1_n_iter}
\end{figure}

\noindent
The factor $C_{\mathrm{numexp}} = 1.28$ accounts for $C_{S}^{\frac{1}{2}} \overline{C}_{B}^{\frac{1}{2}}$; and was set as a minimal value such that the efficiency index \eqref{eq:numexp1_eff} for the variant (i) with the direct solver is above or equal to one.
The results are plotted in \Cref{fig:numexp1}.

The variant (i), where the coarsest-level term is computed using a direct solver, exhibits only a very mild increase of the efficiency index~$I_3$ \eqref{eq:numexp1_eff}. Recall, however, that 
using a direct solver is for large problems in practice unfeasible.

The variant (ii), which uses four iterations of CG to approximate the term on the coarsest level, provides no longer an upper bound on the algebraic error. It is not surprising that a fixed number of CG iterations is not sufficient for problems with increasing size. In the newly proposed adaptive approach, the number of CG iteration varies and it is determined automatically.

For the variant (iii), where the stiffness matrix on the coarsest level is replaced by its scaled diagonal (see \eqref{eq:res0upbound_massspeceq}), the efficiency indices deteriorate with 
the increasing ratio $h_{\Omega}^{2} / \min_{K \in \T_0} h^{2}_K$; see \Cref{tab:matrix_sizes_coarse_size}.
The experiment illustrates that the estimate is not robust with respect to this ratio; see the discussion at the end of \Cref{sec:evaluation_coarse}.

When the term~$\|\nabla r_0 \|$ is approximated using the adaptive computation~(iv) proposed in \Cref{sec:strategy}, the efficiency behaves as in the case~(i). Unlike in~(i), the approximation in~(iv) is computable even for very large problems on the coarsest-level. 
The adaptively chosen number of CG iterations performed within the new procedure is plotted in \Cref{fig:numexp1_n_iter}.

\section{Conclusions}
\label{sec:conclusions}

This paper presents residual-based a~posteriori error estimates on total and algebraic errors in multilevel frameworks inspired by several derivations from the literature.
It starts with algebraic error estimates containing sum of the (scaled) residual norms over the levels, including the coarsest one. Total error estimates incorporate additionally the standard residual-based estimator evaluated on the finest level. 
Efficiency and robustness with respect to the number of levels and the size of the algebraic problem on the coarsest level were for several estimates of this type proved in literature. However, the estimates containing residual norms are not easily computable and applicable in practice.

Approximation of the scaled residual norms, i.e., the terms $\vecT{r}_j \matrx{X}_j^{-1} \vec{r}_{j}$, where $\vec{r}_{j}$ is the algebraic residual associated with the level~$j$ on all but the coarsest level does not represent a significant difficulty. Except for the coarsest level, $\matrx{X}_j$ is the scaled mass matrix denoted in the paper as~$\mMs_j$ and the term $\vecT{r}_j (\mMs_j)^{-1} \vec{r}_{j}$ can be bounded from above by the simpler term~$\vecT{r}_{j} (\matrx{D}_j)^{-1} \vec{r}_{j}$, where $\matrx{D}_j$ is an appropriate diagonal matrix, without affecting the efficiency and robustness.

Evaluating the residual norm~$\|\nabla r_0 \|^2 = \vecT{r}_0 \matrx{A}^{-1}_{0} \vec{r}_{0}$ associated with the coarsest level, where $\matrx{A}_{0}$ is the stiffness matrix, is more subtle. 
When using bounds or techniques to approximate $\vecT{r}_0 \matrx{A}^{-1}_{0} \vec{r}_{0}$ presented in the literature, the resulting (multilevel) estimates on the total and algebraic errors are no longer guaranteed to be independent of the size of the coarsest-level problem. This behaviour is illustrated by numerical experiments.

The approach proposed in this paper approximates the coarsest-level term~$\|\nabla r_0 \|^2$ using the preconditioned conjugate gradient iterates. 
A number of PCG iterations is determined adaptively such that the efficiency of the bound does not deteriorate with increasing size of the coarsest-level problem, and the efficiency and robustness of the multilevel error estimates is preserved. Numerical results support the theoretical findings.

The estimates for total and algebraic errors involve some constants that 
must be approximately determined, which involves heuristics. 
For residual-based error estimates, the constants can be determined for smaller problems with the same or analogous geometry where an approximation with very small algebraic error can be computed; see, e.g., the discussion in~\cite[Section~7]{Becker1995}. 
Since the new result in \Cref{sec:strategy} proves the robustness of the adaptive estimate with respect to the size of the coarsest-level problem, it provides a justification for extrapolating the estimated values of the constants from smaller to larger problems.

In view of a recent trend on using multiple precision in multigrid algorithms (see, e.g., \cite{McCormick2021, Tamstorf2021}), it is worth considering extension of the presented results to include effects of inexact (limited-precision) operations. This will require substantial further analysis. We plan to address this topic in the future.

\backmatter
\bmhead{Acknowledgments}
The authors wish to thank to Peter Oswald for his help on proving \Cref{lemma:decomposition-using-quasi-interpol-to-V-upper-bound} and to Erin C.~Carson for valuable comments, which improved the text.

\bmhead{Statements and Declarations}
Jan Papež and Zdeněk Strakoš are members of Ne\v{c}as Center for Mathematical Modeling.
The work of Petr Vacek was supported by Charles University PRIMUS project no. PRIMUS/19/SCI/11, the Exascale Computing Project (17-SC-20-SC), a collaborative effort of the U.S. Department of Energy Office of Science and the National Nuclear Security Administration, and by the European Union (ERC, inEXASCALE, 101075632). 
The work of Jan Papež has been supported by the Czech Academy of Sciences (RVO~67985840) and by the Grant Agency of the Czech Republic (grant no.~23-06159S).
Views and opinions expressed are those of the authors only and do not necessarily reflect those of the European Union or the European Research Council. Neither the European Union nor the granting authority can be held responsible for them.

\begin{appendices}
\section{Auxiliary results from the theory of PDEs and FEM}
\label{sec:PDElemmas}

The following results are standard in PDE and FEM analysis. They are presented in various forms and sometimes with different names. We provide them in forms suitable for our development, with some standard references where the proofs can be found.

\begin{lemma}[Bramble--Hilbert lemma]
\label{lemma:BH}
There exists a constant $C_{\BH}(\T)>0$ depending only on $d$ and $\gamma_{\T}$ such that for all $K \in \T$
\begin{align}
\inf_{c\in \R} \| w-c \|_{\omega_K} & \leq C_{\BH}(\T) h_{K} \| \nabla w \|_{\omega_K} &&   \forall w \in H^1(\omega_K), \\
\inf_{p\in \mathbb{P}^1(\omega_K)}  \| w-p \|_{\omega_K} & \leq C_{\BH}(\T) h^2_{K} |  w |_{H^2(\omega_K)} &&   \forall w \in H^2(\omega_K).
\end{align}
\end{lemma}
\noindent
For the proof, see, e.g., \cite[p.~490]{Scott1990} and references therein.

\begin{lemma}[Friedrich's inequality]
\label{lemma:Friedrichs}
Let $\omega\subset \R^d$ be a bounded domain. There exists a constant $C_{F}(\omega)>0$ such that for all $w \in H^1(\omega)$ which have a zero trace on a part of the boundary $\partial \omega$ of nonzero measure
\begin{equation}
\label{eq:friedrichs}
  \| w \|_{\omega}\leq C_{F}(\omega) h_{\omega} \| \nabla w \|_{\omega}.  
\end{equation}
\end{lemma}
\noindent
When using Friedrich's inequality on patches associated with the elements of the triangulation $\T$, there exists a constant $C_F(\T)$ depending only on $d$ and $\gamma_{\T}$ such that for all $K\in \T$
\begin{equation*}
    C_F(\omega_K)\leq C_F(\T);
\end{equation*}
see, e.g., \cite[Chapter~18]{RekBook80}.

\begin{lemma}[Trace inequality]
\label{lemma:trace_ineq}
There exists a constant $C_{\TR}(\T)>0$ depending only on $d$ and $\gamma_{\T}$ such that for all $K \in \T$ and all $w \in H^1(K)$
\begin{equation}
\label{eq:trace_ineq}
\| w \|^2_{\partial K} \leq C_{\TR}(\T) \left(  h_{K}^{-1} \| w \|^2_{K} + h_{K} \| \nabla w \|^2_{K} \right).
\end{equation}
\end{lemma}
\noindent
For the proof, see, e.g., \cite[Proposition~4.1]{Carstensen1999}. 
\begin{lemma}[Inverse inequality]
\label{lemma:local_inverse_ineq}
There exists a constant $C_{\INV}(\T)>0$ depending only on $d$ and $\gamma_{\T}$ such that for all $K \in \T$ and all $w_{\T} \in S_{\T}$
\begin{equation}
\label{eq:local_inverse_ineq}
\| \nabla w_{\T} \|_{K} \leq  C_{\INV}(\T) h^{-1}_K \| w_{\T} \|_{K}.
\end{equation}
\end{lemma}
\noindent
For the proof, see, e.g., \cite[Lemma~1.27]{Elman2005}.

The following lemma is a consequence of \Cref{lemma:trace_ineq} and~\Cref{lemma:local_inverse_ineq}.

\begin{lemma}\label{lemma:trace+inverse_ineq}
There exists a constant $C_{\TI}(\T)>0$ depending only on $d$ and $\gamma_{\T}$ such that for all $K \in \T$ and all $w_{\T} \in S_{\T}$
\begin{equation}
\| w_{\T}\|^2_{\partial K} \leq C_{\TI}(\T) h^{-1}_K \| w_{\T}\|^2_{K}.
\end{equation}
\end{lemma}
\begin{proof}
Bounding $ \| w_{\T} \|^2_{\partial K}$ using the trace inequality yields
\begin{equation*}
\|w_{\T}\|^2_{\partial K} \leq C_{\TR}(\T) \left( h^{-1}_K \| w_{\T} \|^2_{K} + h_K \| \nabla w_{\T} \|^2_{K} \right). 
\end{equation*}
Applying the inverse inequality gives
\begin{align*}
\| w_{\T} \|^2_{\partial K} & \leq C_{\TR} (\T) \left( h^{-1}_K \| w_{\T} \|^2_{K} + h^{-1}_K  C^2_{\INV}(\T)  \| w_{\T} \|_{K} \right)  \\
&= C_{\TR}(\T) ( 1 + C^2_{\INV}(\T)) h^{-1}_K  \| w_{\T} \|^2_{K}.
\end{align*}
\end{proof}


\section{Quasi-interpolation operators}
\label{sec:quasi-interpolation-operator}

A quasi-interpolation operator is not explicitly used in the construction of the estimators but it is a crucial tool for proving the bounds.
In this section we present a quasi-interpolation operator as a generalization of nodal interpolation to integrable functions. We consider the quasi-interpolation operator used in \cite{Oswald1994a}, which is closely related to the operator from \cite{Scott1990}. Other, slightly different quasi-interpolation operators can be found, e.g., in  \cite{ Clement1975,Verfurth2013,Carstensen1999}. We list and prove some of the properties of the operator to be used later. The proofs of the properties are based on standard techniques. To keep the text self-contained and formally accurate we provide most of the proofs below.

The results in this section are mostly derived for a single mesh~$\T$. We show that the constants only depend on the dimension~$d$ and the shape-regularity~$\gamma_\T$ and therefore we can again use them in the mesh hierarchy with the dependence on~$d$ and~$\gamma_{0}$.

\subsection{Nodal interpolation and its generalization}

For a node $z \in \N_ \T$, let $\Psi_z:C(\overline{\Omega}) \rightarrow \R$ denote the linear functional evaluation at point~$z$, i.e.,
\begin{equation*}
\Psi_z(w)=w(z) \quad \forall w \in C(\overline{\Omega}).
\end{equation*}
The standard nodal interpolation operator $\mathcal{I}: C(\overline{\Omega}) \to S_\mathcal{T}$ for continuous functions is defined as (see, e.g., \cite{Ciarlet1978,Brenner2007})
\begin{equation*}
\mathcal{I} w = \sum_{z \in \N_\T} \Psi_z(w) \phi_z \quad	\forall w \in C(\overline{\Omega}).
\end{equation*}
In order to construct an analogy of the operator $\mathcal{I}$ for functions from $L^1(\Omega)$, the point evaluation is replaced by an appropriate average of the approximated function.
We will consider the quasi-interpolation operator defined in \cite{Oswald1994a} and \cite{Stevenson2007}.

For a node $z \in \N_{\T}$, let $K_z$ be a fixed element having $z$ as its vertex, i.e., $z\in K_z$. Let $ \mathbb{P}^1(K_z)$ denote the space of linear polynomials on $K_z$ and denote by $\widetilde{\Psi}_{z}$ the restriction of the linear functional $\Psi_z$ to functions from $\mathbb{P}^1(K_z)$. Since $\mathbb{P}^1(K_z)$ is a finite-dimensional space, the linear functional $\widetilde{\Psi}_{z}$ is bounded and it therefore belongs to the dual space $(\mathbb{P}^1(K_z))^{\#}$. Considering the space $\mathbb{P}^1(K_z)$ equipped with the $L^2$-inner product, the Riesz representation theorem (see, e.g., \cite[Sect.~2.4]{Brenner2007}) yields the existence of a function \mbox{$\psi_z \in \mathbb{P}^1(K_z)$} such that
\begin{equation*}
\widetilde{\Psi}_{z}(w)=w(z)=\int_{K_z} w \psi_z, \quad \forall w \in \mathbb{P}^1(K_z).
\end{equation*}
Since $\psi_z$ is the Riesz representation of the point evaluation at $z$, it holds for all $z_1,z_2 \in \N_{\T}$ (recall that $\phi_{z_2}$ is the hat function associated with $z_2$) that
\begin{equation}
\label{eq:quasi-iterpol-zero-one}
\int_{K_{z_1}} \phi_{z_2} \psi_{z_1}=\phi_{z_2}(z_1) =
\begin{cases}
	 1 & z_1=z_2, \\
     0 & z_1\neq z_2.
\end{cases}
\end{equation}

We will consider the quasi-interpolation operators defined as follows
\begin{align}
& I_{S_\T}: L^1(\Omega) \rightarrow S_\T,\quad I_{S_\T} w = \sum_{z \in \N_{\T}} \left(  \int_{K_z} w \psi_z \right) \phi_z, \\
& I_{V_\T}: L^1(\Omega) \rightarrow V_\T,\quad I_{V_\T} w = \sum_{z \in \K_{\T}} \left(  \int_{K_z} w \psi_z \right) \phi_z.
\end{align}

These definitions and relation \eqref{eq:quasi-iterpol-zero-one} imply that $I_{S_{\T}}$ and $I_{V_{\T}}$ are projections onto $S_\T$ and $V_\T$, respectively. Further, $I_{S_{\T}}$ preserves linear polynomials on $\Omega$ and $I_{V_{\T}}$ preserves linear polynomials on $\omega_K$ for any element $K \in \T$ whose patch~$\omega_K$ does not intersect with the boundary of $\Omega$, i.e., $\overline{\omega_K} \cap \partial \Omega = \emptyset$.

\subsection{Local estimates}
\label{sec:interpolation_localestims}
\noindent
We now present local (elementwise) bounds on an interpolant $I_{S_{\T}} w $ and the interpolation error $w - I_{S_{\T}} w $.

\begin{thm}\label{thm:quasi-interpolation_to_S_estimates}
There exist positive constants $\widehat{C}_{I_{S_{\T}},\ell}$, $\ell = 1,2,3,4$, depending only on~$d$ and~$\gamma_{\T}$ such that for all elements $K\in\T$, 
\begin{align}
\label{eq:quasi-interpolation_to_S_L2_stab}
 \| I_{S_{\T}} w \|_{K} & \leq \widehat{C}_{I_{S_{\T}},1} \| w \|_{\omega_K}   &&\forall w \in L^2(\omega_K), \\
 \label{eq:quasi-interpolation_to_S_L2_approx}
\| w - I_{S_{\T}} w \|_{K} & \leq \widehat{C}_{I_{S_{\T}},2} h_K \| \nabla w \|_{\omega_K}  &&\forall w \in H^1(\omega_K),  \\
\label{eq:quasi-interpolation_to_S_L2_approx_H2}
\| w - I_{S_{\T}} w \|_{K} & \leq \widehat{C}_{I_{S_{\T}},3} h^2_K | w |_{H^2(\omega_K)}  &&\forall w \in H^2(\omega_K),  \\
\label{eq:quasi-interpolation_to_S_H1_stab}
\| \nabla I_{S_{\T}} w \|_{K} & \leq \widehat{C}_{I_{S_{\T}},4} \| \nabla w \|_{\omega_K}   &&\forall w \in H^1(\omega_K).
\end{align}
\end{thm}

\begin{proof}
The steps in the proof are inspired by  
\cite[pp.~17--18]{Oswald1994a} and \cite[Sections~3--4]{Scott1990}.

Using standard affine transformation to a reference element it can be shown that there exists a constant $C_{\psi}>0$ depending only on~$d$ and $\gamma_\mathcal{T}$ such that for all $z\in \N_\T$,
\begin{equation}
\label{eq:psi_L_infty_bound}
\| \psi_z \|_{L^{\infty}(K_z)} \leq C_{\psi} \meas{K_z}^{-1},
\end{equation}
and that there exists a constant $C_{\phi}>0$ depending only on~$d$ and $\gamma_\mathcal{T}$ such that for all $K\in \T$ and all $z\in \K_K$,
\begin{equation}\label{eq:phi_L_infty_bound}
\| \nabla \phi_z \|_{L^{\infty}(K)} \leq C_{\phi} \rho^{-1}_K;
\end{equation}
see, e.g., \cite[pp.~487--488]{Scott1990}.

Using Hölder's inequality and \eqref{eq:psi_L_infty_bound} we can show that for  all $z \in \N_\T$ and all $w \in L^2(K_z)$
\begin{equation}
\label{eq:quasi}
\left \lvert \int_{K_z} w \psi_z  \right \rvert ^2\leq \| \psi_z\|^2_{L^{\infty}(K_z)} \left( \int_{K_z} \left \lvert w   \right \rvert \right)^2 \leq C^2_{\psi}  \meas{K_z}^{-2} \meas{K_z} \| w \|^2_{K_z} = C^2_{\psi}  \meas{K_z}^{-1}   \| w \|^2_{K_z}.  
\end{equation}

\medskip
We now proceed to prove the inequality \eqref{eq:quasi-interpolation_to_S_L2_stab}. Using that $0\leq \phi_z \leq 1$ gives
\begin{align*}
\left \lVert  I_{S_{\T}} w  \right \rVert ^2_{K} &=  \left \lVert  \sum_{ z \in \N_K } \left( \int_{K_z} w \psi_z  \right)  \phi_z  \right \rVert ^2_{K}  \leq \left \lvert \sum_{ z \in \N_K}  \int_{K_z} w \psi_z   \right \rvert ^2  \meas{K} \\
& \leq ( \#\N_K ) \meas{K}   \sum_{ z \in \N_K} \left \lvert \int_{K_z} w \psi_z   \right \rvert ^2.
\end{align*}
The inequality \eqref{eq:quasi} and the fact that $\#\N_K \leq d+1$ yields
\begin{align}
\nonumber 
\| I_{S_{\T}} w \|^2_{K} & \leq  ( d+1 ) \meas{K}  \sum_{ z \in \K_K} C^2_{\psi}  \meas{K_z}^{-1}   \| w \|^2_{K_z} \\ \label{eq:local_bound_L2_inter}
& \leq ( d+1 ) \meas{K} C^2_{\psi}   \max_{z \in \N_K} \meas{K_z}^{-1}  \| w \|^2_{\omega_K} \\
& \leq    ( d+1 ) C^2_{\psi}  \frac{\meas{K}}{\min_{z \in \N_K}  \meas{K_z}}   \| w \|^2_{\omega_K}.
\end{align}
Since $|K|$ and $|K_z|$, $z \in \N_K$, are comparable up to a constant depending on~$d$ and $\gamma_\T$ (in a shape-regular mesh, we can compare the size of any neighboring elements), inequality \eqref{eq:quasi-interpolation_to_S_L2_stab} follows.

To prove the inequalities \eqref{eq:quasi-interpolation_to_S_L2_approx} and \eqref{eq:quasi-interpolation_to_S_L2_approx_H2}, let $p$ be a constant or linear polynomial on $\omega_K$. Using the fact that $I_{S_{\T}}$ reproduces linear polynomials and \eqref{eq:quasi-interpolation_to_S_L2_stab} we get
\begin{align*}
\| w - I_{S_{\T}} w\|_{K} &=\| w-p - I_{S_{\T}} (w-p)\|_{K} \\
& \leq \| w-p \|_{K} + \widehat{C}_{I_{S_{\T}},1} \| w - p \|_{\omega_K} \\
&\leq (\widehat{C}_{I_{S_{\T}},1}+1) \| w - p \|_{\omega_K}.
\end{align*}
Using the Bramble--Hilbert lemma (\Cref{lemma:BH}) gives
\begin{equation*}
\| w - I_{S_{\T}}w\|_{K} \leq 
(\widehat{C}_{I_{S_{\T}},1}+1) C_{\BH}(\T)
h_K \| \nabla w \|_{\omega_K}
\end{equation*}
or
\begin{equation*}
\| w - I_{S_{\T}}w\|_{K} \leq 
(\widehat{C}_{I_{S_{\T}},1}+1) C_{\BH}(\T) 
h_K^2 |  w |_{H^2(\omega_K)}.
\end{equation*}

It remains to verify the inequality \eqref{eq:quasi-interpolation_to_S_H1_stab}.
Using the fact that $I_{S_{\T}}$ reproduces constants, we have, for arbitrary $c \in \R$, 
\begin{align*}
\| \nabla I_{S_{\T}}w\|^2_{K} & = \| \nabla I_{S_{\T}}(w-c)\|^2_{K}= \int_K \left \lvert \sum_{z\in \N_K}  \left( \int_{K_z} (w-c) \psi_z \right) \nabla \phi_z \right \rvert^2  \\
& \leq (\# \N_K )   \sum_{z\in \N_K}  \| \nabla \phi_z \|^2_{L^{\infty}(K)} \int_K   \left \lvert  \int_{K_z} (w-c) \psi_z   \right \rvert ^2 \\
& \leq (d+1)   \sum_{z\in \N_K}  \| \nabla \phi_z \|^2_{L^{\infty}(K)}   \left \lvert  \int_{K_z} (w-c) \psi_z   \right \rvert ^2 \meas{K}\\
& \leq  (d+1)  C^2_{\phi} \rho^{-2}_K \meas{K} \sum_{z\in \N_K} \left \lvert  \int_{K_z} (w-c) \psi_z   \right \rvert ^2 ,
\end{align*}
where we also used~\eqref{eq:phi_L_infty_bound}. Then, from~\eqref{eq:quasi}, we get
\begin{align*}
\| \nabla I_{S_{\T}} w\|^2_{K} & \leq  (d+1) C^2_{\phi} \rho^{-2}_K \meas{K}  C^2_{\psi} \max_{z \in \N_K} \meas{K_z}^{-1} \| w-c \|^2_{\omega_K}.
\end{align*}
Using the Bramble--Hilbert lemma (\Cref{lemma:BH}) and rearranging yields
\begin{align*}
\| \nabla I_{S_{\T}} w \|^2_{K} & \leq  
(d+1)  C^2_{\phi} C^2_{\psi} \big( C_{\BH}(\T) \big)^2
  \frac{\meas{K}}{\min_{z \in \K_K} \meas{K_z}}
\cdot
 \frac{h^2_K}{\rho^2_K}  \| \nabla w \|^2_{\omega_K}.
\end{align*}
\end{proof}

For the interpolation operator~$I_{V_{\T}}$, we can derive bounds analogous to those of \Cref{thm:quasi-interpolation_to_S_estimates}.
For the ``inner'' elements, i.e.,~the elements $K\in\T$ such that patch~${\omega}_K$ does not intersect with the boundary of $\Omega$, i.e., $\overline{\omega_K} \cap \partial \Omega = \emptyset$, the forms of the bounds and their proofs are analogous to \Cref{thm:quasi-interpolation_to_S_estimates}, because~$I_{V_{\T}}$ also reproduces constants on~${\omega}_K$. For the elements whose patch intersects with the boundary of $\Omega$, one cannot use this property and the Bramble--Hilbert lemma (\Cref{lemma:BH}) must be replaced by Friedrich's inequality (\Cref{lemma:Friedrichs}) in the proofs.

\begin{thm}\label{thm:quasi-interpolation_to_V_estimates}
There exist positive constants $\widehat{C}_{I_{V_{\T}},\ell}$, $\ell = 1,2,4$, depending only on $d$ and $\gamma_{\T}$ such that for all elements $K\in\T$, 
\begin{equation}
 \label{eq:quasi-interpolation_to_V_L2_stab}
 \| I_{V_{\T}} w \|_{K}  \leq \widehat{C}_{I_{V_{\T}},1} \| w \|_{\omega_K}, \quad \forall w \in L^2(\omega_K), \\
 \end{equation}
and for all $w \in H^1(\omega_K)$ if $\overline{\omega_K} \cap \partial \Omega = \emptyset$, or for all $w \in H^1(\omega_K)\cap H^1_0(\Omega)$ otherwise,
\begin{align}
\label{eq:quasi-interpolation_to_V_L2_approx_2}
\| w - I_{V_{\T}} w \|_{K} & \leq \widehat{C}_{I_{V_\T},2} h_K \| \nabla w \|_{\omega_K},  \\
\label{eq:quasi-interpolation_to_V_H1_stab_2}
\| \nabla I_{V_{\T}} w \|_{K} & \leq \widehat{C}_{I_{V_\T},4} \| \nabla w \|_{\omega_K}.
\end{align}
\end{thm}

For the local interpolation error over the faces, we have the following bound.
\begin{thm} \label{thm:quasi-interpolation-edge-estimates}
There exists a positive constant~$\widehat{C}_{I_{V_{\T}},5}$ depending only on $d$ and $\gamma_{\T}$ such that for all elements $K\in\T$,
\begin{equation}
\| w - I_{V_{\T}} w \|^2_{\partial K} \leq \widehat{C}_{I_{V_{\T}},5} h_K \| \nabla w \|^2_{\omega_K}. 
\end{equation}
\end{thm}
\begin{proof}
Using the trace inequality (\Cref{lemma:trace_ineq}) and the properties of $I_{V_{\T}}$ from \Cref{thm:quasi-interpolation_to_V_estimates} yields
\begin{align*}
\| w - I_{V_{\T}} w \|_{\partial K} 
& \leq C_{\TR}(\T)[ h^{-1}_{K} \| w - I_{V_{\T}} w \|^2_{K}  +  h_{K} \|\nabla( w - I_{V_{\T}} w ) \|^2_{K} ] \\
& \leq C_{\TR}(\T) \left[ h^{-1}_{K}  \| w - I_{V_{\T}} w \|^2_{K} + h_{K} \cdot 2 \cdot  \left( \|\nabla w \|^2_{K}  + \| \nabla I_{V_{\T}} w  \|^2_{K}  \right)  \right] \\
& \leq C_{\TR}(\T) \left[ h^{-1}_{K}  \left(\widehat{C}_{I_{V_{\T}},2} \right)^2 h^2_{K} \| \nabla w \|^2_{\omega_{K}} +  h_{K}\cdot 2\left( 1+\left( \widehat{C}_{I_{V_{\T}},4} \right)^2  \right)  \| \nabla w \|^2_{\omega_{K}} \right].
\end{align*}
\end{proof}

\subsection{Global estimates}
\label{sec:interpolation_globalestims}

We now state global variants of estimates for quasi-interpolants and interpolation errors.

For any $K\in \T$, let $\Covrlp(K)$ denote the number of patches this element is contained in, i.e.,
\begin{equation*}
    \Covrlp(K) = \# \left \lbrace K' \in \T ; K\subset \omega_{K'} \right \rbrace. 
\end{equation*}
The constant $\Covrlp(K)$ depends only on the geometry of the mesh~$\T$, i.e., $d$ and the shape regularity~$\gamma_\T$.

\begin{thm}\label{thm:quasi-interpolation-to-S-global-estimates}
There exist positive constants $C_{I_{S_{\T}},\ell}$, $\ell = 1,2,4$, depending only on $d$ and $\gamma_{\T}$ such that
\begin{align}
\| I_{S_{\T}} w \| & \leq C_{I_{S_{\T}},1} \| w \| && \forall w \in L^2(\Omega), \\
\left( \sum_{K\in \T}   h^{-2}_K  \|  w - I_{S_{\T}}w  \|^2_{K}  \right)^{\frac{1}{2}}
=\| h^{-1}_{\T} ( w - I_{S_{\T}} w ) \|
& \leq C_{I_{S_{\T}},2} \| \nabla w \| && \forall w \in H^1(\Omega), \\
\| \nabla (I_{S_{\T}} w )\| & \leq C_{I_{S_{\T}},4} \| \nabla w \| && \forall w \in H^1(\Omega).
\end{align}
\end{thm}

\begin{proof}
Using \Cref{thm:quasi-interpolation_to_S_estimates},
\begin{align*}
\| I_{S_{\T}} w \|^2   = \sum_{K\in \T} \| I_{S_{\T}} w \|^2_K
&\leq   \sum_{K\in \T}  \left( \widehat{C}_{I_{S_{\T}},1}\right)^2 \| w\|^2_{\omega_K} 
\leq \left( \widehat{C}_{I_{S_{\T}},1}\right)^2 \sum_{K\in \T} \Covrlp (K) \| w\|^2_{K} \\
&\leq \left( \widehat{C}_{I_{S_{\T}},1}\right)^2 \max_{K \in \T} \Covrlp (K) \sum_{K \in \T} \| w\|^2_{K}.
\end{align*}
The proofs of the other three inequalities are analogous.
\end{proof}

\begin{thm}\label{thm:quasi-interpolation-to_V-global-estimates}
There exist positive constants $C_{I_{V_{\T}},\ell}$, $\ell = 1,2,4,5$, depending only on $d$ and the shape-regularity constant $\gamma_{\T}$ such that
\begin{align}
\| I_{V_{\T}} w \| & \leq C_{I_{V_{\T}},1} \| w \| && \forall w \in L^2(\Omega), \\
\label{eq:glob_interp_Vj_2}
\left( \sum_{K\in \T}   h^{-2}_K  \|  w - I_{V_{\T}}w  \|^2_{K}  \right)^{\frac{1}{2}}
=\| h^{-1}_\T ( w - I_{V_{\T}} w ) \|
& \leq C_{I_{V_{\T}},2} \| \nabla w \| && \forall w \in H^1_0(\Omega), \\
\label{eq:quasi-interpolation-to_V-global_grad}
\| \nabla (I_{V_{\T}} w )\| & \leq C_{I_{V_{\T}},4} \| \nabla w \| && \forall w \in H^1_0(\Omega), \\
\left( \sum_{K\in \T} h^{-1}_K \| w - I_{V_{\T}} w \|^2_{\partial K} \right)^{\frac{1}{2}} & \leq C_{I_{V_{\T}},5} \| \nabla w \| && \forall w \in H^1_0(\Omega).
\end{align}
\end{thm}

Let us now consider the mesh hierarchy as in \Cref{sec:multilevel}. Since the constants  $C_{I_{S_{j}},\ell}$ and $C_{I_{V_{j}},\ell}$ depend only on $d$ and $\gamma_j$, they can be bounded by constants $C_{I_{S},\ell}$ and $C_{I_{V},\ell}$ depending only on $d$ and the shape regularity $\gamma_0$ of the initial mesh $\T_0$.

Finally, we bound the difference of quasi-interpolates on two consecutive levels. 

\begin{thm}\label{lemma:difference-quasi-interpolation}
There exists a constant $C_{I,\mathrm{2lvl}}>0$ depending only on $d$ and $\gamma_{0}$ such that for all $j\geq1$ and all $w \in H^1_0(\Omega)$,
\begin{equation}
 \| h^{-1}_{j} ( I_{V_{j}} w - I_{V_{j-1}} w ) \| \leq  C_{I,\mathrm{2lvl}}\| \nabla w\|.
\end{equation}
\end{thm}

\begin{proof}
Using the fact that $h^{-1}_{j}= 2  h^{-1}_{j-1}$ and the estimate \eqref{eq:glob_interp_Vj_2} from \Cref{thm:quasi-interpolation-to_V-global-estimates},
\begin{align*}
\| h^{-1}_j (I_{V_{j}} w - I_{V_{j-1}} w)  \| & \leq  \|h^{-1}_j ( w - I_{V_{j}} w)  \|  +  \|h^{-1}_j ( w - I_{V_{j-1}} w ) \|   \\
& = \|h^{-1}_j ( w - I_{V_{j}} w)  \|  +  2\|h^{-1}_{j-1} ( w - I_{V_{j-1}} w ) \|   \\
&\leq  (C_{I_{V},2} + 2  C_{I_{V},2}) \| \nabla w \|.
\end{align*}
Taking $ C_{I,\mathrm{2lvl}}$ as $C_{I,\mathrm{2lvl}} = C_{I_{V},2} + 2  C_{I_{V},2}$ finishes the proof.
\end{proof}


\section{Stable splitting}
\label{sec:splitting}
This section presents several results on splitting (decomposing) a $H^1_0(\Omega)$-function or a piecewise polynomial function into a sum of piecewise polynomial functions. 
Let a sequence of uniformly refined meshes $\T_j$, $j=0,1, \ldots$ as in \Cref{sec:multilevel} be given.

\subsection{Splitting of $H^1_0(\Omega)$ into subspaces of piecewise linear functions}
\label{subsec:splitting_infinite_dim}
To make the text easier to follow we first state the main result of this section and subsequently provide auxiliary results and proofs. We will show that any function  $w  \in  H^1_0(\Omega)$ can be uniquely decomposed using the quasi-interpolation operators $I_{V_j}$, $j\in \mathbb{N}_0$, as
\begin{equation*}
    w = I_{V_{0}} w + \sum^{+ \infty }_{j=1} (I_{V_{j}} - I_{V_{j-1}}) w;
\end{equation*}
the convergence of the sum is understood in the space $H^1_0(\Omega)$ with the norm $\|\nabla \cdot \|$. This decomposition is stable, meaning that there exist positive constants $c_{S,I_V},C_{S,I_V}$ such that for all $w \in H^1_0(\Omega)$,
\begin{equation}\label{eq:I_V_decompositon_begining}
c_{S,I_V} \| \nabla w\|^2\leq 
\| \nabla I_{V_{0}} w  \|^2 + 
 \sum^{+ \infty}_{j=1}  \| h^{-1}_{j} ( I_{V_{j}} w - I_{V_{j-1}} w ) \| ^2 \leq C_{S,I_V} \| \nabla w\|^2.
\end{equation}
We will also show that the splitting of the space $H^1_0(\Omega)$
into subspaces $V_j$, $j\in \mathbb{N}_0$, is stable in the sense that there exist positive constants $c_S,C_{S}$ such that for all \mbox{$w \in H^1_0(\Omega)$},
\begin{equation}\label{eq:V_decompositon_begining}
c_{S} \| \nabla w \|^2 \leq 
\inf_{  w_j \in V_j ;\ w=\sum^{+\infty}_{j=0} w_j } 
\| \nabla  w_0 \|^2  +  \sum^{ + \infty}_{j=1}   \| h^{-1}_{j} w_j \| ^2 \leq C_{S} \| \nabla w \|^2;
\end{equation}
the infimum is taken over all $(H^1_0(\Omega), \|\nabla \cdot \|) $-convergent decompositions.

We will show that the stability constants $c_{S,I_V},C_{S,I_V}$ and $c_S,C_{S}$ depend \emph{only} on $d$ and the shape regularity $\gamma_0$ of the initial mesh. In particular, the constants do not depend on the quasi-uniformity of the initial mesh or the ratio $h_{\Omega} / \min_{K\in \T_0}  h_K$. This result is important when considering settings where the problem associated with the coarsest level is difficult to solve and in practice can only be solved approximately.

Variants of these results can be found, e.g., in \cite{Oswald1994a,Ruede1993,Dahmen1992,Dahmen1997, Bornemann1993} and references therein. Our form is, however, to the best of our knowledge, not presented in the literature.
The results in \cite{Oswald1994a,Ruede1993,Dahmen1992,Dahmen1997} are derived under the assumption that the initial mesh is quasi-uniform, and the authors do not track the dependence of the constants on $h_{\Omega} / \min_{K\in \T_0}  h_K$.
The results of \cite{Bornemann1993} are derived without the assumption on the quasi-uniformity of the initial mesh. The authors however consider only the splitting of piecewise linear functions.

We combine the approaches from \cite{Oswald1994a} and \cite{Bornemann1993}. We first focus on  showing the upper bound from \eqref{eq:I_V_decompositon_begining}, then continue with the lower bound, and later generalize it to show \eqref{eq:V_decompositon_begining}.

First, consider the $K$-functional in analogy to \cite[Section~7, eq.~(7.4)]{Bornemann1993}. For $\omega \subset \R^d$, $w \in L^2(\omega)$, it is defined as 
\begin{equation}
K(t,w,\omega) = \inf_{g\in H^2(\omega)} \left\lbrace \| w-g \|^2_{L^2(\omega)} + t^2 |g|^2_{H^2(\omega)}  \right\rbrace^{\frac{1}{2}}, \quad t>0.
\end{equation}
\begin{lemma}
\label{lemma:K_functionals_Rd}
There exists a constant $C>0$ such that for all $w\in H^1(\mathbb{R}^d)$ that have compact support in $\R^d$, it holds that 
\begin{equation}
\sum^{ + \infty}_{j=0} 2^{2j} K(2^{-2j},w,\mathbb{R}^d)^2 \leq C \| \nabla w\|^2_{L^2(\mathbb{R}^d)}.
\end{equation}
\end{lemma}
\begin{proof}
A brief proof for $d=2$ is given in~\cite[Lemma~7.3]{Bornemann1993}. We present its key part in more detail and for $d=2,3$.
We will show that the $K$-functional can be expressed in terms of the Fourier transform (here denoted by $F[\cdot]$) as
\begin{equation}
\label{eq:Kfunceq}
K(t,w,\R^d)^2 = \frac{1}{(2\pi)^{\frac{d}{2}}}\int_{\R^d} \frac{t^2 |\xi|^4}{1+t^2 |\xi|^4} \big|F[w](\xi)\big|^2 d\xi.
\end{equation}

For $w\in H^1(\R^d)$, $g \in H^2(\R^d)$, using the properties of the Fourier transform,
\begin{multline}
\label{eq:Kfuncfourier}
\| w - g \|^2_{L^2(\R^d)} + t^2 |g|^2_{H^2(\R^d)}  \\
=\frac{1}{(2\pi)^{\frac{d}{2}}}\int_{\R^d} \frac{t^2 |\xi|^4}{1+t^2 |\xi|^4} \big|F[w](\xi)\big|^2 + (1+t^2 |\xi|^4) \left( F[g](\xi) - \frac{F[w](\xi)}{1+t^2|\xi|^4} \right)^2 d\xi. 
\end{multline}
By simple manipulations, one can show that the minimum is attained for
\begin{equation}
    \widetilde{g}(x) = w(x) \ast F^{-1}\left[ \frac{1}{1+t^2 |\xi|^4} \right](x),
\end{equation}
and it remains to show that $\widetilde{g} \in H^2(\R^d)$.
First, note that
\[
    \int_{\R^d} (1 + |\xi|)^2 \big( \frac{1}{1+t^2 |\xi|^4} \big)^2 < \infty,
\]
and therefore, due to the characterization of Sobolev spaces using Fourier transformations (see, e.g., \cite[Section~3.1.1]{Oswald1994a}), $F^{-1}[ (1+t^2 |\xi|^4)^{-1} ](x) \in H^2(\R^d)$. Then use Young's inequality for convolution (recall that by assumption, $w$ is compactly supported and therefore $w \in L^1(\R^d)$) and the fact that $\partial / \partial \xi_i (f * h) = ( \partial f/ \partial \xi_i * h)$ to show that the $H^2$-norm of $\widetilde{g}$ is bounded.

The equality~\eqref{eq:Kfunceq} then follows by plugging in the expression for~$\widetilde{g}$ into~\eqref{eq:Kfuncfourier} and performing algebraic manipulations. The rest of the proof of the lemma follows as in~\cite[Lemma~7.3]{Bornemann1993}. 
\end{proof}
\begin{lemma}
\label{lemma:K_functionals}
Let $\omega\subset \mathbb{R}^d $ be a domain with a Lipschitz-continuous boundary. There exists a constant $C_{\alpha}(\omega)>0$ depending on the shape of $\omega$ such that for all $w\in H^1(\omega)$, 
\begin{equation}
\sum^{ + \infty}_{j=0} 2^{2j} K(2^{-2j},w,\omega)^2 \leq C_{\alpha}(\omega) \| \nabla w\|^2_{L^2(\omega)}.
\end{equation}
\end{lemma}
\begin{proof}
The proof for $d=2$ is given in \cite[Lemma~7.4]{Bornemann1993}. It is based on the use of an extension operator and \Cref{lemma:K_functionals_Rd}. For the three-dimensional case, the proof is analogous as \Cref{lemma:K_functionals_Rd} is valid also for $d=3$.
\end{proof}

\begin{lemma}\label{lemma:stables_splitting_quasi_to_S}
There exists a constant $C_{\beta}>0$ depending only on $d$ and $\gamma_{0}$ such that for all $K\in \T_0$ and all $w\in H^1_0(\Omega)$,
\begin{equation}
h^{-2}_K \sum^{+ \infty}_{j=0} 2^{2j}  \|  w - I_{S_{j}}w\|^2_K \leq C_{\beta} \| \nabla w \|^2_{\omega_{K}}.
\end{equation}
\end{lemma}
\begin{proof}
The steps in the proof are inspired by the development in \cite[Section~2.3]{Oswald1994a} and \cite[Section~7]{Bornemann1993}.

We will use a scaling argument to consider an element $\widetilde{K}$ with $h_{\widetilde{K}}=1$. This is done by a transformation $x = h_K \widetilde{x}$, where $x\in K$, $\widetilde{x}\in \widetilde{K}$.
We denote $\widetilde{f}(\widetilde{x}):= f(x)$ for any function~$f$ defined on~$\omega_K$. Then
\begin{equation}\label{eq:lemma:stables_splitting_quasi_to_S:scaling}
\|  w - I_{S_{j}}w\|^2_K = h^d_K  \|  \widetilde{w} - \widetilde{I_{S_{j}}w}\|^2_{\widetilde{K}}.
\end{equation}
From the definition of the interpolation operator one can write $\widetilde{I_{S_{j}}w} = \widetilde{I_{S_{j}}} \widetilde{w}$. In words, one can either consider the transformation of the interpolant~$I_{S_{j}}w$ or transform the function~$w$ to the element~$\widetilde{K}$ first and then consider the quasi-interpolation~$\widetilde{I_{S_{j}}}$ associated with the transformed mesh.

We will show that there exists a constant $C_{\delta}>0$ depending only on $d$ and $\gamma_0$ such that
\begin{equation}\label{eq:lemma:stables_splitting_quasi_to_S:scaling_K-fun}
\|  \widetilde{w} - \widetilde{I_{S_{j}}}\widetilde{w}\|^2_{\widetilde{K}}
\leq  C_{\delta}  \cdot \big( K(2^{-2j},\widetilde{w},\omega_{\widetilde{K}}) \big)^2.
\end{equation}
Let $\widetilde{g}\in H^2(\omega_{\widetilde{K}})$. Then
\begin{equation} \label{eq:lemma:stables_splitting_quasi_to_S:local_ineq_0}
\|  \widetilde{w} - \widetilde{I_{S_{j}}} \widetilde{w}\|_{\widetilde{K}}
\leq \| \widetilde{w} - \widetilde{g} \|_{\widetilde{K}}  + \| \widetilde{g} - \widetilde{I_{S_{j}}} \widetilde{g} \|_{\widetilde{K}} + \|   \widetilde{I_{S_{j}}} ( \widetilde{g} - \widetilde{w}) \|_{\widetilde{K}}.
\end{equation}
Let $\widetilde{K}_j \in \widetilde{\T}_j$ such that $\widetilde{K}_j\subset \widetilde{K}$. Then (thanks to the uniform refinement and $h_{\widetilde{K}} = 1$, $h_{\widetilde{K}_j}=2^{-j}$) from \Cref{thm:quasi-interpolation_to_S_estimates} (inequalities \eqref{eq:quasi-interpolation_to_S_L2_stab} and \eqref{eq:quasi-interpolation_to_S_L2_approx_H2}),
\begin{align}\label{eq:lemma:stables_splitting_quasi_to_S:local_eq_1}
\| \widetilde{I_{S_{j}}} (\widetilde{g} - \widetilde{w}) \|_{\widetilde{K}_j} &\leq \widehat{C}_{\widetilde{I_{S_j}},1}   \| \widetilde{g} - \widetilde{w} \|_{\omega_{\widetilde{K}_j}}, \\ \label{eq:lemma:stables_splitting_quasi_to_S:local_eq_2}
\| \widetilde{g} - \widetilde{I_{S_{j}}} \widetilde{g} \|_{\widetilde{K}_j} &\leq \widehat{C}_{\widetilde{I_{S_j}},3} 2^{-2j} | \widetilde{g}|_{H^2(\omega_{\widetilde{K}_j})}.
\end{align}
Define
\begin{equation*}
    U(\widetilde{K},j) = \left\lbrace \cup \omega_{\widetilde{K}_j} ; \widetilde{K}_j \in \widetilde{\T}_j, \widetilde{K}_j \subset \widetilde{K}  \right \rbrace.
\end{equation*}
The term on the right hand side of \eqref{eq:lemma:stables_splitting_quasi_to_S:local_eq_1} can be bounded as
\begin{align} \nonumber
\| \widetilde{I_{S_{j}}} (\widetilde{g} - \widetilde{w}) \|_{\widetilde{K}} 
&= \sum_{\widetilde{K}_j \in \widetilde{\T}_j, \widetilde{K}_j \subset \widetilde{K} } \| \widetilde{I_{S_{j}}} (\widetilde{g} - \widetilde{w}) \|_{\widetilde{K}_j} \\ 
\nonumber
& \leq \sum_{\widetilde{K}_j \in \widetilde{\T}_j, \widetilde{K}_j \subset \widetilde{K}}
\widehat{C}_{\widetilde{I_{S_j}},1}  \| \widetilde{g} - \widetilde{w} \|_{\omega_{\widetilde{K}_j}} \\ 
\nonumber
&\leq  \widehat{C}_{\widetilde{I_{S_j}},1} \max_{\widetilde{K}_j \in \widetilde{\T}_j; \widetilde{K}_j \in U(\widetilde{K},j) } \Covrlp (\widetilde{K}_j)    \| \widetilde{g} - \widetilde{w} \|_{ U(\widetilde{K},j) } \\
\nonumber
& \leq \widehat{C}_{\widetilde{I_{S_j}},1} \max_{\widetilde{K}_j \in \widetilde{\T}_j; \widetilde{K}_j \in U(\widetilde{K},j) }  \Covrlp (\widetilde{K}_j) \| \widetilde{g} - \widetilde{w} \|_{\omega_{\widetilde{K}}} \\
\label{eq:lemma:stables_splitting_quasi_to_S:local_eq_3}
& \leq C_{I_{S},1} \| \widetilde{g} - \widetilde{w} \|_{\omega_{\widetilde{K}}},
\end{align} 
where the last inequality follows from the fact that $\widehat{C}_{\widetilde{I_{S_j}},1} = \widehat{C}_{{I_{S_j}},1}$ (scaling does not change the geometry and shape regularity) and from the definition of~$C_{I_{S},1}$. The term on the right hand side of \eqref{eq:lemma:stables_splitting_quasi_to_S:local_eq_2} can be bounded as
\begin{align}
\nonumber
\| \widetilde{g} - \widetilde{I_{S_{j}}} \widetilde{g} \|^2_{\widetilde{K}} 
&= \sum_{\widetilde{K}_j \in \widetilde{\T}_j, \widetilde{K}_j \subset \widetilde{K} }
 \| \widetilde{g} - \widetilde{I_{S_{j}}} \widetilde{g} \|^2_{\widetilde{K}_j}  \\
\nonumber
& \leq \sum_{\widetilde{K}_j \in \widetilde{\T}_j, \widetilde{K}_j \subset \widetilde{K} }  \big( \widehat{C}_{\widetilde{I_{S_j}},3} 2^{-2j} | \widetilde{g}|_{\omega_{\widetilde{K}_j} } \big)^2 \\
\nonumber
& \leq   \big( \widehat{C}_{\widetilde{I_{S_j}},3} \big)^2  \max_{\widetilde{K}_j \in \widetilde{\T}_j, \widetilde{K}_j \in U(\widetilde{K},j)} \Covrlp (\widetilde{K}_j) 
 2^{-4j} | \widetilde{g}|^2_{H^2( U(\widetilde{K},j))} \\
\nonumber
&\leq  \big( \widehat{C}_{\widetilde{I_{S_j}},3} \big)^2 \max_{\widetilde{K}_j \in \widetilde{\T}_j, \widetilde{K}_j \subset \widetilde{K}} \Covrlp (\widetilde{K}_j)  2^{-4j} | \widetilde{g}|^2_{H^2( \omega_{\widetilde{K}})} \\
\label{eq:lemma:stables_splitting_quasi_to_S:local_eq_4}
&\leq \left( C_{I_{S},3} \cdot 2^{-2j} | \widetilde{g}|_{H^2( \omega_{\widetilde{K}})} \right ) ^2
.
\end{align}
Combining \eqref{eq:lemma:stables_splitting_quasi_to_S:local_ineq_0} - \eqref{eq:lemma:stables_splitting_quasi_to_S:local_eq_4} yields
\begin{equation*}
\|  \widetilde{w} - \widetilde{I_{S_{j}}}\widetilde{w}\|_{\widetilde{K}}
\leq  \max_{\ell =1,3} C_{I_{S},\ell}
\left( \| \widetilde{g} - \widetilde{w} \|_{\omega_{\widetilde{K}}} +  2^{-2j} | \widetilde{g}|_{H^2( \omega_{\widetilde{K}})}  \right).
\end{equation*}
From the definition of the $K$-functional,
\begin{align*}
\|  \widetilde{w} - \widetilde{I_{S_{j}}}\widetilde{w}\|^2_{\widetilde{K}}
& \leq \max_{\ell =1,3}  C_{I_{S},\ell}^2 \left( \| \widetilde{g} - \widetilde{w} \|_{\omega_{\widetilde{K}}} +  2^{-2j} | \widetilde{g}|_{H^2( \omega_{\widetilde{K}})}  \right)^2\\
& \leq  2 \cdot \max_{\ell =1,3}  C_{I_{S},\ell}^2 \left( \| \widetilde{g} - \widetilde{w} \|^2_{\omega_{\widetilde{K}}} +  \left(2^{-2j} \right)^2 | \widetilde{g}|^2_{H^2( \omega_{\widetilde{K}})} \right) \\
&=  2 \cdot \max_{\ell =1,3}  C_{I_{S},\ell}^2 \cdot \big( K(2^{-2j},\widetilde{w},\omega_{\widetilde{K}}) \big)^2.
\end{align*}
In the notation introduced above, $C_{\delta} = 2 \cdot \max_{\ell =1,3}  C_{I_{S},\ell}^2$.

Using \eqref{eq:lemma:stables_splitting_quasi_to_S:scaling}, \eqref{eq:lemma:stables_splitting_quasi_to_S:scaling_K-fun} and \Cref{lemma:K_functionals} yields
\begin{align*}
\sum^{+ \infty}_{j=0} 2^{2j} \|  w - I_{S_{j}}w\|^2_K 
&\leq  h^d_K  C_{\delta}
 \sum^{+ \infty}_{j=0} 2^{2j} \big( K(2^{-2j},\widetilde{w},\omega_{\widetilde{K}}) \big)^2 \\
&\leq h^d_K C_{\delta} C_{\alpha}(\omega_{\widetilde{K}})  \| \nabla \widetilde{w} \|^2_{\omega_{\widetilde{K}}} .
\end{align*}
Re-scaling back to~$K$,
\begin{equation*}
\sum^{+ \infty}_{j=0} 2^{2j} \|  w - I_{S_{j}}w\|^2_K \leq h^{d}_K   C_{\delta} C_{\alpha }(\omega_{\widetilde{K}}) h^{2}_K  h^{-d}_K \| \nabla w \|^2_{\omega_{K}}.
\end{equation*}
Finally, note that the shape of~$\omega_{\widetilde{K}}$ depends on the shape regularity of the initial mesh and therefore $C_{\alpha }(\omega_{\widetilde{K}})$ can be bounded, for all $K \in \T_0$, by a constant~$C_{\alpha }$ depending only on~$d$ and~$\gamma_0$.
\end{proof}

\begin{thm}\label{thm:stables_splitting_quasi_to_S}
There exists a constant $C_{S,I_S}>0$ depending only on $d$ and $\gamma_0$, such that for all $w \in H^1_0(\Omega)$,
\begin{equation}
\label{eq:stables_splitting_quasi_to_S}
\| \nabla I_{S_0} w \|^2 + \sum^{+ \infty}_{j=1} \|  h^{-1}_j (I_{S_j} w - I_{S_{j-1}} w ) \| ^2 
\leq C_{S,I_S} \|\nabla w \|^2  .
\end{equation}
\end{thm}
\begin{proof}
From \Cref{thm:quasi-interpolation-to-S-global-estimates},
\begin{equation*}
    \| \nabla I_{S_0} w \|^2 \leq C_{I_{S_0},4}^2 \|\nabla w \|^2.
\end{equation*}
For the rest of the sum,
\begin{align*}
\sum^{+ \infty}_{j=1} \|  h^{-1}_j (I_{S_j} - I_{S_{j-1}} ) w\|  ^2  
& = \sum^{+ \infty}_{j=1} \|  h^{-1}_j ( I_{S_{j}} w  - w + w - I_{S_{j-1}} w ) \|  ^2    \\
& \leq 2 \sum^{+ \infty}_{j=1} \left(  \|  h^{-1}_j ( w - I_{S_{j}} w ) \|  ^2 + \|   h^{-1}_j ( w - I_{S_{j-1}} w ) \|  ^2 \right)    \\
& \leq 2 \sum^{+ \infty}_{j=1} \left( \|   h^{-1}_j ( w - I_{S_{j}} w ) \|  ^2 + 4 \|   h^{-1}_{j-1} ( w - I_{S_{j-1}} w ) \|  ^2 \right)    \\
& \leq 2 \left( \sum^{+ \infty}_{j=1} \|   h^{-1}_j ( w - I_{S_{j}} w ) \|  ^2 + 4 \sum^{+ \infty}_{j=0} \|   h^{-1}_{j} ( w - I_{S_{j}} w ) \|  ^2 \right)    \\
& \leq 2 \cdot 5  \sum^{+ \infty}_{j=0} \|   h^{-1}_j ( w - I_{S_{j}}w ) \|  ^2     \\
& = 10  \sum_{K \in \T_0} \sum^{+ \infty}_{j=0} 2^{2j} h^{-2}_K \|  w - I_{S_{j}}w\|^2_K \\
& \leq 10  \sum_{K \in \T_0} C_{\beta } \|  \nabla w  \|  ^2_{\omega_K}
 \leq 10 \cdot C_{\beta } \cdot \max_{K \in \T_0} \Covrlp(K)  \| \nabla w \|^2, 
\end{align*}
where we have used \Cref{lemma:stables_splitting_quasi_to_S} in the second to last inequality.
\end{proof}

\begin{thm}\label{lemma:decomposition-using-quasi-interpol-to-V-upper-bound}
There exists a constant $C_{S,I_V}>0$ depending only on $d$ and  $\gamma_0$ such that for all $w \in H^1_0(\Omega)$,
\begin{equation}
\| \nabla I_{V_{0}} w  \|^2 + 
 \sum^{+ \infty}_{j=1}   \| h^{-1}_{j} ( I_{V_{j}} w - I_{V_{j-1}} w )  \| ^2 \leq C_{S,I_V} \| \nabla w\|^2.
\end{equation}
\end{thm}

\begin{proof}
The key steps of the following proof of the upper bound were provided to us by professor P.~Oswald in personal communications.
From \Cref{thm:quasi-interpolation-to_V-global-estimates},
\begin{equation*}
    \| \nabla I_{V_0} w \|^2 \leq C_{I_{V_0},4}^2 \|\nabla w \|^2.
\end{equation*}
To bound $\sum^{+ \infty}_{j=1}   \left \lVert h^{-1}_{j} ( I_{V_{j}} w - I_{V_{j-1}} w )  \right \rVert ^2$, consider the sequence $w_j=(I_{S_j} - I_{S_{j-1}})w$, $j\in \mathbb{N}$.
Let $K \in \T_0$. We will first show that there exists a constant $C_{\epsilon}>0$ depending only on $d$ and $\gamma_0$ such that 
\begin{equation}\label{eq:local_est_sum_V_j}
\sum^{+\infty}_{j=1} 2^{2j} 
 \| (I_{V_{j}}-I_{V_{j-1}})w \|^2_K 
\leq   C_{\epsilon}  \sum^{+ \infty}_{i=1} 2^{2i}  \| w_i \|^2_{\omega_K}.
\end{equation}

Since $I_{V_{j}}$ are projections onto $V_j$, it holds that
\begin{equation}
(I_{V_{j}}-I_{V_{j-1}})w_i = 0,\quad j>i.
\end{equation}
Then for all $j \geq 1$, using the Cauchy--Schwarz inequality for sums and \Cref{thm:quasi-interpolation_to_V_estimates},
\begin{align*}
\| (I_{V_{j}}-I_{V_{j-1}})w \|^2_{K} & = \int_{K}  \left( \sum^{+ \infty}_{i=j} (I_{V_{j}}-I_{V_{j-1}}) w_i \right)^2 \\
&\leq \int_{K}  \left( \sum^{+ \infty}_{i=j} 2^{ -i } \right) \left( \sum^{+ \infty}_{i=j} 2^{ i } \left( (I_{V_{j}}-I_{V_{j-1}}) w_i \right)^2 \right) \\
&\leq 2 \cdot2^{-j} \sum^{+ \infty}_{i=j} 2^{ i} \| (I_{V_{j}} - I_{V_{j-1}})w_i \|^2_K \\
&\leq 2 \cdot2^{-j} \sum^{+ \infty}_{i=j} 2^{ i} \cdot 2\left( \|  I_{V_{j}} w_i \|^2_K + \| I_{V_{j-1}} w_i \|^2_K \right) \\
&\leq 2 \cdot  2^{ -j } \sum^{+ \infty}_{i=j}  2^{i}  \cdot 2 \cdot 2\cdot \left(\widehat{C}_{I_{V},1}\right)^2  \| w_i \|^2_{\omega_K}.
\end{align*}
Consequently, 
\begin{equation*}
\sum^{+\infty}_{j=1} 2^{2j} 
 \| (I_{V_{j}}-I_{V_{j-1}})w \|^2_K 
\leq  \underbrace{ 8 \cdot \left(\widehat{C}_{I_{V_{\T}},1}\right)^2}_{C_{\epsilon}} \sum^{+ \infty}_{j=1} 2^{j}  \sum^{+ \infty}_{i=j}  2^{ i} \| w_i \|^2_{\omega_K}.
\end{equation*}
Changing the order of summation,
\begin{align*}
\sum^{+ \infty}_{j=1} 2^{j}   \sum^{+ \infty}_{i=j}  2^{ i} \| w_i \|^2_{\omega_K}& =    \sum^{+ \infty}_{i=1} 2^{i}  \| w_i \|^2_{\omega_K} \sum^{i-1}_{j=1} 2^{j} 
\leq    \sum^{+ \infty}_{i=1} 2^{2i}  \| w_i \|^2_{\omega_K}.
\end{align*}

For the sum $\sum^{+ \infty}_{j=1} 
 \| h^{-1}_j(I_{V_{j}}-I_{V_{j-1}})w \| ^2 $, using \eqref{eq:local_est_sum_V_j},
\begin{align*}
\sum^{+ \infty}_{j=1} 
 \| h^{-1}_j(I_{V_{j}}-I_{V_{j-1}})w \| ^2 
 &= \sum_{K \in \T_0}  h^{-2}_{K} \sum^{+ \infty}_{j=1} 2^{2j} 
 \| (I_{V_{j}}-I_{V_{j-1}})w \|^2_K \\
 & \leq  C_{\epsilon}  \sum_{K \in \T_0}  h^{-2}_{K} \sum^{+ \infty}_{i=1} 2^{2i}   \| w_i \|^2_{\omega_K} \\ 
 & \leq  C_{\epsilon}   \sum^{+ \infty}_{i=1} 2^{2i} \sum_{K \in \T_0}  h^{-2}_{K} \| w_i \|^2_{\omega_K} \\
 & \leq  C_{\epsilon}  \sum^{+ \infty}_{i=1} 2^{2i} \sum_{K \in \T_0} \Covrlp(K) \max_{\bar{K} \subset \omega_K} h^{-2}_{\bar{K}} \| w_i \|^2_{K} \\
 & \leq \underbrace{C_{\epsilon}  \max_{K \in \T_0} \left[ \Covrlp(K) \frac{\max_{\bar{K} \subset \omega_K} h^{-2}_{\bar{K}}}{h^{-2}_K} \right]}_{C_{\zeta}} \sum^{+ \infty}_{i=1} 2^{2i} \sum_{K \in \T_0}   h^{-2}_K \| w_i \|^2_{K}.
\end{align*}
Finally, \Cref{thm:stables_splitting_quasi_to_S} gives
\begin{equation*}
\sum^{+ \infty}_{j=1}    \| h^{-1}_j (I_{V_{j}} - I_{V_{j-1}})w \| ^2 
\leq C_{\zeta} C_{S,I_{S}} \| \nabla w\|^2.
\end{equation*}
\end{proof}

Now proceed with bounding the norm of the splittings from below by a $H^1$-seminorm. We start with some auxiliary lemmas.

\begin{lemma}
\label{lemma:upper_bound_finite}
There exists a constant $c_{S}>0$ depending only on $d$ and $\gamma_0$ such that for any $N \in \mathbb{N}_0$ and any sequence  $(w_j)^{N}_{j=0} ,w_j\in S_j$, $j=0, \ldots,N$, it holds that 
\begin{align}
 \left \lVert \nabla \left(  \sum^{N}_{j=0} w_j \right)  \right \rVert ^2 \leq \frac{1}{c_{S}} \left( \| \nabla  w_0 \|^2  +  \sum^{N}_{j=1}   \| h^{-1}_{j} w_j \| ^2\right) .
\end{align}
\end{lemma}
\begin{proof}
The proof for $d=2$ is given in \cite[Lemma~3.4]{Bornemann1993}. It is based on the so-called Strengthened Cauchy--Schwarz inequality.
As the Strengthened Cauchy--Schwarz inequality is valid also for $d=3$ (see, e.g., \cite[Lemma~6.1]{Xu1992}), the proof of the theorem in the three-dimensional case is analogous to the two-dimensional one.
\end{proof}
\begin{lemma}\label{lemma:upper_bound_infinite_convergence}
Let $(w_j)^{+ \infty}_{j=0}$,  $w_j\in V_j$, $j\in \mathbb{N}_0$, be a sequence which satisfies
\begin{equation*}
\| \nabla  w_0 \|^2  +  \sum^{+ \infty}_{j=1}   \| h^{-1}_{j} w_j \| ^2 < + \infty.
\end{equation*}
Then $\sum^{+ \infty}_{j=0}w_j$ converges in $\left( H^1_0(\Omega), \| \nabla \cdot \|\right)$.
\end{lemma}
\begin{proof}
We will use Lemma \ref{lemma:upper_bound_finite} to show that $(\sum^{N}_{j=0}w_j)^{+ \infty}_{N=0}$ is a Cauchy sequence in $\left( H^1_0(\Omega), \| \nabla \cdot \|\right)$. Let $\epsilon>0$. Since  $\| \nabla  w_0 \|^2  +  \sum^{+ \infty}_{j=1}   \| h^{-1}_{j} w_j \|^2$ converges in $\R$, there exists $M\in\mathbb{N}$ such that for all $m>n>M$, it holds that 
\begin{equation*}
  \sum^{m}_{j=n} \| h^{-1}_{j} w_j \| ^2 < c_S {\epsilon^2}.
\end{equation*}
Using Lemma \ref{lemma:upper_bound_finite} for $w_j$, $j=n,\ldots,m$, and the previous inequality,  
\begin{equation*}
\left\lVert \nabla \left(  \sum^{m}_{j=n} w_j \right)  \right\rVert^2 \leq \frac{1}{c_S} \sum^{m}_{j=n}   \| h^{-1}_{j} w_j \| ^2 < \epsilon^2,
\end{equation*}
i.e., $(\sum^{N}_{j=0}w_j)^{+ \infty}_{N=0}$ is a Cauchy sequence in $\left( H^1_0(\Omega), \| \nabla \cdot \|\right)$ and thus $\sum^{+\infty}_{j=0}w_j$ converges in $\left( H^1_0(\Omega), \| \nabla \cdot \|\right)$.
\end{proof}
\begin{lemma}\label{lemma:upper_bound_infinite}
Let $c_{S}$ be the constant from  \Cref{lemma:upper_bound_finite}. Let $(w_j)^{+ \infty}_{j=0}$,  $w_j\in V_j$, $j\in \mathbb{N}_0$, be a sequence such that  $\sum^{+ \infty}_{j=0}w_j$ converges in $\left( H^1_0(\Omega), \| \nabla \cdot \|\right)$. Then
\begin{equation*}
\left\lVert \nabla \left(  \sum^{+\infty}_{j=0} w_j \right)  \right\rVert ^2 \leq \frac{1}{c_{S}} \left( \| \nabla  w_0 \|^2  +  \sum^{+ \infty}_{j=1}   \| h^{-1}_{j} w_j \| ^2\right).
\end{equation*}
\end{lemma}
\begin{proof}
For any $N \in \mathbb{N}_0$, \Cref{lemma:upper_bound_finite} gives
\begin{equation*}
    \left\lVert \nabla \left(  \sum^{N}_{j=0} w_j \right)  \right\rVert^2  \leq   \frac{1}{c_S} \left( \left\lVert \nabla  w_0 \right\rVert^2  +  \sum^{N}_{j=1}   \| h^{-1}_{j} w_j \| ^2\right).
\end{equation*}
Since  $\sum^{+ \infty}_{j=0}w_j$ converges in $\left( H^1_0(\Omega), \left\lVert \nabla \cdot \right\lVert\right)$, we may switch the following limit and norm, giving
\begin{align*}
 & \left\lVert \nabla \left(  \lim_{ N \to + \infty} \sum^{N}_{j=0} w_j \right)  \right\rVert^2  = \lim_{N \to + \infty} \left\lVert \nabla \left(  \sum^{N}_{j=0} w_j \right)  \right\rVert^2 \\ &\qquad \qquad  \leq \lim_{N\to +\infty}   \frac{1}{c_S} \left( \left\lVert \nabla  w_0 \right\rVert^2  +  \sum^{N}_{j=1}   \| h^{-1}_{j} w_j \| ^2\right)  =    \frac{1}{c_S} \left( \| \nabla  w_0 \| ^2  +  \sum^{+ \infty}_{j=1}   \| h^{-1}_{j} w_j \| ^2\right).
\end{align*}
\end{proof}

\begin{thm}\label{thm:decomposition-using-quasi-interpol-to-V}
Any function $w \in H^1_0(\Omega)$ can be uniquely decomposed as 
\begin{equation*}
    w = I_{V_{0}} w + \sum^{+ \infty }_{j=1} (I_{V_{j}} - I_{V_{j-1}}) w;
\end{equation*}
(the convergence of the sum is understood in the space $(H^1_0(\Omega), \|\nabla \cdot \|)$).
Let $c_S$ be the constant from Lemma \ref{lemma:upper_bound_finite} and $C_{S,I_V}$ the constant from \Cref{lemma:decomposition-using-quasi-interpol-to-V-upper-bound}. Then for all $w \in H^1_0(\Omega)$,
\begin{equation}\label{eq:stables_splitting_quasi_to_V}
c_S \| \nabla w\|^2\leq 
\left \lVert \nabla I_{V_{0}} w  \right \rVert ^2 + 
 \sum^{+ \infty}_{j=1}  \| h^{-1}_{j} ( I_{V_{j}} w - I_{V_{j-1}} w ) \| ^2 \leq C_{S,I_V} \| \nabla w\|^2.
\end{equation}
\end{thm}
\begin{proof}
The upper bound is proven in \Cref{lemma:decomposition-using-quasi-interpol-to-V-upper-bound}. Now we will prove the lower bound. Having the upper bound, we can use \Cref{lemma:upper_bound_infinite_convergence} to show that the sum $I_{V_{0}}w + \sum^{+ \infty }_{j=1} (I_{V_{j}} - I_{V_{j-1}}) w$ converges in $\left(H^1_0(\Omega), \| \nabla \cdot \| \right)$ and consequently, from \Cref{lemma:upper_bound_infinite} with $w_0 := I_{V_0}w$ and $w_j := (I_{V_{j}} - I_{V_{j-1}}) w$,
\begin{equation*}
c_S \left \lVert \nabla \left(  I_{V_{0}}w + \sum^{+ \infty }_{j=1} (I_{V_{j}} - I_{V_{j-1}}) w \right)  \right \rVert ^2 \leq  \| \nabla I_{V_0} w \|^2 + \sum^{+ \infty}_{j=1} \| h^{-1}_j (I_{V_j} w - I_{V_{j-1}} w ) \| ^2. 
\end{equation*}

It remains to show that $I_{V_{0}}w + \sum^{+ \infty }_{j=0} (I_{V_{j}} - I_{V_{j-1}}) w = w $ in $\left(H^1_0(\Omega), \| \nabla \cdot \| \right)$.
Since, for arbitrary $N \in \mathbb{N}$ (see \Cref{thm:quasi-interpolation-to_V-global-estimates}),
\begin{equation*}
\left \lVert w - \left( I_{V_{0}}w + \sum^{N}_{j=1} (I_{V_{j}} - I_{V_{j-1}}) w  \right) \right \rVert   = \| w - I_{V_{N}} w \| \leq C_{I_{V},2} \max_{K \in \T_N}h_{K} \| \nabla w\|,
\end{equation*}
and $\max_{K \in \T_N}h_{K} \to 0$,
$I_{V_{0}}w + \sum^{+ \infty }_{j=1} (I_{V_{j}} - I_{V_{j-1}}) w = w$ in $L^2(\Omega)$. We will show by contradiction that $I_{V_{0}}w + \sum^{+ \infty }_{j=1} (I_{V_{j}} - I_{V_{j-1}}) w = w$ also in $\left(H^1_0(\Omega), \| \nabla \cdot \| \right)$. Let the sequence $I_{V_{0}}w + \sum^{N}_{j=1} (I_{V_{j}} - I_{V_{j-1}}) w$ converge in $\left(H^1_0(\Omega), \| \nabla \cdot \| \right)$ to \mbox{$\bar{w} \neq w$}. Then, thanks to Friedrich's inequality (\Cref{lemma:Friedrichs}) the sequence converges to $\bar{w}$ in $L^2(\Omega)$, which is a contradiction with the uniqueness of the limit.
\end{proof}

\begin{thm}\label{thm:stable_splitting_to_V}
Let $c_S$ be the constant from 
\Cref{lemma:upper_bound_infinite}.
There exists a constant $C_{S}>0$ depending only on~$d$ and $\gamma_0$ such that for all $w \in H^1_0(\Omega)$,
\begin{equation}
\label{eq:stable_spliting_H1_coarsest}
c_{S} \| \nabla w \|^2 \leq 
\inf_{  w_j \in V_j ;\ w=\sum^{+\infty}_{j=0} w_j } 
\| \nabla  w_0 \|^2  +  \sum^{ + \infty}_{j=1}   \| h^{-1}_{j} w_j \| ^2 \leq C_{S} \| \nabla w \|^2.
\end{equation}
\end{thm}
\begin{proof}
From \Cref{thm:decomposition-using-quasi-interpol-to-V} we know that for any $w\in H^1_0(\Omega)$ there exists a decomposition $w=\sum^{+\infty}_{j=0}w_j$, $w_j \in V_j$, $j\in\mathbb{N}_0$, for which the upper bound holds with the constant $C_{S,I_V}$, so that we can take an infimum over all possible decompositions giving \mbox{$C_{S}\leq C_{S,I_V}$}.
The lower bound in \eqref{eq:stable_spliting_H1_coarsest} follows from \Cref{lemma:upper_bound_infinite}.
\end{proof}
\subsection{Splitting of $H^1_0(\Omega)$ into basis function spaces}
This section presents a result on splitting a $H^1_0(\Omega)$-function into basis function spaces.
Denote by $V_{j,i}$, $j=1,2\ldots$, $i=1,\ldots,\#\K_j$, the space spanned by the basis function~$\phi^{(j)}_i$, $V_{j,i} \subset V_j$.

First we will show that splitting a function $w_j \in V_j$ into the basis function spaces $V_{j,i}$, $i=1,\ldots,\#\K_j$ is stable.
This property is called stability of basis functions in the literature; see, e.g., \cite[Definition~2.5.5]{Ruede1993} and \cite[Assumption~(A1), p.~17]{Oswald1994a}. We present this property in a form which suits our further development.

\begin{lemma}[Stability of basis functions]\label{lemma:stability_of_basis_functions}
There exist positive constants $c_{B}$ and $C_{B}$ 
depending only on~$d$ and $\gamma_0$ such that for all \mbox{$j\in \mathbb{N}_0$} and all
\begin{equation*}
w_j = \sum^{\#\K_j}_{i=1} w_{j,i} \in V_j, \quad w_{j,i} \in V_{j,i},  \quad
i=1,\ldots, \#\K_j,
\end{equation*}
it holds that
\begin{align}
c_{B} \| h^{-1}_j w_j \|^2 & \leq \sum^{\#\K_j}_{i=1} \|\nabla w_{j,i} \|^2 \leq C_{B} \| h^{-1}_j w_j \| ^2.
\label{eq:stability_of_basis_functions1}
\end{align}
Let $\mMs_j$ be the so-called scaled mass matrix and $\matrx{D}_j$ the diagonal matrix defined as 
\begin{equation*}
    \left[ \mMs_j \right]_{m,n} = \int_{\Omega} h^{-2}_j \phi^{(j)}_n \phi^{(j)}_m, \qquad 
    \left[ \matrx{D}_j \right]_{m,m} = \int_{\Omega} \nabla \phi^{(j)}_m \cdot \nabla \phi^{(j)}_m , \qquad
    \forall m,n = 1, \ldots, \#\K_j.
\end{equation*}
Let $\vec{w}_j$ be the vector of coefficients of a function $w_j\in V_j$ in the basis $ \Phi_j$. Then $w_j = \sum_{i=1}^{\#\K_j} w_{j,i}$, $w_{j,i} = [\vec{w}_j]_i \phi^{(j)}_i$ and \eqref{eq:stability_of_basis_functions1} is equivalent to 
\begin{equation}\label{eq:stable_basis_lemma}
c_{B}  \vecT{w}_j \matrx{M}^S_{j} \vec{w}_j  \leq \vecT{w}_j \matrx{D}_j \vec{w}_j  \leq C_{B} \vecT{w}_j \matrx{M}^S_{j} \vec{w}_j.    
\end{equation}
That is, the matrices $\mMs_j$ and $\matrx{D}_j$ are spectrally equivalent with constants $c_B$ and $C_B$.
\end{lemma}
\begin{proof} The proof is inspired by \cite[Proposition~1.30, Problem~1.35]{Elman2005}; see also \cite{Fri72}.
We prove the spectral equivalence of local matrices associated with a mesh element. The assertion of the theorem for global matrices then follows by summing the local inequalities over the elements and taking the overlap into account.

Let $\mMsKj{j}$ be a local scaled mass matrix corresponding to an element $K \in \T_j$ defined as
\begin{equation}
\left[ \mMsKj{j} \right]_{m,n} = \int_{K}  h^{-2}_K \phi^{(j)}_n \phi^{(j)}_m, \qquad \forall m,n \in \N_K,
\end{equation}
and let $\matrx{M}^{\mathrm{S}}_{\hat{K}}$ be the local scaled mass matrix on a reference element $\hat{K}$, which does not depend on~$j$, $K$, or~$\T_j$.
Using standard arguments of affine transformation to a reference element, it holds that
\begin{equation}
\mMsKj{j} = 
\frac{\meas{K}}{h^2_K}
\matrx{M}^{\mathrm{S}}_{\hat{K}}.
\end{equation}
If we denote by $c_{\hat{K}}$ and 
$C_{\hat{K}}$ the smallest and the largest eigenvalues of~$\matrx{M}^{\mathrm{S}}_{\hat{K}}$, respectively, 
the eigenvalues of $\mMsKj{j}$ can be bounded by $c_{\hat{K}} {\meas{K}}/{h^2_K}$ and $C_{\hat{K}} {\meas{K}}/{h^2_K}$. Consequently,
\begin{equation}
\label{eq:locscmassmatrixeigbounds}
\frac{ c_{\hat{K}} \meas{K}}{h^2_K} \vecT{x} \vec{x} \leq \vecT{x} \mMsKj{j} \vec{x} \leq \frac{ C_{\hat{K}} \meas{K}}{h^2_K} \vecT{x} \vec{x}, \quad \forall \vec{x} \in \R^{(d+1)}.
\end{equation}
By choosing $\vec{x}$ as the $m$th column of the identity matrix of size $d+1$,
\begin{equation}
\label{eq:locscmassmatrixeigbounds_basisfuncnorm}
\frac{ c_{\hat{K}} \meas{K}}{h^2_K} \leq \left[ \mMsKj{j}\right] _{m,m} = \frac{\| \phi^{(j)}_m \|^2_K}{h^2_K}  \leq \frac{ C_{\hat{K}} \meas{K}}{h^2_K}.
\end{equation}

Let $\matrx{D}_{j,K}$ be the local variant of $\matrx{D}_j$, i.e.,
\begin{align}
    \left[ \matrx{D}_{j,K} \right]_{m,m} &= \int_{K} \nabla \phi^{(j)}_m \nabla \phi^{(j)}_m = \| \nabla  \phi^{(j)}_m \|^2_K.
\end{align}
Using the inverse inequality~(\Cref{lemma:local_inverse_ineq}) and \eqref{eq:locscmassmatrixeigbounds_basisfuncnorm},
\begin{equation}
    \| \nabla \phi^{(j)}_m \|^2_K \leq C^2_{\INV} h^{-2}_K \| \phi^{(j)}_m \|^2_K\leq C^2_{\INV} C_{\hat{K}} \frac{|K|}{h^{2}_K}.
\end{equation}
Similarly, using Friedrich's inequality~(\Cref{lemma:Friedrichs}),
\begin{equation}
    \frac{c_{\hat{K}} |K|}{C^2_{F} h^{2}_K}  \leq  \frac{1}{C^2_{F} h^{2}_K} \| \phi^{(j)}_m \|^2_K \leq 
    \| \nabla \phi^{(j)}_m \|^2_K. 
\end{equation}
Thus the matrix $\matrx{D}_{j,K}$ is spectrally equivalent to the identity matrix times $\frac{|K|}{h^{2}_K}$. From \eqref{eq:locscmassmatrixeigbounds}, we conclude that $\matrx{D}_{j,K}$ is also spectrally equivalent to $\mMsKj{j}$ with the equivalency constants involving $C_{\INV}$, $C_{F}$, $c_{\hat{K}}$, and $C_{\hat{K}}$, i.e.,~depending only on~$d$ and the shape regularity~$\gamma_j$.
\end{proof}
Since  $\mMs_j$ and $\matrx{D}_j$ are spectrally equivalent matrices and they are symmetric positive definite, we can use the generalized Hermitian eigenvalue decomposition (see, e.g., \cite[Eq.~(5.3)]{BookEigproblems00}) and algebraic manipulations to show that   $\left( \mMs_j \right)^{-1}$ and $\matrx{D}^{-1}_j$ are also spectrally equivalent, i.e., 
\begin{equation}\label{eq:stable_basis_matrix_equivalence_inverse}
    \frac{1}{C_B} \vecT{w} \left( \mMs_j \right)^{-1} \vec{w} \leq \vecT{w} \matrx{D}^{-1}_j \vec{w}  \leq \frac{1}{c_B} \vecT{w} \left( \mMs_j \right)^{-1} \vec{w},
    \qquad \forall \vec{w} \in \mathbb{R}^{\#\K_j}.
\end{equation}
Let $\matrx{M}_j$ denote the mass matrix associated with the $j$th level, i.e.,  $\left[ \matrx{M}_j \right]_{mn} = \int_{\Omega} \phi^{(j)}_n \phi^{(j)}_m$, $m,n=1,\ldots, \# \K_j$. 
Analogously to \eqref{eq:stable_basis_lemma} we can show the spectral equivalence of the mass matrix $\matrx{M}_j$ with the diagonal matrix $\matrx{D}_j$ in the following form. There exist positive constants $c_{M}, C_{M}$ depending only on~$d$ and $\gamma_0$ such that 
\begin{equation}
\label{eq:massspeceq}
c_{M}  \min_{K \in \T_j} h^{-2}_K \vecT{w} \mM_{j} \vec{w} \leq 
 \vecT{w} \matrx{D}_j  \vec{w} 
 \leq C_{M}
 \max_{K \in \T_j} h^{-2}_K \vecT{w} \mM_{j}  \vec{w} , \quad \forall  \vec{w}  \in \mathbb{R}^{\#\K_j}.
\end{equation}

Combining \Cref{thm:stable_splitting_to_V} and \Cref{lemma:stability_of_basis_functions} yields the following theorem on  splitting a $H^1_0(\Omega)$-function into basis function spaces. It can be proven by the same technique as in \cite[Theorem~2.3.1]{Ruede1993}.

\begin{thm}\label{thm:stable_splitting_into_basis_functions}
Let $c_S,C_{S}$ be the constants from \Cref{thm:stable_splitting_to_V}, $c_{B},C_{B}$ the constants from \Cref{lemma:stability_of_basis_functions}, and let $\overline{c}_{B} = \min \lbrace 1,c_{B} \rbrace$ and $\overline{C}_{B} = \max \lbrace 1,C_{B} \rbrace$. 
Then for all \mbox{$w \in H^1_0(\Omega)$,}
\begin{align}
c_{S} \overline{c}_{B} \| \nabla w \|^2 &\leq 
\inf_{  
\substack{ w_0 \in V_0, w_{j,i} \in V_{j,i}  \\ w = w_0 + \sum^{+\infty}_{j=1} \sum^{\#\K_j}_{i=1} w_{j,i} }}
\| \nabla  w_0 \|^2  +  \sum^{+\infty}_{j=1} \sum^{\#\K_j}_{i=1}  \|\nabla w_{j,i}\|^2  \leq C_{S}\overline{C}_{B} \| \nabla w \|^2.
\end{align}
\end{thm}

\subsection{Splitting of spaces of piecewise linear functions}
\label{subsec:splitting_finite_dim}
We now present consequences of the previous theorems for finite-dimensional piecewise linear functions from~$V_J$, $J\geq 0$. The following theorems can be proven by the same techniques as the results in \cite[Section~2.4]{Ruede1993}.

\begin{thm}\label{thm:stable_splitting_of_V_J}
Let $c_{S}$ and $C_{S}$ be the constants from Theorem \ref{thm:stable_splitting_to_V}. Let $J\geq0$. For all $w_J \in V_J$,
\begin{align}
&c_{S} \| \nabla w_J \|^2 \leq 
\inf_{w_j\in V_j; \ w_J = \sum^{J}_{j=0}w_j } 
\| \nabla  w_0 \|^2  +  \sum^{J}_{j=1}   \| h^{-1}_{j} w_j \| ^2 \leq C_{S} \| \nabla w_J \|^2.
\end{align}
\end{thm}

\begin{thm}\label{thm:stable_splitting_of_V_J_to_basis}
Let $c_{S}$ and $C_{S}$ be the constants from Theorem \ref{thm:stable_splitting_to_V} and $\overline{c}_{B},\overline{C}_{B}$ the constants from \Cref{thm:stable_splitting_into_basis_functions}. Let $J\geq0$. For all $w_J \in V_J$,
\begin{align}
c_{S} \overline{c}_{B} \| \nabla w_J \|^2 \leq 
\inf_{ \substack{ w_0 \in V_0, w_{j,i} \in V_{j,i}  \\ w_J = w_0 + \sum^{J}_{j=1} \sum^{\#\K_j}_{i=1} w_{j,i} } }
\| \nabla  w_0 \|^2  +  \sum^{J}_{j=1} \sum^{\#\K_j}_{i=1}  \|\nabla w_{j,i}\|^2
 \leq C_{S} \overline{C}_{B} \| \nabla w_J \|^2.
\end{align}
\end{thm}

\subsection{Frame}\label{subsec:frame}
Finally, we present a consequence of the stability of the splittings presented in \Cref{thm:stable_splitting_into_basis_functions}, which is closely related to the fact that the normalized basis functions form a so-called \emph{frame in}~$\left( H^1_0(\Omega)\right)^{\#}$; see, e.g., \cite[Section~3]{Harbrecht2016}, \cite{Harbrecht2008}. 

\begin{thm}\label{thm:frame_infinite}
Let $c_{S}$ and $C_{S}$ be the constants from Theorem \ref{thm:stable_splitting_to_V} and $\overline{c}_{B},\overline{C}_{B}$ the constants from \Cref{thm:stable_splitting_into_basis_functions}.
For all $g\in \left( H^1_0(\Omega)\right)^{\#}$,
\begin{align}\label{eq:frame_infinite}
c_{S} \overline{c}_{B} \! \left(
\| \nabla  g_0 \|^2  +  \sum^{+\infty}_{j=1} \sum^{\#\K_j}_{i=1}  \frac{  \langle g,\phi^{(j)}_i  \rangle ^2 }{ \| \nabla \phi^{(j)}_i  \| ^2} \right)
&\leq
\| g\|^2_{\left( H^1_0(\Omega)\right)^{\#}} \leq C_{S} \overline{C}_{B} \! \left(
\| \nabla  g_0 \|^2  +  \sum^{+\infty}_{j=1} \sum^{\#\K_j}_{i=1}  \frac{  \langle g,\phi^{(j)}_i   \rangle^2 }{ \| \nabla \phi^{(j)}_i  \| ^2} \right)\!,
\end{align}
where $g_0\in V_0$ is the Riesz representation of the functional~$g$ in the space~$V_0$ with respect to the inner product $(u_0, v_0)_0 = \int_{\Omega}\nabla v_0 \cdot \nabla u_0$, $\forall u_0, v_0 \in V_0$.
\end{thm}

\begin{proof}
The proof is inspired by the proof of \cite[Theorem 2.6.2]{Ruede1993}.

We will start with the upper bound. Let $w\in H^1_0(\Omega)$ and consider an arbitrary decomposition $w = w_0 + \sum^{+\infty}_{j=1} \sum^{\#\K_j}_{i=1} w_{j,i}$, $w_0 \in V_0, w_{j,i} \in V_{j,i}$. Using the fact that 
\[
w_{i,j} =  \frac{ \sign(w_{i,j}) \left \lVert \nabla w_{i,j} \right \rVert }{\| \nabla \phi^{(j)}_i \| } \ \phi^{(j)}_i, 
\]
we have
\begin{align*}
    |\langle g,w \rangle|  &\leq  |\langle g,w_0 \rangle| + \sum^{+\infty}_{j=1} \sum^{\#\K_j}_{i=1} |\langle g,w_{i,j} \rangle| \\
    & \leq  \| g_0 \| \cdot \| \nabla w_0 \| + \sum^{+\infty}_{j=1} \sum^{\#\K_j}_{i=1} \left| \left \langle g, \frac{ \phi^{(j)}_i}{\| \nabla \phi^{(j)}_i\|}\right \rangle \right| \cdot \|\nabla w_{i,j} \| \\
    & \leq  \left(
\| \nabla  g_0 \|^2  +  \sum^{+\infty}_{j=1} \sum^{\#\K_j}_{i=1}  \frac{  \langle g,\phi^{(j)}_i   \rangle^2 }{ \| \nabla \phi^{(j)}_i  \| ^2} \right) ^{\frac{1}{2}} \cdot   \left( \| \nabla  w_0 \|^2  +  \sum^{+\infty}_{j=1} \sum^{\#\K_j}_{i=1}  \|\nabla w_{j,i}\|^2  \right)^{\frac{1}{2}}.
\end{align*}
Taking the infimum over all decompositions $w = w_0 + \sum^{+\infty}_{j=1} \sum^{\#\K_j}_{i=1} w_{j,i}$, $w_0 \in V_0, w_{j,i} \in V_{j,i}$ and using the stability of the decomposition into spaces defined by basis functions (\Cref{thm:stable_splitting_into_basis_functions}) yields
\begin{equation*}
 |\langle g,w \rangle| \leq  \left(
\| \nabla  g_0 \|^2  +  \sum^{+\infty}_{j=1} \sum^{\#\K_j}_{i=1}  \frac{ \langle g,\phi^{(j)}_i  \rangle^2 }{ \| \nabla \phi^{(j)}_i  \| ^2} \right) ^{\frac{1}{2}} \cdot  C_{S}^{\frac{1}{2}}\overline{C}_{B}^{\frac{1}{2}} \| \nabla w \|.
\end{equation*}
Taking the supremum over all $w\in H^1_0(\Omega)$ such that $\| \nabla w\| = 1$ gives the upper bound.

Proving the lower bound is more subtle. We will first show that for any $N\in\mathbb{N}$,
\begin{equation}\label{eq:thm_frame_lower_finite}
\| \nabla  g_0 \|^2  +  \sum^{N}_{j=1} \sum^{\#\K_j}_{i=1}  \frac{  \langle g,\phi^{(j)}_i   \rangle^2 }{ \| \nabla \phi^{(j)}_i  \| ^2}  
\leq \frac{1}{c_S \overline{c}_{B}} \| g \|^2_{\left( H^1_0(\Omega)\right)^{\#}}. 
\end{equation}

First, it holds that
\begin{multline*}
\| \nabla  g_0 \|^2  +  \sum^{N}_{j=1} \sum^{\#\K_j}_{i=1}  \frac{ \langle g,\phi^{(j)}_i  \rangle^2 }{ \| \nabla \phi^{(j)}_i  \| ^2}  
 = \left \langle g , g_0 \right  \rangle + \sum^{N}_{j=1} \sum^{\#\K_j}_{i=1}   \left \langle g, \frac{\langle g,\phi^{(j)}_i  \rangle }{ \| \nabla \phi^{(j)}_i  \| ^2} \phi^{(j)}_i  \right \rangle 
 \\
= \left \langle g ,  g_0 + \sum^{N}_{j=1} \sum^{\#\K_j}_{i=1}    \frac{  \langle g,\phi^{(j)}_i  \rangle }{ \| \nabla \phi^{(j)}_i  \| ^2} \phi^{(j)}_i  \right \rangle. 
\end{multline*}
Let $g_{j,i} = \frac{  \langle g,\phi^{(j)}_i   \rangle }{ \| \nabla \phi^{(j)}_i  \| ^2} \phi^{(j)}_i \in V_{j,i}$. Then using \Cref{thm:stable_splitting_of_V_J_to_basis},
\begin{align*}
\left \langle g ,  g_0 + \sum^{N}_{j=1} \sum^{\#\K_j}_{i=1}g_{j,i} \right \rangle  
&= \| g \|_{\left( H^1_0(\Omega)\right)^{\#}} \left \lVert \nabla \left(  g_0 + \sum^{N}_{j=1} \sum^{\#\K_j}_{i=1}g_{j,i}   \right) \right \rVert \\
& \leq \| g \|_{\left( H^1_0(\Omega)\right)^{\#}} \frac{1}{c_S^{\frac{1}{2}} \overline{c}_{B}^\frac{1}{2}} \left( \|  \nabla  g_0 \|^2  +  \sum^{N}_{j=1} \sum^{\#\K_j}_{i=1}  \|\nabla g_{j,i}\|^2 \right )^\frac{1}{2} \\
& =  \| g \|_{\left( H^1_0(\Omega)\right)^{\#}} \frac{1}{c_S^{\frac{1}{2}} \overline{c}_{B}^\frac{1}{2}} \left( \|  \nabla  g_0 \|^2  +  \sum^{N}_{j=1} \sum^{\#\K_j}_{i=1}  \frac{  \langle g,\phi^{(j)}_i   \rangle^2 }{ \| \nabla \phi^{(j)}_i  \| ^2} \right )^\frac{1}{2} .
\end{align*}
This yields \eqref{eq:thm_frame_lower_finite}. Taking $N$ to infinity in \eqref{eq:thm_frame_lower_finite} finishes the proof. 
\end{proof}

\begin{thm}\label{thm:frames_finite}
Let $c_{S}$ and $C_{S}$ be the constants from Theorem \ref{thm:stable_splitting_to_V} and $\overline{c}_{B},\overline{C}_{B}$ the constants from \Cref{thm:stable_splitting_into_basis_functions}. Let $J\geq0$ and consider the space $V_J$ with the norm $\|\nabla \cdot \|$. For all $g_J \in V_J^{\#}$,
\begin{multline*}
c_{S} \overline{c}_{B} \left(
\| \nabla  g_0 \|^2  +  \sum^{J}_{j=1} \sum^{\#\K_j}_{i=1}  \frac{  \langle g,\phi^{(j)}_i   \rangle^2 }{ \| \nabla \phi^{(j)}_i  \| ^2 } \right)
\leq \| g_J\|^2_{V_J^{\#}}
\leq C_{S} \overline{C}_{B} \left(
\| \nabla  g_0 \|^2  +  \sum^{J}_{j=1} \sum^{\#\K_j}_{i=1}  \frac{\langle g,\phi^{(j)}_i  \rangle^2 }{\| \nabla \phi^{(j)}_i  \| ^2}\right),
\end{multline*}
where $g_0\in V_0$ is the Riesz representation function of the functional $g$ in the space $V_0$
with respect to the inner product $(u_0, v_0) = \int_{\Omega}\nabla v_0 \cdot \nabla u_0$, $\forall u_0, v_0 \in V_0$.
\end{thm}

\begin{proof}
The proof is analogous to the proof of \Cref{thm:frame_infinite}.
\end{proof}
\end{appendices}

\bibliography{sn-bibliography}

\end{document}